\documentclass[]{article}
 \usepackage{latexsym,amsmath,amssymb,amsthm,textcomp}
\usepackage{geometry}
\usepackage{color,calc}
\usepackage[]{graphicx}

 \newtheorem{theorem}{Theorem}[section]
\newtheorem{lemma}{Lemma}[section]
 \newtheorem{proposition}{Proposition}[section]
 \newtheorem{corollary}{Corollary}[section]
 \newtheorem{rem}{Remark}[section]

\begin{document}

%%\begin{article}
%\begin{opening}         
\title{{\LARGE   Triple scale analysis of periodic  solutions and
    resonance  of some asymmetric non linear
    vibrating  systems  
     } } 
 \author{
Nadia Ben Brahim\\
  University of Tunis, El Manar,  Ecole nationale
  d'ing\'enieurs de Tunis (ENIT), \\
Laboratory of  civil engineering (LGC),\\
BP 37, 1002 Tunis Belv\'ed\`ere,
Tunisia \\
 and\\
B. Rousselet,
  email: {\sf br\char64math.unice.fr} \\
%\runningauthor{N. Ben Brahim, B. Rousselet}
%\runningtitle{ Double scale expansion of vibrating  systems with non linear springs}
 University of  Nice Sophia-Antipolis, Laboratory J.A. Dieudonn\'e \\ 
U.M.R.  C.N.R.S. 6621, Parc Valrose, \\ F 06108 Nice,
 C{e}dex 2, France }
\date{}

\maketitle
\bibliographystyle{alpha}

\begin{abstract}
 We consider {\it small solutions} of a vibrating mechanical system with
smooth non-linearities for which we provide an approximate solution by
using a triple scale analysis; a rigorous proof of convergence of the
triple scale method is included; for the forced response, a stability
result is needed in order to prove convergence in a neighbourhood of a
primary resonance. The amplitude of the response with respect to the
frequency forcing is described and it is related to the frequency of a
free periodic vibration.

{\bf Keywords:} triple scale expansion; periodic solutions; nonlinear vibrations; normal modes
\end{abstract}

\section{Introduction}
\label{sec:intro}

In this article, we perform a triple scale analysis of small periodic
solutions of free vibrations of a discrete structure without damping
and with a local smooth non-linearity; then we consider a similar system with
damping and a periodic forcing in a resonance situation.

Several experimental studies  show that it is possible to detect defects in a structure by considering its vibro-acoustic response to an external actuation;  there is a  vast literature in applied physics.
We recall some papers related to the use of the frequency response for
non destructive testing; in particular generation of higher harmonics,
cross-modulation of a high frequency by a low frequency (often called
intermodulations in telecommunication):
\cite{ekimov-didenkulov-kasakov99},
\cite{moussatov-castagnede-gusev}; in \cite{vdb-lagier-groby}, "a
vibro-acoustic method, based on frequency modulation, is developed in
order to detect defects on aluminium and concrete beams''; experiments have been performed on a real bridge by G. Vanderborck
with four prestressed cables: two undamaged cables, a damaged one and
a safe one but damaged at the anchor. With routine experimental
checking of  the lowest natural frequency, the presence of the
damaged cable had only been found by comparison with  data collected
15 years ago; the one damaged at the anchor was not found; see details
in \cite{vdb-lagier04}, \cite{rousseletvdb05}.

However the  analysis {\it per se} of non linear vibration is also an important topic
from the  academic and  industrial viewpoint. In this work,  we are
interested in the  behaviour due to a local non linear stress-strain
law; first, we consider free vibration  and then  forced response of a
 damped system  with excitation frequency close to a frequency
of the free system ; so, this local stress-strain law is assumed to be:
\textbf{$N= k\tilde u +c\tilde u^{2}+d \tilde u^{3}$},
where N is the normal force and $\tilde u$ is the 
elongation. The elastodynamic problem of continuum mechanics leads
after discretization by finite elements to a system of non linear
differential equations of second order, thus, this paper deals with
such systems with several degrees of freedom. We determine an
asymptotic expansion of {\it small periodic solutions} of a discrete
structure; we use   the method of triple scale \cite{nayfeh81} and 
compare these results with a numerical integration program; also, we
perform a numerical Fourier transform to determine the frequencies
and compare with that of the  linear system. 

Our approach is only valid in the low frequency range and we have
bypassed the propagation of acoustic waves in the structure; this
point has been studied in
\cite{junca-lombard09},\cite{junca-lombard12}.
 The case of  rigid  contact which is also important from the point of
 view of theory and applications has been addressed in several papers,
 for example \cite{janin-lamarque01}, and a synthesis in \cite{Bastien-Bernard-lamarque}
; a numerical method to compute periodic solutions is proposed in \cite{laxalde-legrand11}
.% il un ? dans le pdf, fait
Asymptotic expansions have been used for
a long time; such methods are introduced in the famous memoir of Poincar\'e
\cite{poincare92-99}; a classic general book on asymptotic methods is
\cite{Bogolyu-Mitropo-ru} with french and English translations \cite{Bogolyu-Mitropo-fr,Bogolyu-Mitropo-eng}; introductory material is in \cite{nayfeh81},   \cite{miller2006}; 
 a detailed account of the
averaging method with precise proofs of convergence may be found in
\cite{sanders-verhulst}; an analysis of several methods including
multiple scale expansion may be found in \cite{murdock91};
the case of vibrations with
unilateral springs have been presented in \cite{sj-brgdr08,junca-br10,vestroni08,
hr-brgdr08,hazim-tamtam09,hazim-ecssmt}; this topic has been presented
by H. Hazim at  ``Congr\`es Smai'' in 2009; more details are to be
found in his thesis defended at University of Nice Sophia-Antipolis in
2010.
In a forthcoming paper, such a non-smooth case will be considered as well
as a numerical algorithm based on the fixed point method used in
\cite{rousselet:hal-perio-lip}. 
The case of vibrations with weak grazing unilateral contact has been
presented by S. Junca and Ly Tong at  4th Canadian Conference on
Nonlinear Solid Mechanics 2013; in
\cite{jiang-pierre-shaw04}
% il y a des ? dans le pdf pour les trois dernieres references, fait
 a numerical approach for large solutions of
piecewise linear systems is proposed. A review paper for so called
``non linear normal modes'' may be found in \cite{nnm-kpgv}; it includes numerous
papers published by the mechanical engineering community; several application
fields have been addressed by this community; for example in \cite{mikhlin10} ``nonlinear
vibro-absorption problem, the cylindrical shell nonlinear dynamics and
the vehicle suspension nonlinear dynamics are analysed''.
Preliminary versions of these results may be found in
\cite{benbrahim-tamtam09} and have been presented in conferences
  \cite{benbrahimGdrafpac,benbrahimSmai}; a proof of convergence of
  double scale expansion is to be found in the preliminary work
% il y a ? dans le pdf, fait
\cite{nbb-br-doubl}.
 
In the present text and in the conclusion, we compare the  use of
double or triple scale expansion. We emphasize that the use of three time scales, instead of two times
scales presented in the preliminary work \cite{nbb-br-doubl}, provides
a much improved insight in the behavior of the forced response close to resonance.
{\it In this paper}, as an introduction, in a first step, we consider {\it
  small solutions} of a system with one degree of freedom; we compare
free vibration frequency and the frequency of the periodic forcing for
which the amplitude is maximal. Then we address a system with several
degrees of freedom, we look for periodic free vibrations (so called
non linear normal modes in the mechanical engineering community); we
compare this frequency with the response to a periodic forcing close
to resonance.

\section{  One degree of freedom,  quadratic and cubic non linearity}

We consider a stress-strain law with a
strong cubic non linearity:
$$N=k \tilde u + \Phi(\tilde u,\epsilon) \text{ with  }  \Phi(\tilde u,\epsilon)=m c \tilde u^2
+\frac{m d}{\epsilon} \tilde u^3$$
where $\epsilon$ is a small parameter which is also involved in the size
of the solution; $m$ is the mass, $k$ the linear rigidity of the spring and $\tilde u$ the change of length of the spring; the choice of this scaling
provides frequencies which are amplitude dependent at first order.
\subsection{ Free vibration,  triple scale expansion up to second order}
Using second Newton law, free vibrations  of a mass attached to such a
spring are governed by:
\begin{equation}
  \label{eq:grosse-cub-lib}
  \ddot{ \tilde u} + \omega^2 \tilde u +c\tilde u^2 +\frac{d }{\epsilon}\tilde u^3=0.
\end{equation}

\begin{rem}
  \begin{itemize}
  \item
We intend to look for a small solution therefore, we consider a change
of function 
\fbox{$\tilde u=\epsilon u$} and obtain the transformed equation:
$$ \ddot u + \omega^2 u +\epsilon c u^2 + \epsilon d u^3=0.$$

In this form, this is a Duffing equation for which exists a vast literature, for example see the expository book \cite{kov-bren11}.
\item For the scaling we have chosen, 
%bbb 
when we use double scale analysis, we remarked in \cite{benbrahim-tamtam09} that 
%bbbb
the approximation that we obtain does not involve explicitly the coefficient $c$  of the quadratic term; this coefficient is only involved in the proof of the validity of the expansion. In particular the frequency shift only involves the coefficient $d $ of  the cubic term.
\item
However when we use three time scales, the coefficient of the
quadratic term is involved in the frequency shift.
%%bbb

\item On the other hand, if we would let $\epsilon \rightarrow +\infty$ in \eqref{eq:grosse-cub-lib} , we would get a singular perturbation problem; this is not considered here.
 \end{itemize}
\end{rem}
As we look for a {\it small} solution with a triple scale analysis for time; we set  
\begin{equation}
  T_0= \omega t, \quad T_1=\epsilon t,~~   T_2=\epsilon^{2} t,  
\text{ hence } D_0u=\frac{\partial u}{\partial T_0}, \quad
D_1u=\frac{\partial u}{\partial T_1} \; \text{ and } \; D_{2}u= \frac{\partial u}{\partial T_{2}}
\end{equation}
and we obtain
\begin{align*}
  \frac{d u}{dt}&=\omega D_{0}u + \epsilon D_{1}u + \epsilon^{2} D_{2}u\\ 
 \frac{d^{2} u}{dt^{2}}&= \omega^2 D_{0}^{2}u+ 2 \epsilon \omega
 D_{0}D_{1}u+2\epsilon^{2} \omega D_{0}D_{2}u+\epsilon^{2}D_{1}^{2}u+2\epsilon^{3}D_{1}D_{2}u+\epsilon^{4}D_{2}^{2}u.
\end{align*}
As we look for a small solution we consider  initial data
$\tilde u(0)=\epsilon a_{} + \epsilon^{2}v_{1}+ \mathcal  O(\epsilon^3)$ and
$\dot{ \tilde u}(0)=\mathcal  O(\epsilon^3) $; or $ u(0)= a_{} + \epsilon^{}v_{1}+ \mathcal  O(\epsilon^2)$ and
$\dot{  u}(0)=\mathcal  O(\epsilon^2) $;
we expand the solution with   the {\it ansatz}
\begin{equation}
  \label{eq:dev-doublech}
u(t)= u(T_0, T_{1},T_{2})= u^{(1)}(T_0, T_{1},T_{2})+\epsilon^{}u^{(2)}(T_0, T_{1},T_{2})+\epsilon ^{2} r(T_{0}, T_{1},T_{2});
\end{equation}

%fait
so we obtain:

\begin{multline*}
  \frac{du}{dt}= \frac{du^{(1)}}{dt}+ \epsilon^{}\frac{du^{(2)}}{dt}+ \epsilon^{2}\frac{dr}{dt}
  = \frac{du^{(1)}}{dt}+
  \epsilon^{}\frac{du^{(2)}}{dt}+\epsilon^{2}D_{0}r+
  \epsilon^{2}(\frac{dr}{dt}- \omega D_{0}r)\\
= [\omega D_0u^{(1)} +\epsilon D_1 u^{(1)} +
\epsilon^{}D_{2}u^{(1)}] + 
\epsilon[ \omega D_0u^{(2)} +\epsilon D_1 u^{(2)} +
\epsilon^{}D_{2}u^{(2)}]\\ 
+ \epsilon^{2}[ \omega D_0r +\epsilon D_1 r + \epsilon^{2}D_{2}r] 
\end{multline*}
and with the formula 
\begin{equation*}
  \mathcal{D}_3 r = \frac{1}{\epsilon}\left( \frac{d^2 r}{d t^2}
    -\omega^2 D_0^2 r \right)=
 2 \omega D_0D_1 r +\epsilon \left [ 2\omega D_0D_2 r+ D_1^2 r+ 2\epsilon D_1D_2 r,
 \right ]+\epsilon^{3} D_2^2 r,
\end{equation*} 
we get
%ici aussi
\begin{align}
\begin{split}
\label{eq:d2udt2u1u2}
 \frac{d^2u}{dt^2}&=\frac{d^2u^{(1)}}{dt^2}+\epsilon\frac{d^2u^{(2)}}{dt^2}+ \epsilon^{2}\frac{d^2r}{dt^2}=
\frac{d^2u^{(1)}}{dt^2}+\epsilon^{}\frac{d^2u^{(2)}}{dt^2}+\epsilon^2 D_0^2r +\epsilon^3 {\cal D}_3 r\\
&= \omega^2 D_{0}^{2}u^{(1)}+ \epsilon^{}\left[ 2 \omega
  D_{0}D_{1}u^{(1)} + \omega^2 D_{0}^{2}u^{(2)}\right] \\ 
&\quad \quad \quad +\epsilon^{2}\left[ 2\omega D_{0}D_{2}u^{(1)} +
  D_{1}^{2}u^{(1)}+2 \omega D_{0}D_{1}u^{(2)}+ D_{0}^{2}r\right]\\
&\quad \quad \quad  \quad \quad
+\epsilon^{3}\left[2D_{1}D_{2}u^{(1)}+2 \omega D_{0}D_{2}u^{(2)}+D_{1}^{2}u^{(2)}   +  \mathcal{D}_3 r  \right] \\
& \quad \quad \quad  \quad \quad \quad  \quad + \epsilon^{4}\left[  D_{2}^{2}u^{(1)}+ 2D_{1}D_{2}u^{(2)}+ \epsilon D_{2}^{2}u^{(2)}\right] .
\end{split}
\end{align}

We plug expansions \eqref{eq:dev-doublech},\eqref{eq:d2udt2u1u2} into \eqref{eq:grosse-cub-lib};  by identifying the powers of
 $\epsilon$ in the expansion of  equation \eqref{eq:grosse-cub-lib}, we obtain:

\begin{align}
\Bigg \{
  \begin{array}[h]{rl}
 &D_{0}^{2}u^{(1)}+  u^{(1)} =0  \\
& \omega^2 \left[ D_{0}^{2}u^{(2)}+ u^{(2)} \right]= S_{2} \\
& \omega^2 \left[D_{0}^{2}r +  r \right]=S_{3} \label{eq:D03r=S3}\\
  \end{array} 
\end{align}
with
\begin{align*}
&S_{2}=-c {u^{(1)}}^{2} - d u^{(1)3} -2 \omega D_{0}D_{1}u^{(1)} ~~\text{and}\\
&S_3 = -2 {c} u^{(1)}u^{(2)}-3d u^{(1)2}u^{(2)} -2 \omega D_{0}D_{2}u^{(1)}
-D_{1}^{2}u^{(1)} -2 \omega D_{0}D_{1}u^{(2)} - \epsilon R(\epsilon,r,u^{(1)},u^{(2)}),
\end{align*}
with
\begin{multline*}
 R(\epsilon,r, u^{(1)},u^{(2)}) =  
2D_{1}D_{2}u^{(1)}+2\omega D_{0}D_{2}u^{(2)}+D_{1}^{2}u^{(2)} \\  +   c
u^{(2)2} + 2 c r u^{(1)} + 3 d u^{(1)}u^{(2)2}+ 3 d u^{(1)2}r +\mathcal{D}_3 r\\ +
\epsilon  \left (D_{2}^{2}u^{(1)}+ 2D_{1}D_{2}u^{(2)}+ \epsilon
  D_{2}^{2}u^{(2)} \right) + \epsilon \rho (u^{(1)}, u^{(2)},r ,\epsilon)
\end{multline*}
and with $\rho$, a polynomial in $r$:
%bbb
\begin{multline*}
 \rho(u^{(1)},u^{(2)},r,\epsilon)= 
 2 cru^{(2)}+ du^{(2)3}+ +6 d u^{(1)}u^{(2)}r \\
 + \epsilon( cr^{2}+ 3du^{(2)2}r+3du^{(1)}r^{2}) +\epsilon^{2}[3du^{(2)}r^2+\epsilon d r^{3}].
\end{multline*}
%bbbb
For convenience, we perform the change of variable $\theta(T_0,T_1,T_2)=T_0+\beta_{}(T_1,T_{2})$; we notice that $ D_{0}\theta=1; D_{1}\theta= D_{1}\beta_{}$ and $D_{2}\theta=D_{2}\beta_{}$;
we solve  the first equation of \eqref{eq:D03r=S3} with $D_0u^{(1)}(0)=0$, we get:
\begin{equation}
  \label{eq:u1=}
  u^{(1)}=a_{}(T_1,T_{2}) ~\cos(\theta ).
\end{equation}
 \begin{rem}
  We notice that $a_{}$ and $\beta_{}$ are not constants but functions
  of time scales $T_{1} ~and~ T_{2}$ because u depends on these times
  scales. The dependence of these functions with respect to $T_{1}
  ~and~ T_{2}$ will be determined by solving the equations of the
  following orders and eliminating the so-called secular terms.
\end{rem}
 First, we  determine the dependence on $T_{1}$;
with simple manipulation of the second equation of  \eqref{eq:D03r=S3},  we obtain
\begin{equation*}
  S_2= -\frac{c a_{}^2}{2} (\cos(2\theta) +1) - \frac{d a_{}^{3}}{4} \cos(3\theta)+ \cos(\theta)\left( \frac{-3da_{}^{3}}{4}+ 2\omega_{} a_{} D_1 \beta_{}\right) + 2\omega_{}D_1 a_{} \sin(\theta) 
\end{equation*}
we gather terms at angular frequency $\omega_{}$: 
\begin{equation*}
  S_2=-\frac{3 d a_{}^3}{4} \\cos(\theta)  +2\omega_{} \left [D_1 a_{}
    \sin(\theta)+ a_{}D_1 \beta_{} \\cos(\theta) \right]+S_2^{\sharp}
  \quad \text{ where}
\end{equation*}
\begin{equation*}
  S_2^{\sharp}= \frac{-c a_{}^2}{2}(1+ \cos(2 \theta))-\frac{d a_{}^3}{4} \cos(3\theta) 
\end{equation*}
It appears some terms at the frequency of the system, these terms provide a solution
$u^{(2)}$ of the equation \eqref{eq:nddllibre-alpha} which is non periodic
and non bounded over long time intervals.
We will eliminate these so-called secular terms by imposing:
\begin{equation}
  \label{eq:Da,Dbeta-libre}
  D_1a_{}=0 \text{ and } \quad  D_1 \beta_{}=\frac{3d a_{}^2}{8\omega_{} }
\end{equation}
 the solution of the second equation of  \eqref{eq:D03r=S3}, is:
\begin{equation}
 u^{(2)}=  \frac{-ca_{}^{2}}{2\omega_{}^{2}}~ + ~\frac{ca_{}^{2}}{6 \omega_{}^{2}}\cos (2\theta)+ \frac{da_{}^{3}}{32 \omega_{}^{2}}\cos(3\theta).
\end{equation} 
\begin{rem}
 We have omitted the term at frequency $\omega_{}$ which is
redundant with $u^{(1)}$; however this choice is connected to the
value of the initial condition; see Remark \ref{rem:u0=ea+e2}.  
\end{rem}

For the third equation of  \eqref{eq:D03r=S3}, the unknown is $r$;
this equation includes non linearities; we do not solve it but we show
that the solution is bounded on an interval dependent on $\epsilon$. 
We use the values of $u^{(1)}, u^{(2)}$ in $S_3$.
Intermediate computations:
\begin{equation*}
  \label{eq:u1u2=}
  u^{(1)} u^{(2)}=  \frac{-5ca_{}^{3}}{12\omega_{}^{2}}\cos(\theta) + \frac{ca_{}^{3}}{12\omega_{}^{2}}\cos(3\theta)~ + \frac{da_{}^{4}}{64 \omega_{}^{2}}(\cos(2\theta)+\cos(4\theta)).
\end{equation*}
\begin{equation*}
  \label{eq:u1cu2=}
   (u^{(1)})^2 u^{(2)}=\frac{-5ca^4}{24 \omega^2}+\frac{d
     a^5}{128\omega^2}\cos(\theta)-\frac{ca^4}{6\omega^2}\cos(2\theta)+
\frac{da^5}{64\omega^2}\cos(3\theta)+
\frac{ca^4}{24\omega^2}\cos(4\theta) +\frac{d a^5}{128 \omega^2}
\cos(5\theta)
\end{equation*}
The
right hand side,  after some manipulations is:

\begin{multline*}
S_3=   \sin(\theta) \left (2\omega D_{2}a_{}+ 2D_{1}a_{} D_{1}\beta_{}+
   a_{}D_{1}^{2}\beta_{} \right)\\ 
+\cos(\theta) \left(2 \omega a_{} D_{2}\beta_{} 
  -D_{1}^{2}a_{} + a_{} (D_{1}\beta_{})^{2}+ \frac{5c^{2}a_{}^{3}}{6\omega_{}^{2}} -\frac{3d^{2}a_{}^{5}}{128\omega_{}^{2}} \right) \\
+ S^{\sharp}_3 - \epsilon R(r, \epsilon, u^{(1)}, u^{(2)}) 
\end{multline*}
with
\begin{multline*}
S^{\sharp}_3=\frac{5 d c a_{}^{4}}{8 \omega_{}^{2}}+
 \sin(2\theta) \left(  \frac{4ca_{}}{3 \omega_{}} D_{1}a_{}  \right)+
\cos(2\theta) \left(\frac{4ca_{}^{2}}{3 \omega_{}}
  D_{1}\beta_{} + \frac{15 cd  a_{}^{4}}{32\omega_{}^{2}} 
 \right) 
 \\ 
+ \sin(3\theta) \left(  \frac{9da_{}^{2}}{16
    \omega_{}} D_{1}a_{} \right)
+ \cos(3\theta) \left ( \frac{- c^{2} a_{}^{3}}{6\omega_{}^{2}}-
  \frac{3d^{2}a_{}^{5}}{64 \omega_{}^{2}} + \frac{9da_{}^{3}}{16
    \omega_{}} D_{1}\beta_{} \right) \\ + \cos(4\theta)( \frac{-5c d  a_{}^{4}}{32\omega_{}^{2}}) + \cos(5\theta)( \frac{-3d^{2}a_{}^{5}}{128\omega_{}^{2}}).
\end{multline*}

By imposing
\begin{align*}
\label{eq:D2...beta} 
& 2\omega_{} D_{2}a_{}+ 2D_{1}a_{} D_{1}\beta_{}+ a_{}D_{1}^{2}\beta_{}=0\\
&  2 \omega_{} a_{} D_{2}\beta_{}-D_{1}^{2}a_{} + a_{} (D_{1}\beta_{})^{2}+ \frac{5c^{2}a_{}^{3}}{6\omega_{}^{2}}-\frac{3d^{2}a_{}^{5}}{128\omega_{}^{2}}=0
\end{align*}
we get that 
$S_3=S_3^{\sharp}-\epsilon R(\epsilon,u^{(1)},u^{(2)},r)$ no longer contains   any term  at  frequency $\omega_{}$.

As $D_1a=0$ and $ D_1 \beta_{}=\frac{3d a_{}^2}{8\omega_{}}$,
we obtain
 $$ 2 \omega_{} a_{} D_{2}\beta_{} + a_{} \left
  (\frac{9d^2a^4}{64 \omega^2} \right)+ \frac{5c^{2}a_{}^{3}}{6\omega_{}^{2}}-\frac{3d^{2}a_{}^{5}}{128\omega_{}^{2}}=0.$$
So,
 \begin{equation}
\label{eq:a2...b2}	
 D_{2}a_{}(T_{2})=0 ~~~~ and ~~~~ D_{2}\beta_{}(T_{2})=(-\frac{5c^{2}a_{}^{2}}{12\omega_{}^{3}} - \frac{15 d^{2} a_{}^{4}}{256\omega_{}^{3}}).
\end{equation}
As $a$ and $\beta$ do not depend on $T_0$, we note that:
\begin{equation}
\label{eq:a=alpha...D1D2}
\left\{
  \begin{aligned}
 &\frac{da}{dt}=\epsilon D_{1}a_{}+\epsilon^{2}D_{2}a_{}+\mathcal O (\epsilon^{3})\\
 &\frac{d\beta}{dt}=\epsilon D_{1}\beta_{} +\epsilon^{2}D_{2}\beta_{}
 +\mathcal O (\epsilon^{3}),
\end{aligned}
\right.
\end{equation} 
thus taking into account \eqref{eq:Da,Dbeta-libre} and to \eqref{eq:a2...b2}, we obtain:
\begin{equation}
\label{eq:a1b1...b2}
 \frac{da_{}}{dt}=0 ~~~~ \text{ and }~~~~ \frac{d\beta_{}}{dt}=\epsilon \frac{3 d a_{}^{2}}{8\omega_{}} +\epsilon^{2}(\frac{-5c^{2}a_{}^{2}}{12\omega_{}^{3}} - \frac{15d^{2} a_{}^{4}}{256\omega_{}^{3}})
\end{equation}
therefore, the solution of these equations is:
\begin{equation}
\label{eq:da1...db1}
 a_{}= cte~~~~ \text{and} ~~~~ \beta_{}(t)=\left[ \epsilon \frac{3 d a_{}^{2}}{8\omega_{}} +\epsilon^{2}(-\frac{5c^{2}a_{}^{2}}{12\omega_{}^{3}} - \frac{15d^{2} a_{}^{4}}{256\omega_{}^{3}})\right] t.
\end{equation}
The constant of integration is chosen to be zero as the initial
velocity  satisfies $ \dot{u}(0)=\mathcal 0(\epsilon^3)$.

In order to show that,  $r$ is bounded, after eliminating  terms at
angular frequency $\omega_{}$, we go back to the $t$ variable in the third
 equations of \eqref{eq:D03r=S3}.
\begin{equation}
\label{eq:ddotr=free1ddl}
 \frac{d ^{2} r}{d t^{2}} + \omega_{}^2 r  = \tilde S^{}_{3}\\ 
\end{equation}
with
$\tilde S^{}_{3}= S^{\sharp}_3(t, \epsilon) - \epsilon \tilde R^{}(r, \epsilon, u^{(1)}, u^{(2)})$
where
\begin{multline*}
 S^{\sharp}_3(t, \epsilon)=\frac{5 d c a_{}^{4}}{8
   \omega_{}^{2}}
+\cos(2(\omega_{}t+ \beta_{}(t))) \left( \frac{15 c d a^4}{32\omega_{}^{2}} +
   \frac{c d a_{}^{4}}{2 \omega_{}^{2}} \right)
+ \sin(2(\omega_{}t+
 \beta_{}(t))) ( \frac{c d a_{}^{4}}{2 \omega_{}^{2}} ) \\+
 \cos(3(\omega_{}t+ \beta_{}(t))) \left( \frac{- c^{2}
     a_{}^{3}}{6\omega_{}^{2}}- \frac{3d^{2}a_{}^{5}}{64
     \omega_{}^{2}} + \frac{27d^{2}a_{}^{5}}{128 \omega_{}^{2}}
 \right)  + \sin(3(\omega_{}t+ \beta_{}(t)))(
 \frac{9d^{2}a_{}^{5}}{128 \omega_{}^{2}}) \\
+ \cos(4(\omega_{}t+ \beta_{}(t)))( \frac{-3 c d  a_{}^{4}}{32\omega_{}^{2}})
+ \cos(5(\omega_{}t+ \beta_{}(t)))( \frac{-3d^{2}a_{}^{5}}{128\omega_{}^{2}})
\end{multline*}
\begin{equation*}
 \text{ and } \; \tilde R^{}= R(\epsilon,r,u^{(1)},u^{(2)})-\mathcal{D}_{3}r.
\end{equation*}
 in which the remainder $ \tilde R^{}$, the functions $u^{(1)},u^{(2)}$
 and their partial derivatives
with respect to  $T_1, T_2$ are expressed  with  the variable $t$.

\begin{proposition}
 There exists $\gamma>0$ such that for all $t \le
 t_{\epsilon}=\frac{\gamma}{\epsilon^{2}}$, the solution $\tilde
 u=\epsilon u$ of  \eqref{eq:grosse-cub-lib} has the following expansion, 
\begin{align}
\left\{
\begin{array}{rl}
 & \tilde u(t)=\epsilon a_{}~ \cos( \nu_{\epsilon} t) ~+ ~\epsilon^{2}\left( \frac{-ca_{}^{2}}{2\omega_{}^{2}}~ + ~\frac{ca_{}^{2}}{6 \omega_{}^{2}}\cos(2 \nu_{\epsilon} t) + \frac{da_{}^{3}}{32 \omega_{}^{2}}\cos(3 \nu_{\epsilon} t)\right) + \epsilon^3 r(\epsilon, t) \\
& \tilde u(0)=\epsilon a_{} + \epsilon^{2}(\frac{-ca_{}^{2}}{3\omega_{}^{2}} + \frac{da_{}^{3}}{32 \omega_{}^{2}})+O(\epsilon^{3}), \dot {u}(0)=O(\epsilon^2)\\
\end{array}
\right.
\end{align}
with
\begin{equation}
\label{eq:nualpha}
 \nu_{\epsilon}= \omega_{}+ \epsilon \frac{3 d a_{}^{2}}{8\omega_{}} + 
\epsilon^{2} \left (-\frac{5c^{2}a_{}^{2}}{12\omega_{}^{3}} - \frac{15
    d^{2} a_{}^{4}}{256\omega_{}^{3}} \right) +\mathcal O(\epsilon^3)
\end{equation}
 and   $r$ is uniformly bounded in $C^{2}(0,t_{\epsilon})$.
\end{proposition}

\begin{proof}
Let us use lemma \ref{eq:lemmew } with equation \eqref{eq:ddotr=free1ddl}; set $S=S_3^{\sharp}$;  as we
have enforced  \eqref{eq:a1b1...b2}, it is a periodic bounded
function  orthogonal to $e^{\pm it}$, it satisfies  lemma hypothesis;
similarly set  $g=\tilde R$; it is a polynomial in variable $r$ with
coefficients which are bounded functions, so it is a lipschitzian
function on bounded subsets and satisfies lemma  hypothesis.
\end{proof}
% bbb important
\begin{rem}
  We notice that if we increase $c$, there is a change of convexity
  of the mapping $a \mapsto \nu_{\epsilon}$; this is an effect which
  cannot be noticed by just obtaining  a first order approximation of the
  frequency with a double scale approximation of the solution as in
  \cite{nbb-br-doubl}. See numerical results at the end of  subsection \ref{subsec:forced1ddl}.
\end{rem}
\begin{rem}
\label{rem:u0=ea+e2}
  We can notice that we can also derive the solution which satisfies
  $u(0)=\epsilon a$ by adding to the solution $ -
  \epsilon^{2}(\frac{-ca_{}^{2}}{3\omega_{}^{2}} +
  \frac{da_{}^{3}}{32 \omega_{}^{2}})\cos(\nu_{\epsilon} t)$
\end{rem}
\subsection{ Numerical Results}
In the figure \ref{fig:fft1ddlfree}, we find plots of the Fourier
transform of solutions;
on the left, the linear case, we notice one frequency and on the
right, three frequencies  in the non linear case.
the Fourier transform displays the frequencies,
$\nu_{1}=0.164$;  $2\nu_{1}=0.329$;  $3 \nu_{1}= 0.493$
% les figures ne sont pas de meme tailles est ce expré? si non il vaut mieux qu'elles en soient
\begin{figure}[h] 
\includegraphics[height=7cm,width=6cm]{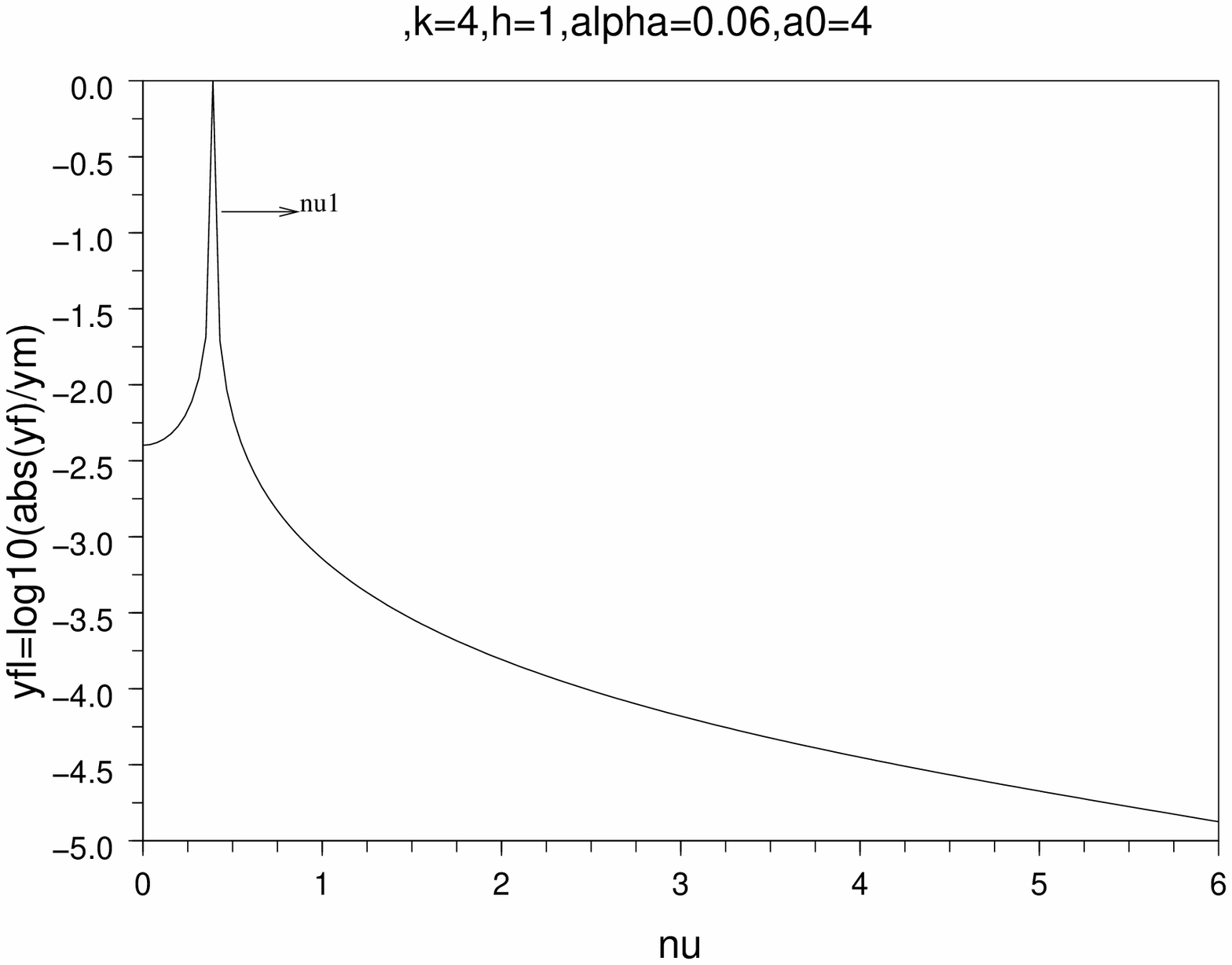}
\includegraphics[height=7cm,width=8cm]{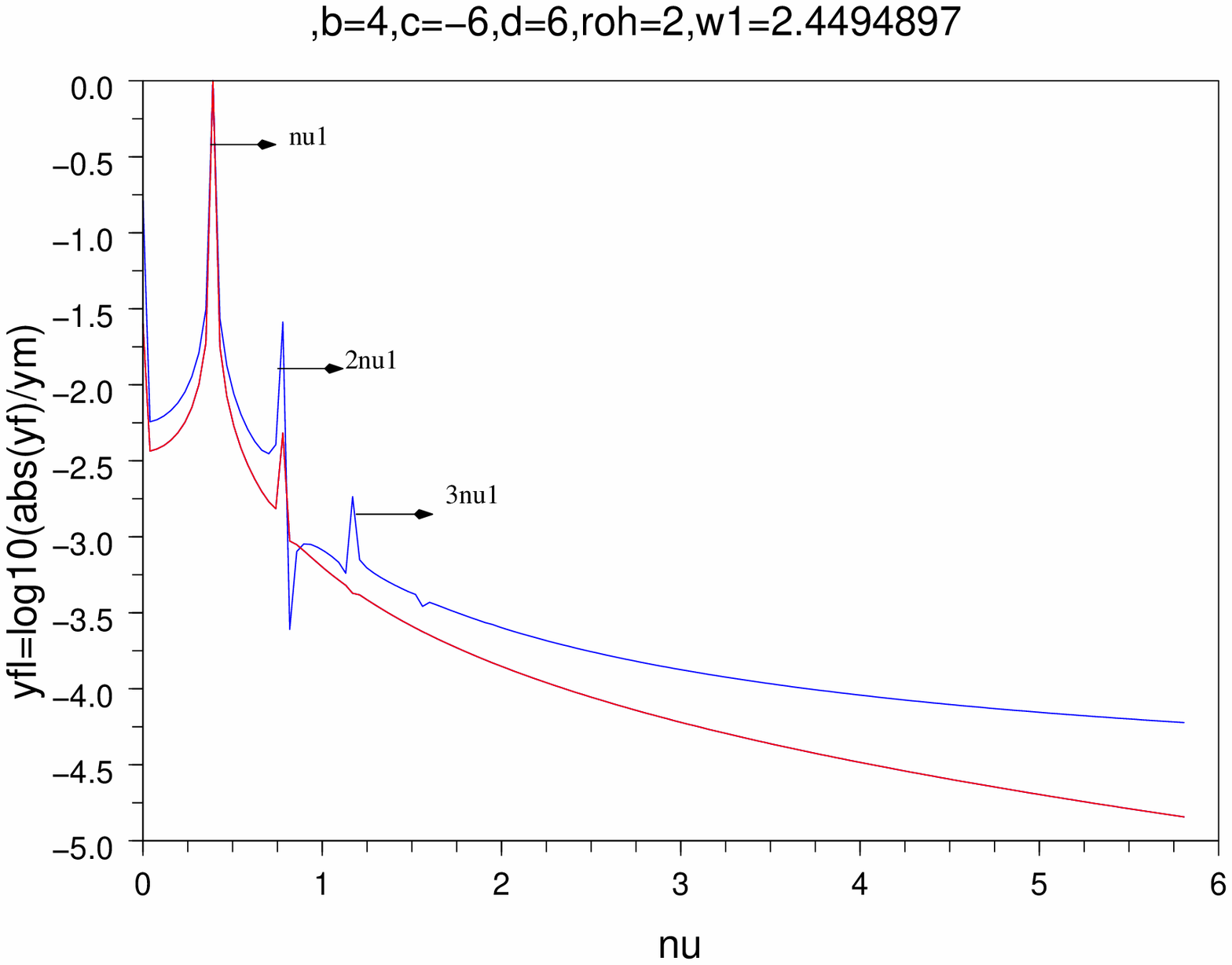}
\caption{ Dynamic frequency shift(fft) linear(left) and a non linear
  element with two methods(numerical(blue), asymptotic expansion( red))}
\label{fig:fft1ddlfree}
\end{figure}

%bbb

We notice good correlation between analytical results of asymptotic expansion and an integration step by step (with Scilab program ODE and numerical fast Fourier transform).

\subsection{Forced vibration, triple scale expansion up to second order}
\label{subsec:forced1ddl}
\subsubsection{Derivation of  the expansion}
Here we consider a similar system with a sinusoidal  forcing at a frequency close
to the free frequency; in the linear case without damping, it is well
known that the solution is no longer bounded when the forcing
frequency goes to the free frequency. Here, we consider the
mechanical system of previous section but with periodic  forcing and we
include some  damping term; the scaling of the forcing term
is chosen so that the expansion works properly; this is a known point, for
example see \cite{nayfeh86}.
%nnn j'ai mis \tildeomega _epsilon pour suivre le reste; bbb oui
\begin{equation}
  \label{eq:grosse-cub-forc}
  \ddot{ \tilde u} + \omega_{}^2 \tilde u +\epsilon \lambda \dot \tilde u +c\tilde u^2 +\frac{d }{\epsilon}\tilde u^3=\epsilon^2 F_m \cos(\tilde{\omega_{\epsilon}} t),
\end{equation}
where $F_m=\frac{F}{m}$ with the mass  $m$;
we assume positive damping, $\lambda >0$ and  excitation  frequency $\omega$ is {\it
  close} to an eigenfrequency of the  linear system in the following way:
\begin{equation}
\label{eq:tiledeomeag=}
\tilde{\omega_{\epsilon}}= \omega_{} +\epsilon \sigma.
\end{equation}
\begin{rem}
  \begin{itemize}
  \item 
 We look for a small solution with a triple scale expansion; as for the free vibrations, we consider a change of function \fbox{$\tilde u =\epsilon u$} and obtain the transformed equation
% il vaut mieux sans numero cette equation
\begin{equation*}
%  \label{eq:grosse-cub-forc}
  \ddot u + \omega_{}^2 u +\epsilon \lambda \dot u +\epsilon cu^2 +
  \epsilon d u^3=\epsilon F_m \cos(\tilde{\omega_{\epsilon}} t).
\end{equation*}
\item 
To simplify the computations, the fast scale $T_0$ is chosen to be
$\epsilon$ dependent.
  \end{itemize}
\end{rem}
 We set:
\begin{equation*}
  \label{eq:T0T1D0D1}
  T_0= \tilde \omega_{\epsilon} t, \quad T_1=\epsilon t \text{ and } T_{2}=\epsilon^{2}t, \text{ therefore } \;
   D_0u=\frac{\partial u}{\partial T_0}, \quad D_1u=\frac{\partial u}{\partial T_1}~ \text{ and }~ D_{2}u= \frac{\partial u}{\partial T_{2}},
\end{equation*}
so
\begin{align}
 \begin{split}
  \label{eq:dotu...ramort}
 &\frac{d u}{dt}=\tilde \omega_{\epsilon} D_{0}u + \epsilon D_{1}u +
 \epsilon^{2} D_{2}u \quad 
\text{and} \\
 &\frac{d^{2} u}{dt^{2}}= \tilde \omega_{\epsilon}^2D_{0}^{2}u+ 2
 \epsilon \tilde \omega_{\epsilon} D_{0}D_{1}u+2\epsilon^{2}D_{0}D_{2}u+\epsilon^{2}D_{1}^{2}u+2\epsilon^{3}D_{1}D_{2}u+\epsilon^{4}D_{2}^{2}u.
 \end{split}
\end{align}

With \eqref{eq:tiledeomeag=}, \eqref{eq:dotu...ramort} and the
following {\it  ansatz}, we look for a small solution:
%nnnn la meme remarque que dans le cas vibration libre
\begin{equation}
  \label{eq:dev-triplech-forc}
 u(t)= u(T_0, T_{1},T_{2})= u^{(1)}(T_0, T_{1},T_{2})+\epsilon^{}u^{(2)}(T_0, T_{1},T_{2})+\epsilon ^{2} r( T_{0},T_{1},T_{2})
\end{equation}

we obtain:
%nnnn ici aussi
\begin{multline*}
  \frac{du}{dt}= \frac{du^{(1)}}{dt}+ \epsilon^{}\frac{du^{(2)}}{dt}+ \epsilon^{2}\frac{dr}{dt}
  = \frac{du^{(1)}}{dt}+ \epsilon^{}\frac{du^{(2)}}{dt}+\epsilon^{2}D_{0}r+ \epsilon^{2}(\frac{dr}{dt}- D_{0}r)\\
  = [(\omega+\epsilon \sigma) D_0u^{(1)} +\epsilon D_1 u^{(1)}
  + \epsilon^{}D_{2}u^{(1)}] + \epsilon [ (\omega+\epsilon
  \sigma)D_0u^{(2)} +\epsilon D_1 u^{(2)} +
  \epsilon^{2}D_{2}u^{(2)}]\\ + \epsilon^{2} \omega D_0r
  +\epsilon^{2}(\frac{dr}{dt}- \omega D_{0}r)
\end{multline*}
where we  remark that $\frac{dr}{dt}- \omega D_{0}r= \epsilon D_{1}r+ \epsilon^{2}D_{2}r$ is of degree 1 in $\epsilon$. For the  second derivative, as for the case without forcing, we introduce
\begin{align*}
  \mathcal{D}_3 r&= \frac{1}{\epsilon} (\frac{d^{2}r}{dt^{2}}- 
  \omega^2 D_{0}^{2}r)\\
&=2\tilde \omega D_0D_1 r+ \epsilon\left[ 2\tilde\omega D_{0}D_{2}r+ D_1^2r +2 \epsilon D_2D_1r \right]+ \epsilon^{3}D_{2}^{2}r
\end{align*}
and we get
%nnnn et ici 
\begin{align*}
  \frac{d^2u}{dt^2}&=\frac{d^2u^{(1)}}{dt^2}+\epsilon\frac{d^2u^{(2)}}{dt^2}+ \epsilon^{2}\frac{d^2r}{dt^2}=
\frac{d^2u^{(1)}}{dt^2}+\epsilon^{}\frac{d^2u^{(2)}}{dt^2}+\epsilon^2 \tilde\omega^2D_0^2r +\epsilon^3 {\cal D}_3 r\\
 &= \tilde\omega^2 D_{0}^{2}u^{(1)}+ \epsilon^{}\left[ 2
   \tilde\omega D_{0}D_{1}u^{(1)} + \tilde\omega^2
   D_{0}^{2}u^{(2)}\right] \\ \quad \quad \quad &\quad \quad
 +\epsilon^{2}\left[ 2\tilde\omega D_{0}D_{2}u^{(1)} +
   D_{1}^{2}u^{(1)}+2 \tilde\omega D_{0}D_{1}u^{(2)}+ \tilde\omega^2
   D_{0}^{2}r\right]\\ \quad &\quad \quad \quad \quad \quad \quad
 +\epsilon^{3}\left[2D_{1}D_{2}u^{(1)}+2 \tilde\omega D_{0}D_{2}u^{(2)}+D_{1}^{2}u^{(2)}   +  \mathcal{D}_3 r  \right]\\
& \quad \quad \quad \quad \quad \quad \qquad \qquad   +\epsilon^{4}\left[  D_{2}^{2}u^{(1)}+ 2D_{1}D_{2}u^{(2)}+ \epsilon D_{2}^{2}u^{(2)}\right].
\end{align*}
% the last two terms in the right hand side  will be part of the remainder  $R$ of equation \eqref{eq:1ddlamort-alpha3-force}.\\
We plug previous expansions into \eqref{eq:grosse-cub-forc};
we obtain:
\begin{align}
\label{eq:1ddlamort-alpha3-force}
\Bigg \{
\begin{array}[h]{rl}
& D_{0}^{2}u^{(1)}+  u^{(1)}= 0 \\
& \omega^2 \left(D_{0}^{2}u^{(2)}+    u^{(2)} \right)=S_{2} \\
& \omega^2 \left( D_{0}^{2}r+   r \right)= S_{3} 
\end{array}
\end{align}
with
\begin{align}
S_{2}&=-c {u^{(1)}}^{2} - d u^{(1)3} - 2 \omega  D_{0}D_{1}u^{(1)}-\lambda \omega D_{0}u^{(1)} -2\omega\sigma D_0^2u^{(1)}  + F_m \cos(T_{0}  )~~
  \text{ and } \\S_{3}&=-2 {c} u^{(1)}u^{(2)}-3d u^{(1)2}u^{(2)}- 2 \omega D_{0}D_{2}u^{(1)} -D_{1}^{2}u^{(1)}  - 2 \omega D_{0}D_{1}u^{(2)} -\sigma^2 D_0^2 u^{(1)} -2\sigma D_0 D_1 u^{(1)}   \\ 
&\quad \quad -2 \omega \sigma D_0^2 u^{(2)}- \lambda \omega D_{0} u^{(2)}- \lambda D_{1} u^{(1)}- \lambda \sigma D_0 u^{(1)}- \epsilon R(\epsilon,r,u^{(1)},u^{(2)})\label{rq:s3...amort}
\end{align}
with

à revoir
\begin{multline*}
 R(\epsilon,r,u^{(1)},u^{(2)})=2D_{1}D_{2}u^{(1)}+2 \omega D_{0}D_{2}u^{(2)}+D_{1}^{2}u^{(2)} \\
    +c u^{(2)2}+ 2cu^{(1)} r+ 3 d u^{(1)}u^{(2)2} +3 d u^{(1)2}r
 + \lambda(\omega D_0r + D_{2}u^{(1)}+ D_{1}u^{(2)}+ \epsilon D_{2}u^{(2)}) \\
  + \epsilon  \left ( D_{2}^{2}u^{(1)}+
 2 D_{1}D_{2}u^{(2)} + \epsilon D_{2}^{2}u^{(2)} \right )+\mathcal{D}_3 r
  +\lambda(\frac{dr}{dt}-\omega D_0r)+ \epsilon \rho(u^{(1)},u^{(2)},r,\epsilon) 
\end{multline*}

and
\begin{align*}
 \rho(u^{(1)},u^{(2)},r,\epsilon)&= 
 2 cru^{(2)}+ du^{(2)3}+ +6 d u^{(1)}u^{(2)}r \\
&\quad \quad + \epsilon( cr^{2}+ 3du^{(2)2}r+3du^{(1)}r^{2}) +\epsilon^{2}[3du^{(2)}r^2+\epsilon d r^{3}].
\end{align*}

% We note
% \begin{align*}
% \cos(\omega_{1} + \epsilon \sigma)T_{0}&= \cos(\omega_{1} T_{0}+ \epsilon \sigma T_{0})= \cos(\omega_{1} T_{0}+ \sigma T_{1})\\ 
% &= \cos(\omega_{1} T_{0}+ \beta_{1} + \sigma T_{1}-\beta_{1})= \cos(\theta + \sigma T_{1}-\beta_{1})\\
% &= \cos(\theta)\cos(\sigma T_{1}-\beta_{1})- \sin(\theta)\sin(\sigma T_{1}-\beta_{1})
%\end{align*}
 We solve the first equation of \eqref{eq:1ddlamort-alpha3-force}:
\begin{equation}
  \label{eq:u1forc=}
  u^{(1)}=a_{}(T_1,T_{2}) ~\cos\theta
\end{equation}
where we have set $\theta(T_{0},T_{1},T_{2})=
T_0+\beta_{}(T_1,T_{2})$; we use $\cos(T_0)=\cos(\theta)\cos(\beta)+\sin(\theta)\sin(\beta)$
and we obtain
\begin{multline*}
  S_2= -\frac{c a_{}^2}{2} (\cos(2\theta)+1) - \frac{d a_{}^{3}}{4}
  \cos(3\theta)  
+\sin(\theta ) \left[2\omega_{} D_1 a_{}+ \lambda \omega_{}a_{}  + F_m \sin(\beta_{})\right] \\
+ \cos(\theta) \left[2 \omega_{} a_{} D_1 \beta_{}- \frac{ 3 d a_{}^{3}}{4} +2\omega a \sigma + F_m \cos(\beta_{})\right]   
\end{multline*}
\begin{multline*}
\text{ or } \;  S_2= \cos(\theta) \left [\frac{-3d a^3}{4} +F_m \cos(
    \beta) \right] +2
  \omega [D_1a \sin(\theta)+a (D_1 \beta+\sigma ) \cos(\theta)] \\+\sin(\theta)
 \left  [\lambda \omega a  +F_m  \sin(  \beta) \right ]
 + S_2^{\sharp}
\end{multline*}
\begin{equation*}
\text{with } \;   S_2^{\sharp}= -\frac{c a_{}^2}{2} (\cos(2\theta)+1) - \frac{d a_{}^{3}}{4}
  \cos(3\theta).  
\end{equation*}
By imposing 
\begin{align}
\label{eq:d1...Famort}
\left\{
\begin{array}{rl}
& 2  \omega_{} D_1 a_{} + \lambda \omega_{}a_{} = -F_m \sin( \beta_{}) \\
& 2 \omega_{}a_{}  D_1 \beta +2\omega a \sigma-\frac{3da^3}{4}= -F_m\cos( \beta_{}),
\end{array}
\right.
\end{align}
the solution of the second equation of \eqref{eq:1ddlamort-alpha3-force} is:
\begin{equation}
 u^{(2)}=  \frac{-ca_{}^{2}}{2\omega_{}^{2}}~ + ~\frac{ca_{}^{2}}{6 \omega_{}^{2}}\cos(2\theta)+ \frac{da_{}^{3}}{32 \omega_{}^{2}}\cos(3\theta)
\end{equation} 
where we have omitted the term at the frequency $\omega_{}$ is which redundant with $u^{(1)}.$

The third equation of \eqref{eq:1ddlamort-alpha3-force} includes non
linearities, the unknown is $r$, we do not  solve it, but we show that
the solution is bounded on an interval which is $\epsilon$ dependent; the right
hand side is:
 \begin{multline*}
  S_3 = 
 \sin \theta \left[  2 \omega_{}
  D_{2}a_{} + \lambda a_{} D_{1}\beta_{}+ 2 D_{1}a_{}D_{1}\beta_{} + a_{}D_{1}^{2}\beta_{} +2 \sigma D_1a +\lambda a \sigma \right] \\
+\cos\theta \left[  2 \omega_{}a_{}D_{2}\beta_{} - \lambda D_{1}a_{}- D_{1}^{2}a_{}+ a_{} (D_{1}\beta_{})^{2}+ \sigma^2 a +2\sigma a D_1 \beta+ \frac{5 c^{2} a_{}^{3}}{6\omega_{}^2} -\frac{3 d^{2} a_{}^{5}}{128\omega_{}^2}\right]  \\ 
+S_{3}^{\sharp} - \epsilon R(\epsilon,r,u^{(1)},u^{(2)}) 
 \end{multline*}

where
à revoir
\begin{multline}
\label{eq:s3sharptheta}
S_{3}^{\sharp}= \frac{5 c d a_{}^{4}}{8\omega_{}^{2}} + \sin 2\theta
\left[ \frac{4 c a_{}}{3\omega_{}}D_{1}a_{}+ \lambda \frac{ c
    a_{}^{2}}{3\omega_{}} \right] 
+ \cos 2\theta \left[ \frac{4 c
    a_{}^{2}}{3\omega_{}} D_{1}\beta_{}+ \frac{ 15c d
    a_{}^{4}}{32\omega_{}^{2}} \right]  + \\ 
\sin 3\theta \left[ \frac{9 d a_{}^{2}}{16\omega_{}}D_{1}a_{} +
  \frac{3 \lambda d a_{}^{3}}{16\omega_{}} \right] + \cos 3\theta
\left[  \frac{9 d a_{}^{3}}{16\omega_{}}D_{1}\beta_{} - \frac{
    c^{2} a_{}^{3}}{6\omega_{}^{2}} - \frac{3 d^{2}
    a_{}^{5}}{64\omega_{}^{2}}\right]+ \\ 
\cos 4 \theta \left[ \frac{-3 c d a^{4}}{32\omega_{}^{2}} \right] - \frac{3 d^{2} a_{}^{5}}{128\omega_{}^{2}}\cos 5 \theta  
\end{multline}

To eliminate the secular terms,
we impose:
\begin{align}
\label{eq:d2...theta}
\left\{
\begin{array}{rl}
&  2 \omega_{} D_{2}a_{} +\lambda a_{} D_{1}\beta_{} + 2 D_{1}a_{}D_{1}\beta_{} + a_{}D_{1}^{2}\beta_{} +2\sigma D_1 a +\lambda a \sigma= 0\\
& 2 \omega_{}a_{}D_{2}\beta_{}- \lambda D_{1}a_{}- D_{1}^{2}a_{}+ a_{} (D_{1}\beta_{})^{2}+ \sigma^2 a + 2\sigma a D_1 \beta
+ \frac{5 c^{2} a_{}^{3}}{6\omega_{}^2} -\frac{3 d^{2} a_{}^{5}}{128\omega_{}^2} = 0.
\end{array}
\right.
\end{align}

In the system \eqref{eq:d1...Famort} 
the expression of $ D_{1}a_{}, D_{1}\beta$ can be extracted:
\begin{align}
\label{eq:D1a1...gamma}
\left\{
\begin{array}{rl}
&  D_1 a_{}=-\frac{F_m \sin(\beta)}{2\omega_{}}- \frac{\lambda a_{}}{2}  \\
&  D_1 \beta =  -\sigma-\frac{F_m\cos(\beta)}{2 a_{}\omega_{}}+\frac{3da_{}^2}{8\omega_{}}
\end{array}
\right.
\end{align}
%%bbb
As the functions $a_{}$ and $\beta_{}$ do not depend on $T_0$, the 
following relations hold:
\begin{align}
\label{eq:a=alpha...D1D1}
  \frac{da}{dt}=\epsilon D_{1}a+\epsilon^{2}D_{2}a+\mathcal 0\epsilon^{3})\\
 \frac{d\beta}{dt}=\epsilon D_{1}\beta_{} +\epsilon^{2}D_{2}\beta+\mathcal 0\epsilon^{3}).
\end{align}
We are going to express $ \frac{da}{dt},  \frac{d\beta}{dt}$ as
functions of $a, \beta$.
%bbbb
 We manipulate equation  \eqref{eq:d2...theta}
\begin{align*}
%\label{eq:d2...a-gamma}
\left\{
\begin{array}{rl}
& 2 \omega_{} D_{2}a_{}+  (\lambda a_{} +2D_1 a) ( \sigma + D_{1}\beta)   - a_{}D_{1}^{2}\beta_{} = 0\\
& 2 \omega_{}a_{}D_{2}\beta_{}- \lambda D_{1}a_{}-
D_{1}^{2}a_{}+ a_{} (\sigma + D_{1}\beta)^{2}+ \frac{5 c^{2} a_{}^{3}}{6\omega_{}^2} -\frac{3 d^{2} a_{}^{5}}{128\omega_{}^2} = 0
\end{array}
\right.
\end{align*}
then,  we replace $D_1a_{}, D_1 \beta_{}$ by their expression in
 \eqref{eq:D1a1...gamma},  we get
\begin{align*}
%\label{eq:d2...a-gamma}
\left\{
\begin{array}{rl}
& 2 \omega_{} D_{2}a_{}-  \frac{F_m \sin(\beta)}{\omega} ( \sigma + D_{1}\beta) - a_{}D_{1}^{2}\beta_{} = 0\\
& -2 \omega_{}a_{}D_{2}\beta_{}- \lambda D_{1}a_{}-
D_{1}^{2}a_{}+ a_{} (\sigma + D_{1}\beta)^{2}+ \frac{5 c^{2} a_{}^{3}}{6\omega_{}^2} -\frac{3 d^{2} a_{}^{5}}{128\omega_{}^2} = 0
\end{array}
\right.
\end{align*}
and
\begin{align}
\label{eq:d2...a-gamma}
\left\{
\begin{array}{rl}
& 2 \omega_{} D_{2}a_{}-  \frac{F_m \sin(\beta)}{\omega_{}} (-\frac{F_m\cos(\beta)}{2 a_{}\omega_{}}+\frac{3da_{}^2}{8\omega_{}} )    - a_{}D_{1}^{2}\beta_{} = 0\\
& -2 \omega_{}a_{}D_{2}\beta_{}- \lambda (-\frac{F_m \sin(\beta)}{2\omega_{}}- \frac{\lambda a_{}}{2})-
D_1^2a + 
a_{} (\frac{F_m\cos(\beta)}{2 a_{}\omega_{}}-\frac{3da_{}^2}{8\omega_{}})^{2}+ 
\frac{5 c^{2} a_{}^{3}}{6\omega_{}^2} -\frac{3 d^{2} a_{}^{5}}{128\omega_{}ç2} = 0.
\end{array}
\right.
\end{align}
On the other hand, we can determine $D_1^{2}a_{}$ and $D_1^{2}\beta $ by differentiating \eqref{eq:D1a1...gamma};
\begin{align*}
&D_1^2a=-\frac{F_m \cos(\beta) D_1 \beta}{2 \omega}-\frac{\lambda D_1
  a}{2}\\
&D_1^2 \beta=\frac{F_m \sin(\beta) D_1 \beta}{2 a
  \omega} 
+\left( \frac{F_m \cos(\beta) }{2 a^2 \omega}   +\frac{3da}{4 \omega}
\right)D_1 a
\end{align*}
or with \eqref{eq:D1a1...gamma}
\begin{align*}
&D_1^2a=-\frac{F_m \cos(\beta) }{2
  \omega} \left (-\sigma-\frac{F_m\cos(\beta)}{2
    a_{}\omega_{}}+\frac{3da_{}^2}{8\omega_{}}  \right) -\frac{\lambda }{2}  \left (-\frac{F_m \sin(\beta)}{2\omega_{}}-
  \frac{\lambda a_{}}{2} \right)\\
&D_1^2 \beta=\frac{F_m \sin(\beta) }{2 a
  \omega}  \left ( -\sigma-\frac{F_m\cos(\beta)}{2 a_{}\omega_{}}+\frac{3da_{}^2}{8\omega_{}}  \right) 
+\left( \frac{F_m \cos(\beta) }{2 a^2 \omega}
+\frac{3da }{4 \omega}
\right) \left (-\frac{F_m \sin(\beta)}{2\omega_{}}- \frac{\lambda a_{}}{2} \right)
\end{align*}
or
\begin{align*}
  &D_1^2 a=\frac{\sigma F_m \cos(\beta)}{2\omega}+\frac{F_m^2
    \cos^2(\beta)}{4 a \omega^2}-\frac{3 d a^2 F_m \cos(\beta)}{16
    \omega^2}+\frac{\lambda F_m \sin(\beta)}{4 \omega}+\frac{\lambda^2
    a}{4}\\
& D_1^2 \beta=-\frac{\sigma F_m \sin(\beta)}{2 a \omega}-\frac{F_m^2
  \sin(\beta) \cos(\beta)}{2 a^2 \omega^2}-\frac{3 d a F_m
  \sin(\beta)}{16 \omega^2} -\frac{\lambda  F_m \cos(\beta)}{4 a \omega}
-\frac{3d \lambda a^2}{8 \omega}.
\end{align*}
%We use \eqref{eq:D1a1...gamma} 
Then, in \eqref{eq:d2...a-gamma} we use   previous formula
\begin{align*}
%\label{eq:d2...a-beta}
\left\{
\begin{array}{rl}
& 2 \omega_{} D_{2}a_{}- \frac{F_m \sin(\beta)}{\omega_{}}  (-\frac{F_m\cos(\beta)}{2 a_{}\omega_{}}+\frac{3da_{}^2}{8\omega_{}} ) \\
&\qquad \qquad + a_{} \left(  
-\frac{\sigma F_m \sin(\beta)}{2 a \omega}-\frac{F_m^2
  \sin(\beta) \cos(\beta)}{2 a^2 \omega^2}-\frac{3 d a F_m
  \sin(\beta)}{16 \omega^2} -\frac{\lambda  F_m \cos(\beta)}{4 a \omega}
-\frac{3d \lambda a^2}{8 \omega}
\right) = 0\\
&2 \omega_{}a_{}D_{2}\beta_{}-\lambda (-\frac{F_m \sin(\beta)}{2\omega_{}}-\frac{\lambda a_{}}{2})-
\left(\frac{\sigma F_m \cos(\beta)}{2\omega}+\frac{F_m^2
    \cos^2(\beta)}{4 a \omega^2}-\frac{3 d a^2 F_m \cos(\beta)}{16
    \omega^2}+\frac{\lambda F_m \sin(\beta)}{4 \omega}+\frac{\lambda^2
    a}{4}
\right) \\
&\qquad \qquad \qquad \qquad \qquad  \qquad \qquad \qquad  + a_{} (\frac{-F_m\cos(\beta)}{2 a_{}\omega_{}}+\frac{3da_{}^2}{8\omega_{}})^{2}+ 
\frac{5 c^{2} a_{}^{3}}{6\omega_{}^2} -\frac{3 d^{2} a_{}^{5}}{128\omega_{}^2} = 0
\end{array}
\right.
\end{align*}
we manipulate
% nnnn 
% 
\begin{align*}
\Bigg\{
\begin{array}{rl}
&2\omega D_2 a -\frac{9 d a^2 F_m \sin(\beta)}{16 \omega^{2}}  -\frac{\sigma F_m
  \sin(\beta)}{2 \omega} -\frac{\lambda F_m\cos(\beta)}{4
  \omega}-\frac{3 d \lambda a^3}{8 \omega}=0\\
&2 \omega a D_2 \beta +\frac{\lambda F_m \sin(\beta)}{4
  \omega}+\frac{\lambda^2 a}{4} -\frac{\sigma F_m \cos(\beta)}{2
  \omega}-\frac{3d a^2 F_m \cos(\beta)}{16 \omega^2}-\frac{15 d^{2} a_{}^{5}}{128 \omega_{}^{2}}+ \frac{5
 c^{2}a_{}^{3}}{6\omega_{}^{2}}=0
\end{array}
\end{align*}
 and we obtain:
\begin{align}
\label{eq:D2a1...D2gamma}
\left\{
\begin{array}{rl}
&D_{2}a_{} = \frac{3 d \lambda a_{}^{3}}{16 \omega_{}^{2}} + \frac{\sigma F_m \sin\beta}{4 \omega_{}^{2}}+ \frac{\lambda F_m \cos\beta}{8\omega_{}^{2}}+ \frac{9da_{}^{2}F_m\sin\beta}{32 \omega_{}^{3}}\\
& D_{2}\beta =-\frac{\lambda^{2}}{8\omega_{}}- \frac{15 d^{2} a_{}^{4}}{256 \omega_{}^{3}}- \frac{5
 c^{2}a_{}^{2}}{12\omega_{}^{3}}
%\\ 
 +\frac{\sigma F_m \cos\beta}{4\omega_{}^{2}a_{}}+ \frac{3 d a_{}F_m \cos\beta}{32 \omega_{}^{3}}-\frac{\lambda F_m \sin\beta}{8\omega_{}^{2}a_{}}.
\end{array}
\right.
\end{align}

Now we return to \eqref{eq:a=alpha...D1D1} introducing
\eqref{eq:D1a1...gamma} and \eqref{eq:D2a1...D2gamma}, we obtain:
\begin{align}
\label{eq:dta1...dtgamma}
\left \{
  \begin{array}[h]{rl}
&\frac{da_{}}{dt} = \epsilon \left (-\frac{F_m \sin(\beta)}{2\omega_{}}-
   \frac{\lambda a_{}}{2} \right ) \\
& \qquad  \qquad  \qquad + \epsilon^{2} \left ( \frac{3 d
     \lambda a_{}^{3}}{16 \omega_{}^{2}} + \frac{\sigma F_m
     \sin\beta}{4 \omega_{}^{2}}+ \frac{\lambda F_m
     \cos\beta}{8\omega_{}^{2}}+ \frac{9da_{}^{2}F_m\sin\beta}{32
     \omega_{}^{3}} \right )+O(\epsilon^{3})\\ 
& \frac{d\beta}{dt}= \epsilon \left(-\sigma+\frac{3da_{}^2}{8\omega_{}}
-\frac{F_m\cos(\beta)}{2 a_{}\omega_{}} \right)+ 
\epsilon^{2} \bigg( -\frac{\lambda^{2}}{8\omega_{}}- \frac{15 d^{2}
  a_{}^{4}}{256 \omega_{}^{3}}- \frac{5
  c^{2}a_{}^{2}}{12\omega_{}^{3}}\\ 
& \qquad  \qquad  \qquad  \qquad  \qquad  \qquad +\frac{\sigma F_m
  \cos\beta}{4\omega_{}^{2}a_{}}+ \frac{3 d a_{}F_m \cos\beta}{32
  \omega_{}^{3}}-\frac{\lambda F_m \sin\beta}{8\omega_{}^{2}a_{}}
\bigg)+O(\epsilon^{3})
\end{array}
\right .
\end{align}

\paragraph*{Orientation: amplitude and phase equation.}
Equations \eqref{eq:dta1...dtgamma} ensure that
%%bb
$S_{3}^{\sharp}$ 
%%bbbb
has no term at frequency of $\omega_{1}$ or which goes to $\omega_{1}$ .\\
 This will allow us to justify this expansion in certain conditions; before we need to consider the stationnary solution of the system  \eqref{eq:dta1...dtgamma} and the stability of the solution close to the stationary solution.
This equation  \eqref{eq:dta1...dtgamma} is an extension for triple scale analysis of a similar equation  introduced in a preliminary work with  double scale analysis in \cite{nbb-br-doubl}.
\begin{rem}
  In this approach, we are using {\em the method of reconstitution}; this
  term has been introduced in 1985 in \cite{nayfeh86} in order to resolve a
  discrepancy between higher order approximation solutions obtained by
  multi scales method on the one hand and generalised averaging method on the other hand;
  it has been discussed in \cite{vestroni08} and from the engineering
  point of view, the controversy has
  been resolved in \cite{nayfeh-controv2005}; however the present
  mathematical proof of convergence seems new.
\end{rem}
\begin{rem}
The previous  equations are of importance to derive the solution of the equation \eqref{eq:grosse-cub-lib}; their stationary solution will provide an approximate periodic solution of \eqref{eq:grosse-cub-lib}.
\end{rem}

\subsubsection{Stationnary solution and stability}
\label{sss:stat-stab}
Let us consider the  stationary solution of\eqref{eq:dta1...dtgamma},
it satisfies:
\begin{align}
\label{eq:g1g2=0}
\left \{
  \begin{array}[h]{rl}
 & g_1(a,\beta,\sigma,\epsilon)=0, \\ & g_2(a,\beta,\sigma,\epsilon)=0
  \end{array}
\right .
\end{align}
with
\begin{align}
\label{eq:dta1...dtgamma=0}
\left\{
\begin{array}{rl}
&g_1=  \epsilon (-\frac{F_m \sin(\beta)}{2\omega_{}}- \frac{\lambda
  a_{}}{2} ) +\\
& \qquad    \qquad \epsilon^{2}( \frac{3 d \lambda a_{}^{3}}{16
  \omega_{}^{2}} + \frac{\sigma F_m \sin\beta}{4 \omega_{}^{2}}+
\frac{\lambda F_m \cos\beta}{8\omega_{}^{2}}
+\frac{9da_{}^{2}F_m\sin\beta}{32 \omega_{}^{3}})+\mathcal O(\epsilon^{3})\\ 
&
g_2= \epsilon
(-\sigma+\frac{3da^2}{8\omega_{}}-\frac{F_m\cos(\beta)}{2
  a_{}\omega_{}})\\
&\qquad +\epsilon^{2}( -\frac{\lambda^{2}}{8\omega_{}}- \frac{15 d^{2} a_{}^{4}}{256 \omega_{}^{3}}- \frac{5c^{2}a_{}^{2}}{12\omega_{}^{3}}
+\frac{\sigma F_m \cos\beta}{4\omega_{}^{2}a_{}}+ \frac{3 d a_{}F_m
  \cos\beta}{32 \omega_{}^{3}}-\frac{\lambda F_m
  \sin\beta}{8\omega_{}^{2}a_{}})+ \mathcal O(\epsilon^{3}).
\end{array}
\right.
\end{align}

Now, we study the stability of the solution of  \eqref{eq:dta1...dtgamma=0} 
in a neighbourhood of this  stationary solution noted  $(\bar{a}, \bar{\beta})$;
set  $a_{}=
\bar{a}+ \tilde{a} \text{ and } \beta_{}=\bar{\beta}+\tilde{\beta}$, the linearised  system is written :
$$\binom{\frac{d \tilde a}{dt}}{ \frac{d \tilde \beta}{dt}}=J\binom{\tilde a}{\tilde \beta}$$
with the jacobian matrix
\begin{displaymath}
J=\left(\begin{array}{ll}
 \partial_{\bar{a_{}}}g_{1} & \partial_{\beta}g_{1}  \\ 
   \partial_{\bar{a_{}}}g_{2} & \partial_{\beta}g_{2}
\end{array}\right)
\end{displaymath}
%nnnn on ne peut pas se débarasser de ce calcul? il me semble qu'il n'est pas necessaire de le mettre à vous de voir
we compute the partial derivatives:
\begin{align*}
  & \partial_{\bar{a_{}}}g_{1}=\epsilon(-\frac{ \lambda }{2})
  + \mathcal O(\epsilon^2) 
& \partial_{\bar{a_{}}}g_{2}=\epsilon \left( \frac{3d \bar a}{4 \omega}
  +\frac{F_m\cos(\beta)}{2 a^2 \omega} \right)+  \mathcal O(\epsilon^2)\\
& \partial_{\beta}g_{1}=-\epsilon \frac{F_m \cos(\beta)}{2\omega} +
\mathcal O(\epsilon^2)
& \partial_{\beta}g_{12}=\epsilon \frac{F_m \sin(\beta)}{2a \omega} +
\mathcal O(\epsilon^2)
\end{align*}
or:
\begin{align*}
& \partial_{\bar{a_{}}}g_{1}=\epsilon(-\frac{ \lambda }{2}) + \mathcal
O(\epsilon^2) %+ \epsilon^{2}(\frac{9 d \lambda \bar{a_{}}^{2}}{16\omega_{}^{2}})
& \partial_{\bar{a_{}}}g_{2}=\epsilon
(\frac{\sigma}{\bar{a_{}}}+\frac{9 d \bar{a_{}}}{8 \omega_{}})+
\mathcal O(\epsilon^2)\\ %+ \epsilon^{2}( \frac{\lambda^{2}}{8 \omega_{}\bar{a_{}}}+ \frac{75 d^{2} \bar{a_{}}^{3}}{256 \omega_{}^{3}} - \frac{5c^{2}\bar{a_{}}}{4\omega_{}^{3}}- \frac{3 d F_m \cos\gamma}{16\omega_{}^{3}})\\
& \partial_{\gamma}g_{1}=\epsilon (\sigma \bar{a_{}}-\frac{3 d
  \bar{a_{}}^{3}}{8 \omega_{}}) + \mathcal O(\epsilon^2) %+ \epsilon^{2}( - \frac{\lambda^{2}\bar{a_{}}}{8\omega_{}} +\frac{5c^{2}\bar{a_{}}^{3}}{12\omega_{}^{3}} -\frac{15 d^{2}\bar{a_{}}^{5}}{256\omega_{}^{3}} + \frac{3d\bar{a_{}}^{2}F_m\cos\gamma}{32\omega_{}^{3}}) \\
& \partial_{\gamma}g_{2}=\epsilon (-\frac{ \lambda}{2}) + \mathcal O(\epsilon^2)\\%+ \epsilon^{2}( \frac{3 d \lambda \bar{a_{}}^{2}}{16\omega_{}^{2}} - \frac{3d\bar{a_{}}F_m\sin\gamma}{16\omega_{}^{3}}).
\end{align*}
The matrix trace is 
$ tr(J)=-\lambda \epsilon $ and the  determinant is 
\\%+\epsilon^{2}(\frac{3d\lambda
                             %\bar{a_{}}^{2} }{4
                             %\omega_{}^{2}}-\frac{3d\bar{a_{}}F_m\sin\beta}{16\omega_{}^{3}})
                             %+ \mathcal O(\epsilon^3)$ 
\begin{equation}
 \det(J)=\epsilon^{2}\left[ -\frac{\lambda^{2}}{4}+ \sigma^{2}-\frac{3 d \sigma \bar{a_{}}^{2}}{2 \omega_{}}+ \frac{27 d^{2} \bar{a_{}}^{4}}{64 \omega_{}^{2}}\right] + O(\epsilon^{3})
\end{equation}
the two eigenvalues are negative for $\epsilon$ is small enough; when $$\sigma \le \frac{3 d \bar a^2}{4 \omega} -\frac{1}{2}\sqrt{\frac{9d^{2}\bar a^4}{16 \omega^{2}}-\lambda^2}$$
then the  solution of the  linearised system goes to zero; with the theorem of Poincar\'e-Lyapunov (look in the appendix for  the theorem \ref{th:poinc-lyapu}) when the initial data is close enough to the stationary solution, the solution of the system  \eqref{eq:dta1...dtgamma}, goes to the stationary solution.

\begin{proposition}
\label{prop:stab1ddl}
 When  $$\sigma \le \frac{3 d \bar a^2}{4 \omega}
 -\frac{1}{2}\sqrt{\frac{9d^{2}\bar a^4}{16 \omega^{2}}-\lambda^2}$$
and $\epsilon$ small enough, the stationary solution $(\bar a, \bar
\beta)$ of   \eqref{eq:dta1...dtgamma} is  stable in the sense of
Lyapunov (if the dynamic solution starts  close to the stationary
solution of\eqref{eq:dta1...dtgamma=0},  it remains close to it and
converges to it ); to the  stationary case corresponds the approximate
solution $\tilde{u}_{app}=\epsilon{} u_{app}$ of
\eqref{eq:grosse-cub-forc} 
$$\tilde{u}_{app}=\epsilon \bar a_{} \cos( \tilde \omega_{\epsilon}t + \bar
\beta)+ \epsilon^{2} \left [
\frac{-c\bar{a_{}}^{2}}{2\omega_{}^{2}}+\frac{c\bar{a_{}}^{2}}
{6\omega_{}^{2}}\cos(2( \tilde \omega_{\epsilon}t+  \bar \beta))+
\frac{d\bar{a_{}}^{3}}{32\omega_{}^{2}}\cos(3( \tilde
\omega_{\epsilon}t + \bar \beta )) \right] 
$$
with  
$$  \tilde \omega_{\epsilon}=\omega +\epsilon \sigma$$
It is periodic up to the order two.
\end{proposition}
\begin{rem}
 The expression of $u_{app} $ uses the remark
\begin{align*}
 u^{(1)}&= a_{}~ \cos(T_0 + \beta_{})= a_{}~ \cos(\tilde \omega_{\epsilon} t + \beta)
\end{align*}
and similarly for $u^{(2)}$.
\end{rem}
With this result of stability, we can  state precisely the approximation of the solution of \eqref{eq:grosse-cub-forc}

\subsubsection{Convergence of the expansion}
 
\begin{proposition}
\label{prop:conv-dev-forc}
 Consider the solution $\tilde u=\epsilon u$ of \eqref{eq:grosse-cub-forc} with initial
 conditions
 \begin{align}
   &\tilde u(0)=\epsilon a_{0}\cos(\beta_0) +
\epsilon^{2}[\frac{-ca_{0}^{2}}{2\omega_{}^{2}} +\frac{ca_{0}^{2}}{6\omega_{}^{2}}\cos(2\beta_0)
+\frac{da_{}^{3}}{32 \omega_{}^{2}}]\cos(3\beta_0) + \mathcal O(\epsilon^3),\\
&\dot{\tilde{ u}}(0)=-\epsilon \omega a_0
\sin(\beta_0)+ \epsilon^{2} [
\frac{-ca_{0}^{2}}{2\omega_{}^{2}}\sin(2\beta_0)-\frac{da_{0}^{3}}{32
  \omega_{}^{2}} \sin(3\beta_0)]+ \mathcal O(\epsilon^{3})
 \end{align}
with $(a_0,\beta_0)$ close of the stationary solution $(\bar a_{}, \bar \beta_{});$ 
%nnnnn j'ai mis epsilon² au lieu de epsilon bbb à voir
$$|a_0-\bar a| \le \epsilon^{2} C_1, |\beta -\bar \beta| \le \epsilon^{2} C_1$$
when  $\sigma \le \frac{3 d \bar a^2}{4 \omega}
-\frac{1}{2}\sqrt{\frac{9d^{2}\bar a^4}{16 \omega^{2}}-\lambda^2}$ and
$\epsilon$ small enough, there exists \\ $\varsigma>0$ such that for
all  $t<t_{\epsilon}=\frac{\varsigma}{\epsilon^{2}}$, the following
expansion of $\tilde u=\epsilon u$ is satisfied
\begin{align}
\left\{
\begin{array}{rl}
& \tilde u(t)=\epsilon a_{} (t)\cos(\tilde \omega_{\epsilon} t + \beta(t)) + 
\qquad \qquad    \\
& \qquad \qquad \epsilon^{2}[ \frac{-ca_{}^{2}}{2\omega_{}^{2}}+
~\frac{ca_{}^{2}}{6 \omega_{}^{2}}\cos(2 (\tilde \omega t +
\beta(t)))  + \frac{da_{}^{3}}{32 \omega_{}^{2}}\cos(3 ( \tilde \omega t + \beta(t)))] + \epsilon^3 r(\epsilon,t) \\
%& \text{with previous initial conditions}
 \end{array}
\right.
\end{align}
with $\tilde \omega_{\epsilon}= \omega  +\epsilon \sigma $ and 
  $r$ uniformly bounded  in  $C^{2}(0,t_{\epsilon})$ and with $a_{}, ~\beta_{}$ solution of \eqref {eq:dta1...dtgamma}
\end{proposition}

\begin{proof}
 Indeed after  eliminating
 terms at frequency $\nu_{1}$, we go back to the variable $t$ for  the third equation  \eqref{eq:1ddlamort-alpha3-force}.
\begin{equation*}
 \frac{d ^{2} r}{d t^{2}}+  \omega_{}^2 r= \tilde{S_3}
\end{equation*}
  with\\ 
$\tilde{S_3}=S_3^{\sharp}(t,\epsilon) -  \epsilon  \tilde R(u^{(1)}, u^{(2)},r,\epsilon)$ with $\tilde R=R-\mathcal{D}_{3}r- \lambda(\frac{dr}{dt}- D_0 r)$\\
with all the terms expressed with the variable $t$.
%bbb
We express $S_2^{\sharp}$ in \eqref{eq:s3sharptheta} by inserting $D_1a, D_1 \beta$ by
their expressions in \eqref{eq:d1...Famort} and using $\theta=\tilde
\omega_{\epsilon} t +\beta$;
this function is {\it} not periodic but is  {\it close} to a
periodic function $S_3^{\sharp}$ by replacing $\beta$ by $\bar \beta$.

 As  the solution of  \eqref{eq:dta1...dtgamma} is stable,  for $t\le \frac{\varsigma}{\epsilon^{2}}$:
\begin{equation*}
  |\beta_{}(\epsilon t,\epsilon^{2} t)-\bar{\beta_{}}|\le \epsilon^{2} C_1, \quad |a_{}(\epsilon t,\epsilon^{2}t) -\bar{a_{}}| \le \epsilon^{2} C_2
\end{equation*}
%bbbb
and
\begin{equation*}
  | S_3^{\sharp}- S_3^{\natural}| \le \epsilon^{2} C_3
\end{equation*}
so this difference may be included in the remainder $\tilde R$.
We use lemma 5.1 of Appendix (already introduced in \cite{nbb-br-doubl}); 
with $S =S_3^{\natural}$;  it  satisfies
lemma  hypothesis; similarly, we use $R=\tilde R$; it satisfies the hypothesis because it is a  polynomial  in the
variables $r,u_1, \epsilon$,with coefficients which are bounded
functions, so it is  lipschitzian on  bounded subsets.\\
\end{proof}

\begin{rem}
  The previous proposition states that for well prepared data close to the stationary solution, the triple  scales approximation converges in the sense that the difference between the solution and its approximation is equal to $\epsilon^3 r$ where $r$ is a function which remains bounded in $C^{2}(0,t_{\epsilon})$   with $t_{\epsilon}= \frac{\gamma}{\epsilon}$,  for some constant $\gamma$, with $\epsilon$ going to $0$.
\end{rem}

\subsubsection{Maximum of the  stationary solution,  primary
    resonance}
\label{subsubsec:maxstatsol}
We consider the stationary solution of \eqref{eq:dta1...dtgamma}, it satisfies,

\begin{align}
\label{eq:g1=0g2=0}
\left \{
  \begin{array}[h]{ll}
 & g_{1}(a,\beta,\sigma, \epsilon)=0, \\
 & g_{2}(a,\beta,\sigma, \epsilon)=0
  \end{array}
\right .
\end{align}
with  formulae \eqref{eq:dta1...dtgamma=0}.
We are going to find an expansion of $a, \beta, \sigma$ with respect to
 the small parameter $\epsilon$ when $\sigma \mapsto a$ reaches a maximum.
 The idea is that the functions $(\sigma, \epsilon) \mapsto (a,\beta)$ are defined implicitly by the
previous equations; the  jacobian matrix is

\begin{gather*}
\left (
  \begin{array}[h]{llll}
  g_{1a},  &g_{1\beta} & g_{1\sigma} & g_{1\epsilon}\\ 
 g_{2a} & g_{2\beta}  & g_{2\sigma}  & g_{2\epsilon} 
  \end{array}
\right )
\end{gather*}

and its  sub matrix $J_{a \beta}$ is:
\begin{gather*}
J(a, \beta)=\left (
  \begin{array}[h]{ll}
  g_{1a}  & g_{1\beta}\\
g_{2a}  & g_{2\beta} 
  \end{array}
\right )
\end{gather*}
 in paragraph \ref{sss:stat-stab},
 we have proved previously that  when $\sigma, \epsilon$ are small
 enough, $J_{a \beta}\neq 0$ and so with the
implicit  function theorem, in a neighbourhood of the
stationary solution, there exists a regular function 
$$ (\sigma, \epsilon) \longmapsto (a,\beta).$$
We first transform \eqref{eq:g1g2=0} 
\eqref{eq:dta1...dtgamma=0} in the following way
\begin{align}
\label{eq:g1+A1=0}
%\left\{
%\begin{aligned}
&  g_{1}(a,\beta,\sigma, \epsilon)= (-\frac{F_m \sin(\beta)}{2\omega}- \frac{\lambda
  a}{2} ) +    \qquad \epsilon A_1(a,\beta,\sigma)  + \mathcal
O(\epsilon^{2})=0 \\
\label{eq:g2+A2=0}
&  g_{2}(a,\beta,\sigma, \epsilon)= (-\sigma-\frac{3da^2}{8\omega}-\frac{F_m\cos(\beta)}{2
  a\omega})+ \qquad \epsilon A_2(a,\beta,\sigma) + \mathcal O(\epsilon^{2})=0
%\end{aligned}
%\right.
\end{align}
with 
\begin{align*}
 & A_1(a,\beta,\sigma)=  \frac{3 d \lambda a^{3}}{16
  \omega_{}^{2}} + \frac{\sigma F_m \sin\beta}{4 \omega_{}^{2}}+
\frac{\lambda F_m \cos\beta}{8\omega_{}^{2}}+
\frac{9da^{2}F_m\sin\beta}{32 \omega_{}^{3}} \\
&A_2(a,\beta,\sigma)= -\frac{\lambda^{2}}{8\omega_{}}- \frac{15 d^{2} a^{4}}{256 \omega_{}^{3}}- \frac{5c^{2}a^{2}}{12\omega_{}^{3}} \nonumber \\ 
& \qquad  \qquad  \qquad \qquad  \qquad+\frac{\sigma F_m \cos\beta}{4\omega_{}^{2}a}+ \frac{3 d aF_m \cos\beta}{32 \omega_{}^{3}}-\frac{\lambda F_m \sin\beta}{8\omega_{}^{2}a}
\end{align*}

\paragraph{ We  derive a first approximation} of $\sin \beta$ and
$\cos\beta$ by neglecting terms of order one in $\epsilon$:
\begin{align}
\label{eq:sin-cos-rough}
\left \{
  \begin{array}[h]{rl}
 &\frac{F_m \sin\beta}{2\omega_{}}=-\frac{\lambda a}{2}+ \mathcal O(\epsilon)\\
  &\frac{F_m \cos\beta}{2\omega_{}}= \frac{3da^{3}}{8\omega_{}}-\sigma a + \mathcal O(\epsilon) 
  \end{array}
\right .
\end{align}
 Using $\frac{d g_1}{d \sigma}=0$, we get
\begin{equation}
  \frac{F_m \cos(\beta)}{2 \omega}\frac{\partial \beta}{\partial
    \sigma}- \frac{\lambda}{2} \frac{\partial a}{\partial \sigma}+
  \epsilon \frac{d A_1}{d \sigma} +\mathcal O(\epsilon)=0
\end{equation}
When $a$ is maximum with respect to $\sigma$, we get another equation $\frac{\partial
  a}{\partial \sigma}=0$; with previous equation, we get a third
equation  $g_3=0$ with
% g3
\begin{equation*}
  g_{3}(a,\beta,\sigma, \epsilon)= 
  \frac{F_m \cos(\beta)}{2\omega}\frac{\partial \beta }{\partial
    \sigma}+ \epsilon \frac{d A_1}{d \sigma} +\mathcal O(\epsilon)
\end{equation*}
We have for $\epsilon=0$, $\frac{\partial g_3}{\partial a}=0$,
$\frac{\partial g_3}{\partial \sigma{}}=0$; we denote
$a_0^{\ast},\beta_0^{\ast}, \sigma_0^{\ast}$ the solution of the 3
equations for $\epsilon=0$.

We differentiate \eqref{eq:sin-cos-rough} with respect to $\sigma$; when $\frac{\partial
  a}{\partial \sigma}=0$, we obtain for the first approximation
\begin{gather}
\label{g1+Ag2+Adiff}
\left \{
  \begin{array}[h]{rl}
&\frac{F_m \cos( \beta^{\ast}_0)}{\omega} \frac{\partial
  \beta^{\ast}_0}{\partial \sigma}=0,\\
&-\frac{F_m\sin(\beta^{\ast}_0)}{2 \omega}\frac{\partial
  \beta^{\ast}_0}{\partial \sigma} +a^{\ast}_0=0
\end{array}
\right .
\end{gather}
and so $\cos( \beta^{\ast}_0)=0, \; \sin( \beta^{\ast}_0)= \pm 1$; if
we use  \eqref{eq:g1+A1=0}, we notice that a change of sign of
$\sin( \beta^{\ast}_0)$ changes the sign of $a$; so we choose $ \sin(
\beta^{\ast}_0)= -1$ and $a_0$ has the sign of $F_m$; then with
\eqref{eq:g1+A1=0}, \eqref{eq:g2+A2=0}, the following
equalities hold:
%nnn 
\begin{gather}
   \label{eq:a0s0=} a^{\ast}_0=\frac{ F_m}{\lambda \omega}, \; \sigma^{\ast}_0
=\frac{3da^{\ast 2}_0}{8 \omega}=\frac{3dF_m^2}{8 \lambda^2 \omega^3}; 
\end{gather}
with \eqref{g1+Ag2+Adiff},  we get also $\frac{\partial
  \beta^{\ast}_0}{\partial \sigma}=\frac{2 \omega
  a_0^*}{F_m}=\frac{2}{\lambda}$.
We remark that $c$ is not involved in these formulas.
Then we can compute for $\epsilon=0$,
$\frac{\partial g_3}{\partial a}=0; \;
\frac{\partial g_3}{\partial
  \beta}=
-\frac{F_m\sin(\beta)}{2\omega}
 \frac{\partial   \beta}{ \partial \sigma}
=-\frac{F_m}{\lambda \omega}; \; \frac{\partial g_3}{\partial \sigma}=0$.
So we obtain that the determinant of the extended matrix
% pareil ici
\begin{gather*}
J^{\clubsuit}(a, \beta, \sigma)=\left (
  \begin{array}[h]{lll}
  g_{1a}  & g_{1\beta}& g_{1,\sigma}\\
g_{2a}  & g_{2\beta} &g_{2,\sigma} \\
g_{3a}  & g_{3\beta} &g_{3,\sigma} 
  \end{array}
\right )
\end{gather*}
is not zero for $(a_0^*, \beta_0^*, \sigma_0^*)$; so once more, we can use the implicit function theorem to
define differentiable  functions 
$$  \epsilon \longmapsto (a^*,\beta^*, \sigma^*)$$
where  we denote  $a^{\ast},\beta^{\ast}, \sigma^{\ast}$ the
solution of the 3 equations.
\paragraph{After this first approximation, }
 we look for an expansion of these
functions:
$  \epsilon \longmapsto (a^*,\beta^*, \sigma^*)$;
%nnn il manque a^*_1 que j'ai mis, bbb oui
\begin{equation}
  \label{eq:ags_exp}
  a^{\ast}=a^{\ast}_0 +\epsilon a^{\ast}_1 + \mathcal
O(\epsilon^2), \; \beta^{\ast}=  \beta^{\ast}_0+\epsilon
\beta^{\ast}_1+ \mathcal O(\epsilon^2) ,\;
 \sigma^{\ast}= \sigma^{\ast}_0 +  \epsilon \sigma^{\ast}_1
 + \mathcal O(\epsilon^2).
\end{equation}
We perform some preliminary computations of $A_{1,0}^{\ast}=A_1(a_0^*,\beta_0^*,\sigma_0^*), \; A_{2,0}^{\ast}=A_2(a_0^*,\beta_0^*,\sigma_0^*)$;
$$A_{1,0}^{\ast}=  \frac{3 d \lambda a^{\ast3}_0}{16
  \omega_{}^{2}} + \frac{\sigma_0^{*} F_m \sin(\beta^{\ast}_0)}{4 \omega_{}^{2}}
+\frac{9da^{\ast2}_0F_m\sin\beta^{\ast}_0 }{32 \omega_{}^{3}} \label{eq:A1=}$$
$$A_{2,0}^{\ast}= -\frac{\lambda^{2}}{8\omega_{}}- \frac{15 d^{2} a^{\ast 4}_0}{256 \omega_{}^{3}}- \frac{5c^{2}a^{\ast 2}_0}{12\omega_{}^{3}} -\frac{\lambda F_m
 \sin\beta^{\ast}_0}{8\omega_{}^{2} a_0}$$

then, we use the values of \eqref{eq:a0s0=} and we get 
\begin{gather}
\begin{split}
\label{eq:A10A20=}
&  A_{1,0}^{\ast}=-\frac{F_m \sigma_0^*}{2 \omega^2}=-\frac{\lambda a_0^{\ast}\sigma_0^{\ast}}{2 \omega} ,
\quad \quad  \quad  \quad
  \frac{\partial A_{1,0}^{*} }{\partial \sigma}=\frac{F_m
    \sin(\beta^{\ast}_0)}{4 \omega^2} =-\frac{a_0^{*}
    \lambda}{4 \omega}\\
&  A_{2,0}^{\ast}= -\frac{15 d^{2} a^{\ast 4}_0}{256 \omega_{}^{3}}- \frac{5
 c^{2}a^{\ast 2}_0}{12\omega_{}^{3}}=
-\frac{5\sigma_0^{\ast 2} }{12 \omega}-\frac{5 c^2
    a_0^{\ast 2}}{12 \omega^3} , \quad \quad \quad
  \frac{\partial A_{2,0}^{*}}{\partial \sigma}=\frac{-F_m
    \cos(\beta_0^*)}{4\omega^2 a}=0
\end{split}
\end{gather}

\begin{align}
  \frac{\partial A_{1,0}^*}{\partial \beta}&=\frac{\sigma F_m
    \cos(\beta_0^*)}{4 \omega^2}-\frac{\lambda F_m \sin(\beta_0^*)}{8
    \omega^2} +\frac{9 d a^2 F_m \cos(\beta_0^*)}{32 \omega^3}
=\frac{\lambda F_m}{8
    \omega^2}=\frac{\lambda^2 a_0^*}{8 \omega}\\
  \frac{\partial A_{2,0}^*}{\partial \beta}&=
-\frac{\sigma_0^* F_m \sin(\beta_0^*)}{4 \omega^2  a_0^*}-\frac{3da_0^*F_m
\sin(\beta_0^*)}{32 \omega^3} -\frac{\lambda F_m \cos(\beta_0^*)}{8
\omega^2 a}
\\
&\qquad =\frac{\sigma_0^* F_m}{4 \omega^2
    a_0^*}+\frac{3da_0^*F_m}{32 \omega^3}=\frac{\sigma_0^*F_m}{2 \omega^2
    a_0^*}=\frac{ \sigma_0^* \lambda}{ 2\omega};
\end{align}

On the other hand, we notice that $\sin(\beta_0 +\epsilon \beta_1
+ \mathcal O(\epsilon^2))=-1+ \mathcal O(\epsilon^2)$ and with \eqref{eq:a0s0=}, 
we expand  formula \eqref{eq:g1+A1=0} to obtain
at second order
$$\frac{\lambda a_1^*}{2}=A_{1,0}^*= -\frac{\lambda a_0^* \sigma_0^*}{2 \omega}$$
and therefore 
\begin{equation}
  \label{eq:a1=}
a_1^*=-\frac{ a_0^* \sigma_0^*}{ \omega} 
\end{equation}

We compute 
\begin{align}
  \frac{\partial g_{1,0}^*}{\partial \sigma}&=\epsilon \frac{\partial
    A_{1,0}^*}{\partial \sigma} +\mathcal O (\epsilon^2) =-\epsilon  \frac{\lambda a_0}{4
    \omega} +\mathcal  \mathcal O (\epsilon^2)\\
 \frac{\partial g_{1,0}^*}{\partial \beta}&= \frac{F_m
   \cos(\beta)}{2 \omega} +\epsilon  \frac{\partial
    A_{1,0}^*}{\partial \beta} +\mathcal  \mathcal O(\epsilon^2)=-\epsilon \frac{F_m
   \beta_1^*}{2 \omega} -\epsilon\frac{\lambda^2 a_0^*}{8 \omega} +\mathcal O(\epsilon^2);
\end{align}
where we have used  $\cos(\beta_0+\epsilon \beta_1 +  \mathcal O(\epsilon^2))=- \epsilon
\beta_1 + \mathcal O(\epsilon^2)$ and $\frac{\partial a}{\partial \sigma}=0$
\begin{equation}
  \begin{split}
    \frac{d g_1}{d \sigma}&= \frac{\partial g_{1,0}^*}{\partial
      \sigma}+\frac{\partial g_{1,0}^*}{\partial \beta} 
    \frac{\partial \beta}{\partial \sigma} 
+\frac{\partial g_{1,0}^*}{\partial \beta} 
    \frac{\partial a}{\partial a} +\mathcal O(\epsilon^2) \\
&=\epsilon \left [ -\frac{\lambda a_0^*}{4  \omega} +\left (-\frac{F_m
   \beta_1^*}{2 \omega} -\frac{\lambda^2 a_0^*}{8 \omega}
\right) \left (-\frac{2}{\lambda} \right ) \right ]  +\mathcal O(\epsilon^2)\\
&= -\epsilon a_0^* ( \beta_1^*+\frac{\lambda}{2 \omega} ) +\mathcal O(\epsilon^2)
  \end{split}
\end{equation}
as $ \frac{d g_1}{d \sigma}=0$, we get

\begin{equation}
  \label{eq:gamma1=}
  \beta_1^*=-\frac{\lambda}{2 \omega}.
\end{equation}

 We use these approximations in the second equation
 \eqref{eq:g2+A2=0} to obtain
 \begin{equation}
   -(\sigma_0^*+\epsilon \sigma_1^*) +\frac{3da_0^{*2}}{8
     \omega}+6\epsilon \frac{da_0^*a_1^*}{8 \omega}
+\frac{F_m \beta_1^*}{2 a_0^* \omega}+\epsilon A_{2,0}^*+\mathcal O(\epsilon^2)=0
 \end{equation}
and hence
\begin{align}
  \sigma_1^*&= \frac{3da_0^*a_1^*}{4\omega}+\frac{F_m \beta_1^*}{2 a_0^*
    \omega}+A_{2,0}^*\\
&=\frac{3da_0^*}{4\omega}\left ( \frac{-a_0^*
    \sigma_0^*}{\omega}\right) +\frac{F_m }{2 a_0^*\omega} \left(
  \frac{-\lambda}{2 \omega}\right) +A_{2,0}^* \nonumber\\
&=-2 \frac{\sigma_0^2}{\omega}-\frac{\lambda^2}{4 \omega}
+A_{2,0}^*  \nonumber\\
&=-29\frac{\sigma_0^{*2}}{12 \omega}-\frac{5c^2 a_0^2}{12 \omega^3} -\frac{\lambda^2}{4 \omega}
\end{align}
\par{\it We can check the computations by using another way, see
  Appendix in subsection \ref{sec:anotherway} }
We remark that we get a frequency slightly different of  the free vibration frequency associated to the same amplitude.

\paragraph*{We have obtained the following important result.}
%nnn 
\begin{proposition}

  The stationary solution of \eqref{eq:dta1...dtgamma}
%,\ref{eq:D1aD1beta2} 
satisfies 
\begin{align}
\label{eq:g1+A_g2+A=0}
\left\{
\begin{array}{rl}
&  (\frac{F_m \sin(\beta)}{2\omega_{}}- \frac{\lambda
  a}{2} ) +    \qquad \epsilon A_1(a,\beta,\sigma)  + \mathcal O(\epsilon^{2})=0\\
&  (\sigma-\frac{3da^2}{8\omega_{}}+\frac{F_m\cos(\beta)}{2
  a\omega_{}})+ \qquad \epsilon A_2(a,\beta,\sigma) + \mathcal O(\epsilon^{2})=0
\end{array}
\right.
\end{align}
with 
\begin{align*}
 & A_1(a,\beta,\sigma)=  \frac{3 d \lambda a^{3}}{16
  \omega_{}^{2}} + \frac{\sigma F_m \sin\beta}{4 \omega_{}^{2}}+
\frac{\lambda F_m \cos\beta}{8\omega_{}^{2}}+
\frac{9da_{1}^{2}F_m\sin\beta}{32 \omega_{}^{3}} \\
&A_2(a,\beta,\sigma)= -\frac{\lambda^{2}}{8\omega_{}}- \frac{15 d^{2} a^{4}}{256 \omega_{}^{3}}- \frac{5c^{2}a_{1}^{2}}{12\omega_{}^{3}}  
 +\frac{\sigma F_m \cos\beta}{4\omega_{}^{2}a_{1}}+c \frac{3 d aF_m \cos\beta}{32 \omega_{}^{3}}-\frac{\lambda F_m \sin\beta}{8\omega_{}^{2}a_{1}}
\end{align*}

this stationary solution reaches its maximum amplitude for $\sigma=\sigma_0^*+
\epsilon \sigma_1^* + \mathcal O(\epsilon^2 ), \; a^*=a_0^*+\epsilon a_1^*
+\mathcal O(\epsilon^2), \; \beta^*=\beta_0^*+\epsilon \beta_1^* +\mathcal O(\epsilon^2)$
with 
\begin{gather}
    a^{\ast}_0=\frac{ F_m}{\lambda \omega}, \; \sigma^{\ast}_0
=\frac{3da^{\ast 2}_0}{8 \omega}=\frac{3F_m ^2}{8 \lambda^2 \omega^3} ,
\quad \beta_0^*=-\frac{\pi}{2}
\end{gather}
and
$$ \sigma_1^*=
-\frac{29}{12 \omega}\sigma_0^{*2}
- \frac{5c^{2}a_{0}^{* 2}}{12\omega_{}^{3}} -\frac{\lambda^2}{4 \omega}= -\frac{87 d^2 a_0^4}{256  \omega^3}- \frac{5c^{2}a_{0}^{* 2}}{12\omega_{}^{3}} -\frac{\lambda^2}{4 \omega} , \quad \beta_1^*= \frac{-\lambda}{2  \omega},
 \quad a_1^*=-\frac{ a_0^* \sigma_0^*}{ \omega}$$
%-\frac{\lambda^2a_0^*}{4 F},
 the periodic forcing  is at the angular frequency 
$$\tilde \omega_{\epsilon}= \omega_{}+\epsilon \sigma_0^*
+\epsilon^2 \sigma_1^* + \mathcal
O(\epsilon^2)
$$
it is slightly different of the approximate
angular frequency $\nu_{\epsilon}$ of  the undamped free periodic
solution associated to the same amplitude. 
\eqref{eq:nualpha};
for this frequency, the approximation (of the  solution $\tilde
u=\epsilon u$ of \eqref{eq:grosse-cub-forc} up to the
order $\epsilon^2$)  is periodic:
\begin{align}
\left\{
\begin{array}{rl}
&\tilde u(t)=\epsilon  a^* \cos( \tilde \omega_{\epsilon}t +
 \beta^* t) \\
& \quad \quad \quad + \epsilon^{2}[
\frac{-c{a}_{}^{*2}}{2\omega_{}^{2}}+\frac{c{a}_{}^{*2}}
{6\omega_{}^{2}}\cos(2 (\tilde \omega_{\epsilon}t +  \beta^* ))+
\frac{d{a}_{}^{*3}}{32\omega_{}^{2}}\cos(3( \tilde
\omega_{\epsilon} t +  \beta^* ) )] + \epsilon^3 r(\epsilon,t)\\
&\tilde u(0)=\epsilon {a}_{}^* + \epsilon^{2}[\frac{-c{a }_{}^{*2}}{3\omega_{}^{2}} +
\frac{d {a}_{}^{*3}}{32 \omega_{}^{2}}] + \mathcal O(\epsilon^3),
\quad \dot{u}(0)= \mathcal O(\epsilon^3)
 \end{array}
\right.
\end{align}
with r bounded in $C^{2}(0,t_{\epsilon})$
\end{proposition}
\begin{rem}
  We remark that, for $\epsilon$ small enough, this value of $\sigma^*$ is
  indeed smaller than  the maximal value that  $\sigma$ may reach
  in order that the previous expansion  converges as
indicated in  proposition \ref{prop:conv-dev-forc}.
\end{rem}
\begin{rem}
  We have obtained an expansion of $\tilde \omega_{\epsilon}$ up to
  order $\epsilon^2$ to be compared with the expansion with a double
  scale analysis (see in
  \cite{nbb-br-doubl}); in particular the amplitude dependence on the
  frequency of the applied force depends on the ratio of c and d; see
  numerical results below.

We have justified the basic behaviour of a primary resonance; many
other phenomena may appear like subharmonic resonances, see for
example  \cite{nayfeh86}.
\end{rem}
\begin{figure}[h]
\includegraphics[height=7cm,width=9cm]{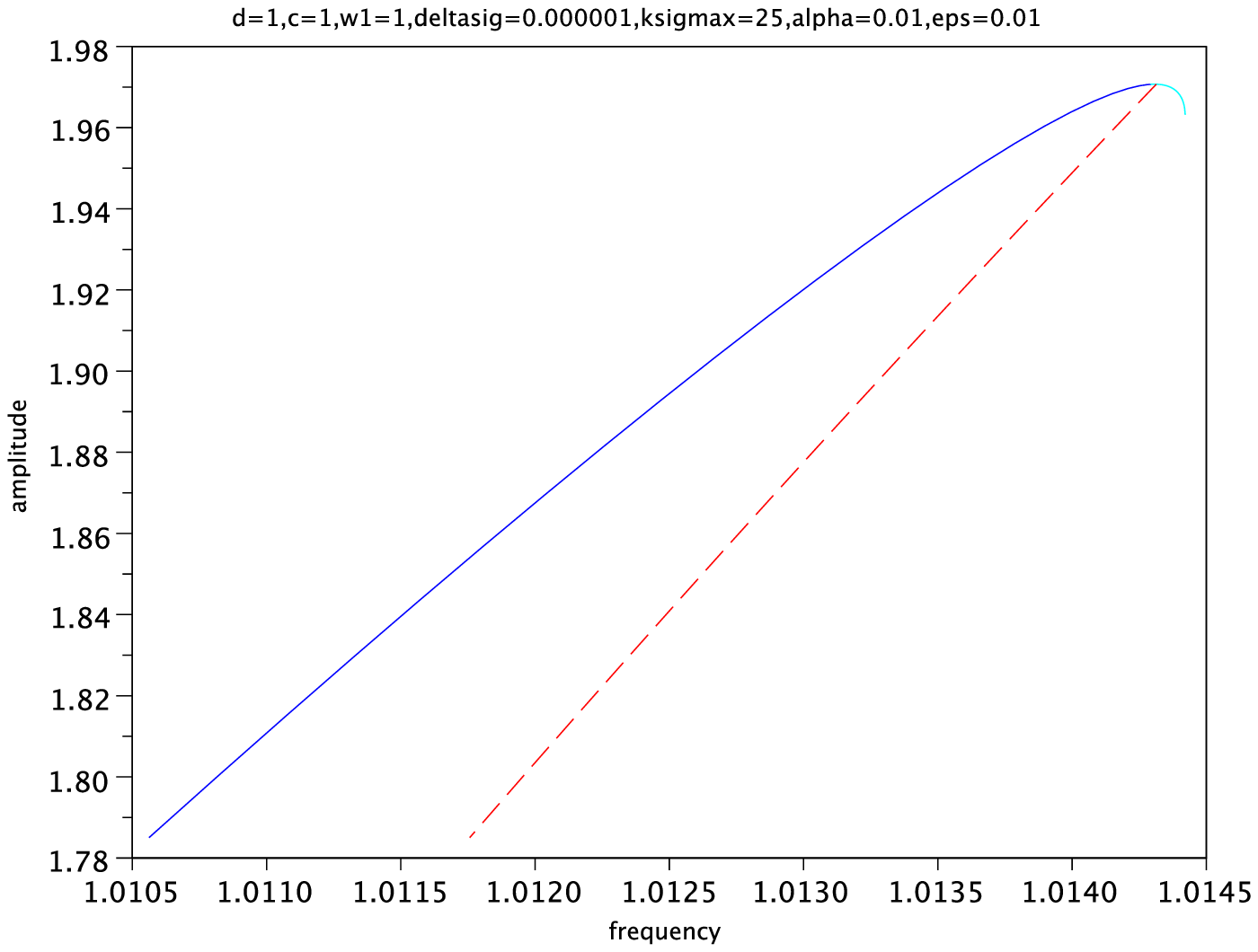}
\includegraphics[height=7cm,width=7cm]{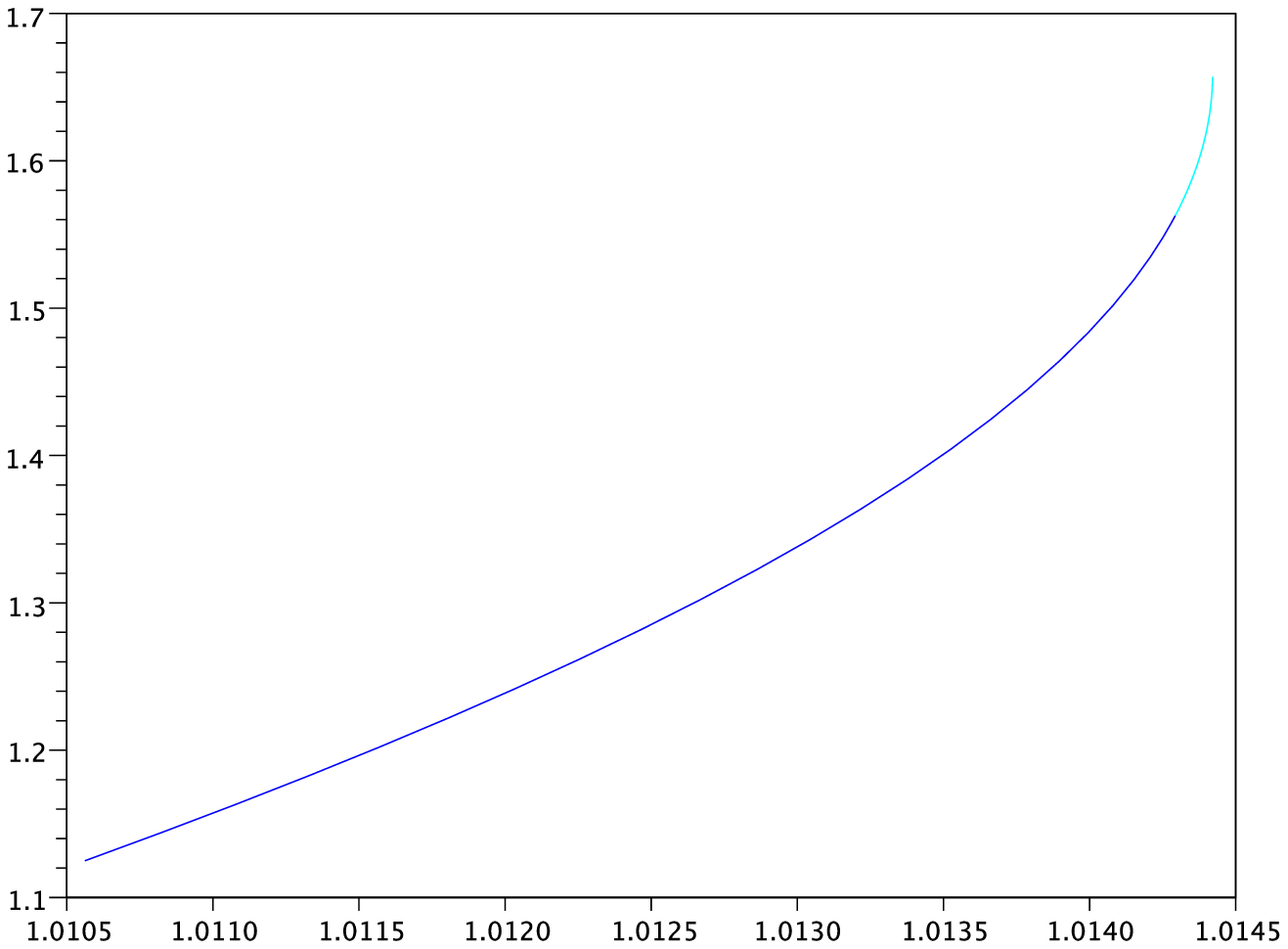}
\caption{ Left: amplitude versus frequency of stationary forced
  solution in blue and magenta; amplitude of free solution in
  red. Right: phase versus frequency of stationary forced solution}
\label{fig:ampligamm-freq1ddlc1}
\end{figure}
\begin{figure}[h]
\includegraphics[height=7cm,width=9cm]{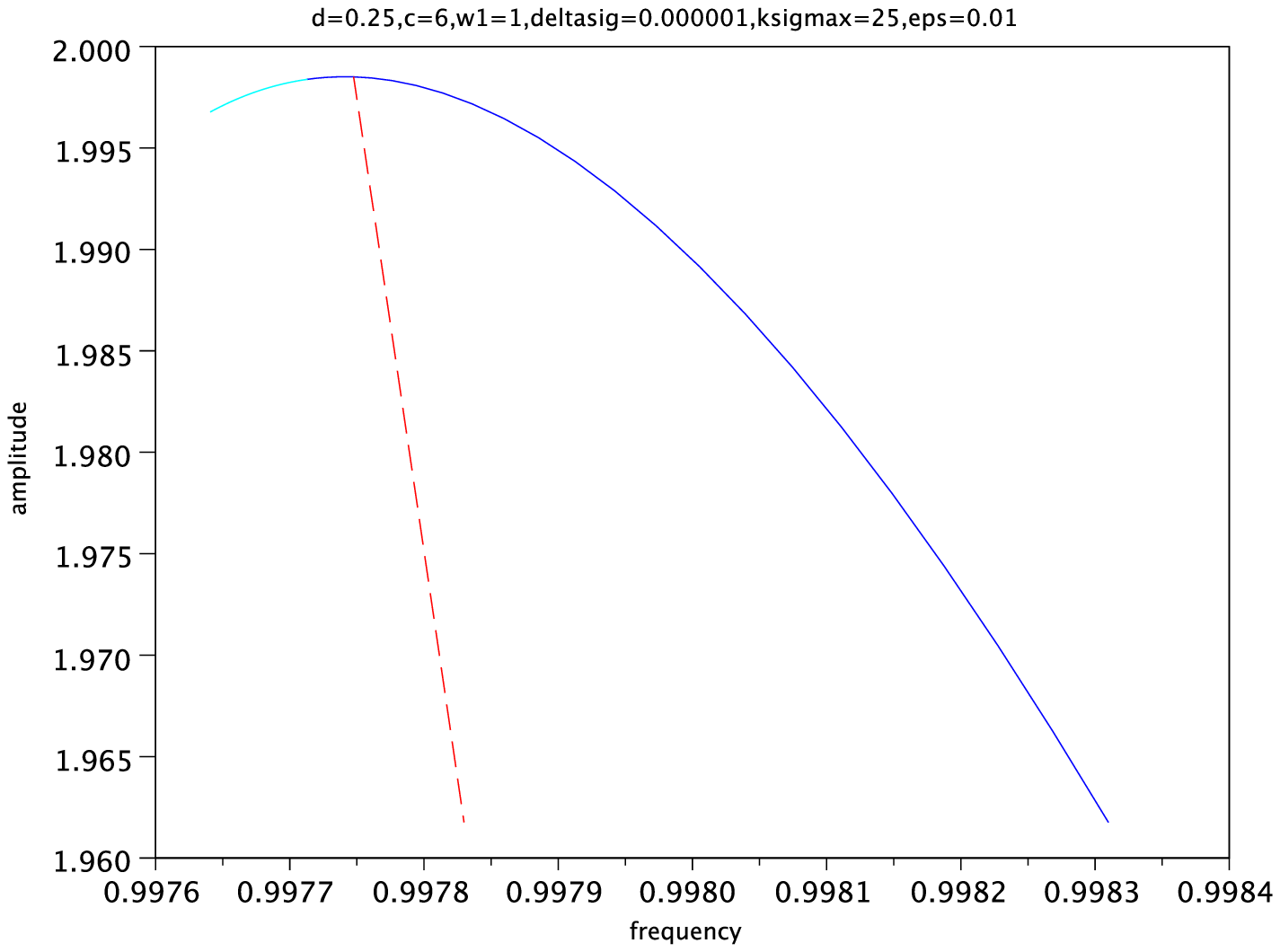}
\includegraphics[height=7cm,width=7cm]{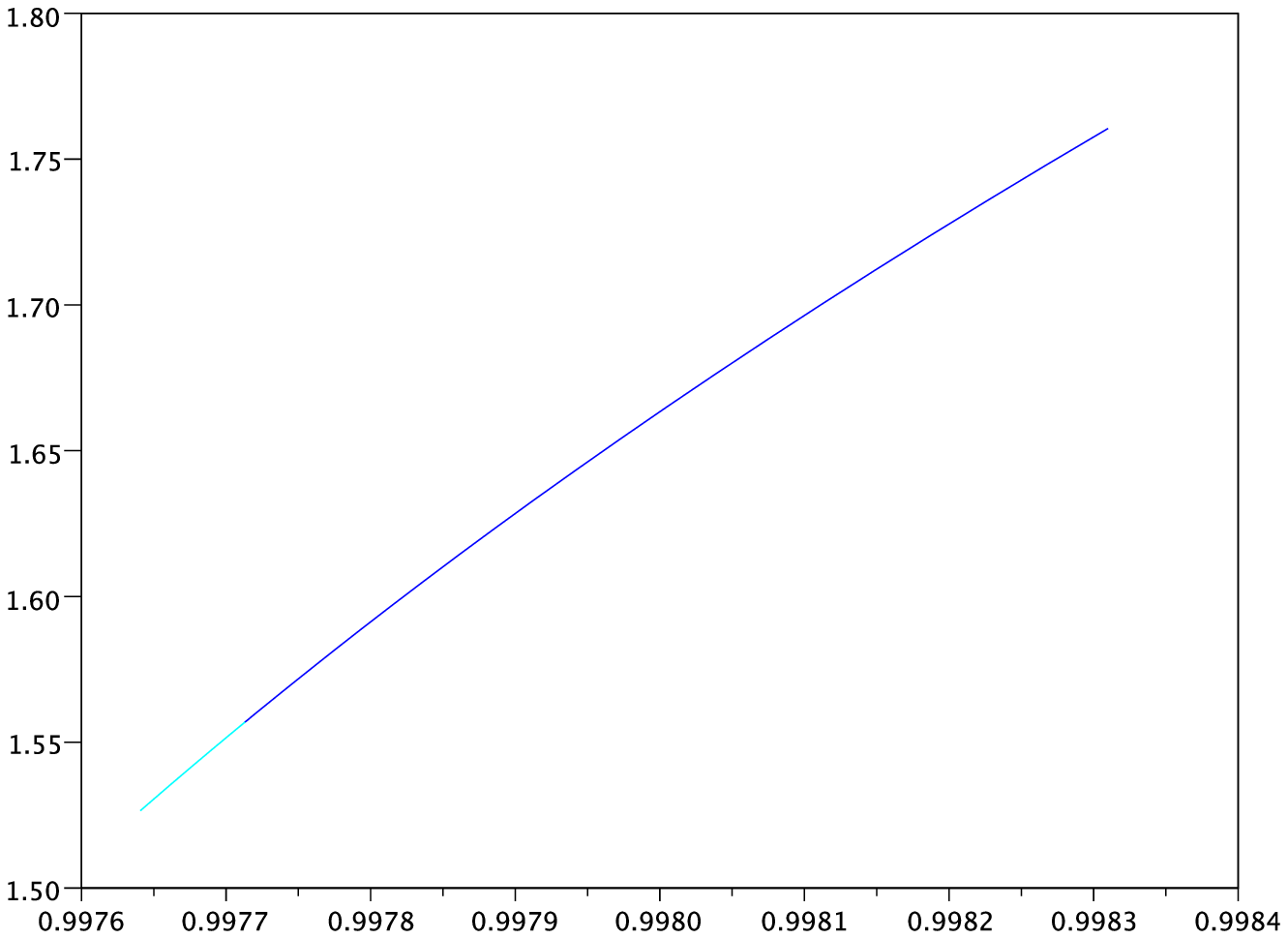}
\caption{ Left: amplitude versus frequency of stationary forced
  solution in blue and magenta; amplitude of free solution in red. Right: phase versus frequency of stationary forced solution }
\label{fig:ampligamm-freq1ddld025c6}
\end{figure}
In figure \ref{fig:ampligamm-freq1ddlc1}, we use 
 $\epsilon=0.01, \lambda=1/2, c=1, d=1, \omega=1, F=1$.
On the left,
the solid line displays the amplitude of the solution of this equation
with respect to
values of the frequency; we have solved \eqref{eq:g1=0g2=0}  with the routine {\sc fsolve} of
Scilab; it implements a variant of the hybrid method of Powell. In
proposition \ref{prop:stab1ddl}, the solution is stable when sigma is
small enough; the routine {\sc fsolve} fails to solve the equation when
$\sigma$ is too large; then we have exchanged the use of $\sigma$ and
$a$. The dotted line plots the amplitude of the free solution with
respect to its frequency. On the right, the phase $\gamma=-\beta$ is plotted
with respect to the frequency; it is also obtained by solving \eqref{eq:g1=0g2=0}  with the routine {\sc fsolve}.

In figure \ref{fig:ampligamm-freq1ddld025c6}, we use 
 $\epsilon=0.01, \lambda=1/2, c=6, d=1/4, \omega=1, F=1$.
On the left the solid line displays the amplitude of the solution with
respect to values of the frequency; on the right the phase $\gamma$ is
plotted. We notice that the behaviour is quite different of the
previous plots. 
\begin{rem}
  We emphasise that the behaviour of the last plots is linked to the
  ration of $c$ and $d$; this type of behaviour cannot be obtained with
  double scale expansion ; see \cite{nbb-br-doubl}.
\end{rem}

\section{System with  local quadratic and cubic non linearity}
\subsection{Free vibrations, triple scale expansion up to second order}
\label{subsec:syst-free-vib}
We consider a system of several vibrating masses attached to springs:

  \begin{equation}
\label{eq:syst}
 M \ddot {\tilde u} + K \, \tilde u + \Phi(\tilde u,\epsilon) = 0
\end{equation}
The mass matrix $M$ and the rigidity matrix $K$ are assumed to be symmetric and positive definite.
We assume that the non linearity is local, all components are zero except for two components $p-1, \;p$ which correspond to the end points of some spring assumed to be non linear:
\begin{equation}
  \label{eq:Phi=}
  \Phi_{p-1}(\tilde u)=c(\tilde u_p-\tilde u_{p-1})^2
  +\frac{d}{\epsilon}(\tilde u_p-\tilde u_{p-1})^3, \; \Phi_p=- \Phi_{p-1}
\end{equation}
In order to get an approximate solution, we are going to display the equation in the generalised eigenvector basis:
\begin{gather}
\label{eq:vect-val-propre}
  K \phi_k= \omega_k^2 M \phi_k, \text{ with } \phi_k^T ~M ~\phi_l= \delta \phi_{kl}, ~~~ k,l=1 \dots, n
\end{gather}
So we perform the change of functions:
\begin{equation}
\label{eq:uyphi}
 \tilde u=\sum_{k=1}^{n} \tilde y_{k}\phi_{k};~K\tilde
 u=\sum_{k=1}^{n} \tilde y_{k}K\phi_{k}=\sum_{k=1}^{n} \tilde y_{k}\omega_{k}^{2} M \phi_{k};~ 
M \ddot{\tilde u}=\sum_{k=1}^{n}  \ddot{\tilde y}_{k} M\phi_{k}
\end{equation}
we obtain
\begin{equation*}
  \ddot{\tilde y}_k +\omega_k^2 \tilde y_k+\phi_k^T \Phi(\sum_{i=1}^n
  \tilde y_i \phi_i,\epsilon)=0, ~~~ k=1 \dots, n
\end{equation*}
As $\Phi$ has only 2 components which are not zero, it can be written
\begin{equation*}
\label{eq:yk..Phip}
  \ddot{\tilde y}_k +\omega_k^2 \tilde y_k+\left (
    \phi_{k,p-1}-\phi_{k,p}\right )\Phi_{p-1}(\sum_{i=1}^n \tilde y_i \phi_i,\epsilon)=0,  ~~~ k=1 \dots, n
\end{equation*}
or more precisely
\begin{multline}
\label{eq:yktilde..Phipcd}
  \ddot{\tilde y}_k +\omega_k^2 \tilde y_k+\left (
    \phi_{k,p-1}-\phi_{k,p}\right ) \Bigg[c \left (\sum_{i=1}^n \tilde
    y_i (\phi_{i,p}-\phi_{i,p-1}) \right)^2 +\\
 \frac{d}{\epsilon} \left (\sum_{i=1}^n \tilde y_i( \phi_{i,p}-\phi_{i,p-1}) \right)^3 \Bigg]=0,  ~~~ k=1 \dots, n
\end{multline}
\begin{rem}
  As we intend to look for a small solution, we consider a change of
  function \fbox{ $\tilde y_k=\epsilon y_k$} and we obtain the transformed equation:
\end{rem}
\begin{multline}
\label{eq:yk..Phipcd}
  \ddot y_k +\omega_k^2 y_k+\left ( \phi_{k,p-1}-\phi_{k,p}\right )
  \Bigg[\epsilon c \left (\sum_{i=1}^n y_i (\phi_{i,p}-\phi_{i,p-1}) \right)^2 +\\
 \epsilon d \left (\sum_{i=1}^n y_i( \phi_{i,p}-\phi_{i,p-1}) \right)^3 \Bigg]=0,  ~~~ k=1 \dots, n
\end{multline}

\subsubsection{ Derivation of  an asymptotic expansion}
As for the 1 degree of freedom case, we use a triple  scale expansion to compute
an approximate  small solution; more precisely, {\it we look for a solution
close to a normal mode of the associated linear system}; we denote
this mode by subscript $\omega_{1}$; obviously by permuting the coordinates,
this subscript could be anyone (different of  $p$, this case would
give similar results with slightly different formulae); we set
\begin{equation}
  T_0= \omega_1 t, \quad T_1=\epsilon t , \quad T_2=\epsilon^{2} t  ~ \text{hence
     } \quad  D_{0}y_k= \frac{\partial y_k}{\partial T_{0}},~  D_{1}y_k=
  \frac{\partial y_k}{\partial T_{1}} \text{  and }  D_{2}y_k= \frac{\partial y_k}{\partial T_{2}}
\end{equation}
and we use the {\it ansatz}:
\begin{equation}
\label{eq:yk=alpha...r}
y_{k}(t)= y_{k}(T_0, T_{1},T_{2}) =  y_{k}^{(1)}(T_0, T_{1},T_{2}) + \epsilon^{}y_{k}^{(2)}(T_0, T_{1},T_{2})+\epsilon ^{2} r_{k}(T_{0}, T_{1},T_{2})
\end{equation} 

So we have:
\begin{align}
\begin{split}
\label{eq:yk2=eps...r}
 \frac{d^2y_{k}}{dt^2}&= \omega_1^2 D_{0}^{2}y_{k}^{(1)}+
 \epsilon^{}\left[ 2 \omega_1 D_{0}D_{1}y_{k}^{(1)} + D_{0}^{2}y_{k}^{(2)}\right] \\
&\quad\quad +\epsilon^{2}\left[ 2\omega_1 D_{0}D_{2}y_{k}^{(1)} +
  D_{1}^{2}y_{k}^{(1)}+2\omega_1 D_{0}D_{1}y_{k}^{(2)}+
  D_{0}^{2}r\right]\\ &\quad \quad \quad
+\epsilon^{3}\left[2D_{1}D_{2}y_{k}^{(1)}+2 \omega_1 D_{0}D_{2}y_{k}^{(2)}+D_{1}^{2}y_{k}^{(2)}   +  \mathcal{D}_3 r_{k}  \right]\\
&\quad \quad \quad \quad  +\epsilon^{4}\left[  D_{2}^{2}y_{k}^{(1)}+ 2D_{1}D_{2}y_{k}^{(2)}+ \epsilon D_{2}^{2}y_{k}^{(2)}\right] 
\end{split}
\end{align}
with
\begin{equation*}
  \mathcal{D}_3 r_{k} = \frac{1}{\epsilon}\left( \frac{d^2 r_{k}}{d
      t^2} -\omega_1^2 D_0^2 r_{k} \right)= 2 \omega_1 D_0D_1 r_{k} + \epsilon [2
  \omega_1 D_0D_2 r_{k}+ D_1^2 r_{k}]+ 2\epsilon^{2} D_1D_2 r_{k}+\epsilon^{3} D_2^2 r_{k}
\end{equation*} 

We plug previous expansions \eqref{eq:yk=alpha...r} and \eqref{eq:yk2=eps...r} into \eqref{eq:yk..Phipcd};
by identifying the coefficients of the powers of $\epsilon$ in the expansion of \eqref{eq:yk..Phipcd}, we get:
\begin{align} \left\{
\begin{array}{rll}
& \omega_1^2 D_{0}^{2}y_{k}^{(1)}+\omega_{k}^{2}y_{k}^{(1)}=0~~~~,~~ &k=1 \dots, n \\ 
&  \omega_1^2 D_{0}^{2}y_{k}^{(2)}+\omega_{k}^{2}y_{k}^{(2)}= S_{2,k}~~,~~&k=1
\dots, n \label{eq:nddllibre-alpha}\\ 
& \omega_1^2 D_{0}^{2}r_{k}+\omega_{k}^{2} r_{k} = S_{3,k}~~~~,~~~~&k=1 \dots, n 
%\label{eq:nddllibre-alpha3}
\end{array}
\right.
\end{align}
where $ S_{2,k},  S_{3,k}$ are defined below; to simplify the manipulations, we  set  $\delta \phi_{kp}=(\phi_{k,p}-\phi_{k,p-1})$; 
\begin{multline*}
 S_{2,k}= -c \delta \phi_{kp}\left( \sum_{l,m} y_{l}^{(1)}\delta \phi_{lp}
  y_{m}^{(1)}\delta \phi_{mp}\right) - d \delta \phi_{kp}\left(
  \sum_{g,l,o}y_{g}^{(1)}y_{l}^{(1)}\delta \phi_{lp}y_{o}^{(1)}\delta \phi_{gp}\delta \phi_{op}\right)
- 2 \omega_1 D_{0}D_{1} y_{k}^{(1)}
\end{multline*}

\begin{multline*}
   S_{3,k}= - c \delta \phi_{kp}\left( \sum_{l,j}
     y_{l}^{(1)}y_{j}^{(2)}\delta \phi_{lp} \delta \phi_{jp}\right) -  d
   \delta \phi_{kp}\left( \sum_{h,g,l}
     y_{h}^{(1)}y_{g}^{(1)}y_{l}^{(2)}\delta \phi_{hp} \delta \phi_{gp}
     \delta \phi_{lp}\right)  \\-2 \omega_1 D_{0}D_{2}y_{k}^{(1)}  -
   D_{1}^{2} y_{k}^{(1)}-2 \omega_1 D_{0}D_{1}y_{k}^{(2)} - \epsilon R_{k}(y_{1}^{(1)}, y_{1}^{(2)},r_{k}, \epsilon)
\end{multline*}
with
\begin{multline*}
 R_{k}(\epsilon,r_{k}, y_{k}^{(1)},y_{k}^{(2)}) =
 2D_{1}D_{2}y_{k}^{(1)}+2 \omega_1 D_{0}D_{2}y_{k}^{(2)}+D_{1}^{2}y_{k}^{(2)}\\
+c \delta \phi_{kp}\left( \sum_{l,j}
     y_{j}^{(2)}y_{j}^{(2)}\delta \phi_{lp} \delta \phi_{jp}\right)  
+c \delta \phi_{kp}\left( \sum_{l,j}
     y_{j}^{(1)}r_{l}\delta \phi_{jp} \delta \phi_{lp}\right)  \\
+d \delta \phi_{kp}\left( \sum_{h,g,l}
     y_{h}^{(1)}y_{g}^{(2)}y_{l}^{(2)}\delta \phi_{hp} \delta \phi_{gp}
     \delta \phi_{lp}\right)
+d \delta \phi_{kp}\left( \sum_{h,g,l}
     y_{h}^{(1)}y_{g}^{(1)}r_{l}\delta \phi_{hp} \delta \phi_{gp}
     \delta \phi_{lp}\right) \\
+\mathcal{D}_3 r_{k}  + \epsilon(D_{2}^{2}y_{k}^{(1)}+ 2D_{1}D_{2}y_{k}^{(2)}+ \epsilon D_{2}^{2}y_{k}^{(2)})
  + \epsilon \rho(y_{k}^{(1)}, y_{k}^{(2)},r_{k} ,\epsilon)
\end{multline*}

and with a polynomial in the variables $ r_n$ with coefficients  $y_l^{(1)},y_m^{(2)}$,
%bbb pas de deltakp???c bon
\begin{multline}
\label{eq:rho=}
\rho (y_{k}^{(1)}, y_{k}^{(2)},r_{k},\epsilon)=  c \delta \phi_{kp}\left( \sum_{l,j}
     y_{l}^{(2)}r_{j}\delta \phi_{lp} \delta \phi_{jp}\right) 
+d \delta \phi_{kp}\left( \sum_{h,g,l}
     y_{h}^{(1)}y_{g}^{(2)}r_{l}\delta \phi_{hp} \delta \phi_{gp}
     \delta \phi_{lp}\right) \\
 \qquad  \qquad +d \delta \phi_{kp}\left( \sum_{h,g,l}
     y_{h}^{(2)}y_{g}^{(2)}y_{l}^{(2)}\delta \phi_{hp} \delta \phi_{gp}
     \delta \phi_{lp}\right) \\
+\epsilon c\Bigg  [  c \delta \phi_{kp}\left( \sum_{l,j}
     r_{l}r_{j}\delta \phi_{lp} \delta \phi_{jp}\right)  
+d \delta \phi_{kp}\left( \sum_{h,g,l}
     y_{h}^{(2)}y_{g}^{(2)}r_{l}\delta \phi_{hp} \delta \phi_{gp}
     \delta \phi_{lp}\right)
\\
 \qquad  \qquad  \qquad  \qquad + d\delta \phi_{kp}\left( \sum_{h,g,l}
     y_{h}^{(1)}r_{g}r_{l}\delta \phi_{hp} \delta \phi_{gp}
     \delta \phi_{lp}\right) \Bigg ] \\
+ \epsilon^2   d \delta \phi_{kp}\left( \sum_{h,g,l}
     y_{h}^{(2)}r_{g}r_{l}\delta \phi_{hp} \delta \phi_{gp}
     \delta \phi_{lp}\right)  
+ \epsilon^3   d \delta \phi_{kp}\left( \sum_{h,g,l}
     r_{h}r_{g}r_{l}\delta \phi_{hp} \delta \phi_{gp}
     \delta \phi_{lp}\right)
\end{multline}
We set $\theta(T_0,T_1,T_{2})= T_0+\beta_{1}(T_1,T_{2})$; we
note that $D_0 \theta=1, \; D_1 \theta=D_1 \beta$ and $D_{2}\theta=D_{2}\beta_{1}$;\\
 we solve the first set of equations \eqref{eq:nddllibre-alpha},
 imposing $O(\epsilon^3)$ initial Cauchy data for $k \neq 1$ and $D_0 y_1^{(1)}(0)=0$; we get:
 
 \begin{equation}
\label{eq:y1=}
\left\{ 
\begin{array}{rl}
& y_{1}^{(1)}=a_{1}(T_1,T_{2}) ~\cos(\theta )\\
& y_{k}^{(1)}= 0 ,~~~~ k=2 \dots n
\end{array}
\right.
\end{equation} 

 \begin{rem}
  We note that $a_{1}$ and $\beta_{1}$ are not constants but functions
  of times $T_{1}$ and $T_{2}$ because u depends on these times
  scales. The dependence of these functions with respect to $T_{1}$
  and $ T_{2}$ will be determined by solving the equations of the following orders and eliminating secular terms.
\end{rem}
  First, we determine the dependence in $T_{1}$;
we manipulate the right hand sides:
\begin{multline*}
%\label{eq:S21=theta}
 S_{2,1}= -\delta \phi_{1p}
 \Big[\frac{ca_1^2}{2}(1+\cos(2\theta))\delta
 \phi_{1p}^2+\frac{da_1^3}{4} (\cos(3\theta)+3\cos(\theta))\delta
 \phi_{1p}^3 \Big] \\ 
+ 2\omega_1 \left [ a_1D_1 \beta_1 \cos(\theta)+ D_1 a_1 \sin(\theta)
\right ]
\end{multline*}

\begin{multline*}
 S_{2,k}= -\delta \phi_{kp}\Big[\frac{ca_1^2}{2}(1+\cos(2\theta))\delta
 \phi_{1p}^2+\frac{da_1^3}{4}(\cos(3\theta)+3\cos(\theta)) \delta
 \phi_{1p}^{3}\Big] , \text{ for } k \neq 1\\
\end{multline*}
In $S_{2,1}$, we gather the terms at angular frequency $\omega_1$;
\begin{equation}
  S_{2,1}= -3\frac{da_1^3}{4}\cos(\theta))\delta
 \phi_{1p}^4  
+ 2\omega_1 \left [ a_1D_1 \beta_1 \cos(\theta)+ D_1 a_1 \sin(\theta)
\right ] +S_2^{\sharp}
\end{equation}
with
\begin{equation*}
   S_{2,1}^{\sharp}= -\delta \phi_{1p}
 \Big[\frac{ca_1^2}{2}(1+\cos(2\theta))\delta
 \phi_{1p}^2+\frac{da_1^3}{4} \cos(3\theta) \delta
 \phi_{1p}^3 \Big ]
\end{equation*}
It appears some terms at the frequency of the system, these terms provide a solution
$y_{1}^{(2)}$ of the equation \eqref{eq:nddllibre-alpha} which is non periodic
and non bounded over long time intervals.
We will eliminate these terms by imposing: 

 \begin{align}
\label{eq:a1...beta1}
\left\{ 
\begin{array}{rl}
  & D_{1}a_{1}=0\\
  & D_{1}\beta_{1}= \frac{3 d \delta \phi_{1p}^{4} a_{1}^{2}}{8\omega_{1}}
\end{array}
\right.
 \end{align}

and if we assume that $\omega_1^2$ is a simple eigenvalue and
$\omega_k^2 \neq 9\omega_1^2 , ~ \omega_k^2 \neq 4 \omega_1^2$ (no internal resonance), the solution of the second equation \eqref{eq:nddllibre-alpha}  is:
\begin{equation}
\left\{ 
\begin{array}{rl}
 & y_{1}^{(2)}= \delta \phi_{1p}^{3}[-\frac{c a_{1}^{2}}{2\omega_{1}^{2}}~ + ~\frac{c a_{1}^{2}}{6 \omega_{1}^{2}}\cos(2\theta)]+\delta \phi_{1p}^{4}\frac{d a_{1}^{3}}{32 \omega_{1}^{2}}\cos(3\theta)\\
& y_{k}^{(2)}=\delta \phi_{kp}\delta \phi_{1p}^{2}[-\frac{c
  a_{1}^{2}}{2\omega_{k}^{2}}+
\frac{ca_{1}^{2}}{2(4\omega_{1}^{2}-\omega_{k}^{2})}
\cos(2\theta)]+\delta \phi_{kp}\delta \phi_{1p}^{3}\frac{d
  a_{1}^{3}}{4(9\omega_{1}^{2}-\omega_{k}^{2})} \cos(3\theta), \quad
k=2, \dots , n.
\end{array}
\right.
\end{equation}
 
where we have omitted the term at angular frequency $\omega_{1}$ which is redundant with $y_{1}^{(1)}$.\\
For the third set of equations of \eqref{eq:nddllibre-alpha}, $r$ is the
unknown, this equation contains non-linearities, we do not  solve it
but we show that the solution is bounded on an interval dependent of
$\epsilon$. The right hand side, after some manipulations is:
\begin{multline*}
S_{3,1}= \sin(\theta) \big(2\omega_{1} D_{2}a_{1}+ 2D_{1}a_{1} D_{1}\beta_{1}+
a_{1}D_{1}^{2}\beta_{1} \big)\\ 
\cos(\theta) \bigg ( 2 \omega_{1} a_{1} D_{2}\beta_{1} -D_{1}^{2}a_{1}+
a_{1}(D_{1}\beta_{1})^{2}
+\frac{5c^{2}\delta \phi_{1p}^{6}  a_{1}^{3}}{6\omega_{1}^{2}} -\frac{3d^{2}\delta \phi_{1p}^{8}
  a_{1}^{5}}{128\omega_{1}^{2}}
\bigg ) \\ 
+ S^{\sharp}_{3,1} - \epsilon R_{1}(r_{1}, \epsilon, y_{1}^{(1)}, y_{1}^{(2)}) 
\end{multline*}
where
%bbb  !!à revoir!!
\begin{multline}
\label{eq:s31sharp=}
S^{\sharp}_{3,1}=\frac{5 c d \delta \phi_{1p}^{7}
  a_{1}^{4}}{8\omega_{1}^{2}} + \sin 2\theta \left[ \frac{4 c\delta
    \phi_{1p}^{3} a_{1}}{3\omega_{1}}D_{1}a_{1} \right]
 + \cos 2\theta \left[ \frac{4 c  \delta \phi_{1p}^{3}a_{1}^{2}}{3\omega_{1}} D_{1}\beta_{1} +\frac{ 15c d \delta \phi_{1p}^{7}a_{1}^{4}}{32\omega_{1}^{2}}\right]\\
\quad \quad  + \sin 3\theta \left[ \frac{9 d\delta \phi_{1p}^{4}
    a_{1}^{2}}{16\omega_{1}}D_{1}a_{1} \right]   + \cos 3\theta \left[
  -\frac{ c^{2}\delta \phi_{1p}^{6} a_{1}^{3}}{6\omega_{1}^{2}}-\frac{3 d^{2}\delta
    \phi_{1p}^{8} a_{1}^{5}}{64\omega_{1}^{2}} +\frac{9 d a_1^3\delta \phi_{1p}^{4}}{16
    \omega}D_1 \beta_1 \right]\\ 
\quad \quad \quad \quad \quad + \cos 4 \theta \left[ -\frac{5c d\delta \phi_{1p}^{7} a_{1}^{4}}{32\omega_{1}^{2}}\right]-\frac{3 d^{2}\delta \phi_{1p}^{8} a_{1}^{5}}{128\omega_{1}^{2}}\cos 5 \theta  
\end{multline}
and
\begin{equation*}
S_{3,k}=  
\cos(\theta) \left (\frac{5c^{2}\delta \phi_{kp}\delta \phi_{1p}^{6} a_{1}^{3}}{6\omega_{1}^{2}} -\frac{3d^{2}\delta \phi_{kp}\delta \phi_{1p}^{8} a_{1}^{5}}{128\omega_{1}^{2}} \right) + S^{\sharp}_{3,k} - \epsilon R_{k}(r_{k}, \epsilon, y_{1}^{(1)}, y_{1}^{(2)}) 
\end{equation*}
where
\begin{multline*}
S^{\sharp}_{3,k}=\frac{5 c d\delta \phi_{kp} \delta \phi_{1p}^{6} a_{1}^{4}}{8\omega_{1}^{2}} + \sin 2\theta \left[ \frac{4 c\delta \phi_{kp}\delta \phi_{1p}^{2} a_{1}}{3\omega_{1}}D_{1}a_{1} \right]
 + \cos 2\theta \left[ \frac{4 c \delta \phi_{kp} \delta \phi_{1p}^{2}a_{1}^{2}}{3\omega_{1}} D_{1}\beta_{1} +\frac{ 15c d \delta \phi_{kp}\delta \phi_{1p}^{6}a_{1}^{4}}{32\omega_{1}^{2}}\right]\\
\quad \quad \quad + \sin 3\theta \left[ \frac{9 d\delta \phi_{kp}\delta \phi_{1p}^{3} a_{1}^{2}}{16\omega_{1}}D_{1}a_{1} \right] 
 + \cos 3\theta \left[  -\frac{ c^{2}\delta \phi_{kp}\delta
    \phi_{1p}^{5} a_{1}^{3}}{6\omega_{1}^{2}}-\frac{3 d^{2}\delta
    \phi_{kp}\delta \phi_{1p}^{7} a_{1}^{5}}{64\omega_{1}^{2}} 
+\frac{9d\delta\phi_{kp} \delta\phi_{11}^3a_1^3}{16\omega_1} D_1
\beta_1 \right]\\
\quad \quad \quad \quad \quad + \cos 4 \theta \left[ -\frac{5c d\delta \phi_{kp}\delta \phi_{1p}^{6} a_{1}^{4}}{32\omega_{1}^{2}}\right]-\frac{3 d^{2}\delta \phi_{kp}\delta \phi_{1p}^{7} a_{1}^{5}}{128\omega_{1}^{2}}\cos 5 \theta 
\end{multline*}
By imposing
\begin{align*}
\left\{
\begin{array}{rl}
&2\omega_{1} D_{2}a_{1}+ 2D_{1}a_{1} D_{1}\beta_{1}+
a_{1}D_{1}^{2}\beta_{1} =0 \\
& 2 \omega_{1} a_{1} D_{2}\beta_{1}   
 -D_{1}^{2}a_{1} +
a_{1}(D_{1}\beta_{1})^{2}+  \frac{5c^{2}\delta \phi_{1p}^{6}
  a_{1}^{3}}{6\omega_{1}^{2}} -\frac{3d^{2}\delta \phi_{1p}^{8} a_{1}^{5}}{128\omega_{1}^{2}}
=0 
\end{array}
\right .
\end{align*}
we get that $S_{3,1}=S^{\sharp}_{3,1} - \epsilon R_{1}(r_{1},
\epsilon, y_{1}^{(1)}, y_{1}^{(2)}$) contains no terms at the
frequency of the system.

 As $D_{1}a_{1}=0$ and $D_{1}\beta_{1}= \frac{-3 d \delta \phi_{1p}^{4}
   a_{1}^{2}}{8\omega_{1}}$, we obtain
 \begin{equation*}
 2 \omega_{1} a_{1} D_{2}\beta_{1}   +   a_{1} \Big (\frac{3 d \delta \phi_{1p}^{4}
   a_{1}^{2}}{8\omega_{1}} \Big)^{2}+  \frac{5c^{2}\delta \phi_{1p}^{6}
  a_{1}^{3}}{6\omega_{1}^{2}} -\frac{3d^{2}\delta \phi_{1p}^{8} a_{1}^{5}}{128\omega_{1}^{2}}
=0 
 \end{equation*}
 so:
 \begin{equation}
\label{eq:a2...beta2}	
 D_{2}a_{1}(T_{2})=0 ~~~~ \text{ and } ~~~~ D_{2}\beta_{1}(T_{2})=-\frac{5c^{2}\delta \phi_{1p}^{6} a_{1}^{2}}{12\omega_{1}^{3}} - \frac{15 d^{2}\delta \phi_{1p}^{8} a_{1}^{4}}{256\omega_{1}^{3}}
\end{equation}
 
As $a, \beta$ do not depend on $T_0$,  
\begin{align}
\left\{
\begin{array}{rl}
&  \frac{da_1}{dt}=\epsilon D_{1}a_{1}+\epsilon^{2}D_{2}a_{1}+\mathcal O(\epsilon^{3})\\
& \frac{d \beta}{dt}=\epsilon D_{1}\beta_{} +\epsilon^{2}D_{2}\beta_{} +\mathcal O(\epsilon^{3})
\end{array}
\right.
\end{align} 
and taking into account \eqref{eq:a1...beta1} and  \eqref{eq:a2...beta2},we obtain:\\
\begin{equation}
\label{eq:na1b1...b2}
 \frac{da_{1}}{dt}=0 ~~~~ \text{and } ~~~~
 \frac{d\beta_{1}}{dt}=\epsilon \frac{3 d \delta \phi_{1p}^{4}
   a_{1}^{2}}{8\omega_{1}} +\epsilon^{2} \Big(\frac{-5c^{2}\delta
   \phi_{1p}^{6} a_{1}^{2}}{12\omega_{1}^{3}} - \frac{15 d^{2}\delta
   \phi_{1p}^{8} a_{1}^{4}}{256\omega_{1}^{3}} \Big)
\end{equation}
As a result, the solution of these equations is:
\begin{equation}
\label{eq:nda1...db1}
 a_{1}= cte~~~~ \text{ and } ~~~~ \beta_{1}=\left[ \epsilon \frac{3 d
     \delta \phi_{1p}^{4} a_{1}^{2}}{8\omega_{1}}
   +\epsilon^{2} \Big(-\frac{5c^{2}\delta \phi_{1p}^{6}
     a_{1}^{2}}{12\omega_{1}^{3}} - \frac{15 d^{2}\delta \phi_{1p}^{8}
     a_{1}^{4}}{256\omega_{1}^{3}} \Big)\right] t
\end{equation}

In order to show that $r_{1}$ is bounded, after eliminating the
secular terms, we can go back to the variable $t$ in  the equation of  $r_{k}$, we get:
\begin{align*}
 &\frac{d ^{2} r_{1}}{d t^{2}} + \omega_{1}^2 r_{1}  = \tilde S^{}_{3,1}
\quad \text{ with } \quad
\tilde S^{}_{3,1}= S^{\sharp}_{3,1}(t, \epsilon) - \epsilon \tilde
R_{1}^{}(r_{1}, \epsilon, y_{1}^{(1)}, y_{1}^{(2)})\\
 &\frac{d ^{2} r_{k}}{d t^{2}} + \omega_{1}^2 r_{k}  = \tilde S^{}_{3,k}
\quad \text{ with } \quad
\tilde S^{}_{3,k}= S^{\sharp}_{3,k}(t, \epsilon) - \epsilon \tilde
R_{k}^{}(r_{k}, \epsilon, y_{k}^{(1)}, y_{k}^{(2)}) \quad k=2, \dots, n
\end{align*}
where $ S^{\sharp}_{3,1}$ is in \eqref{eq:s31sharp=} where all time
scales $T_0, T_1, T_2$ are expressed with the time variable $t$.
\begin{equation*}
 \tilde R_{1}^{}= R_{1}(\epsilon,r_{1},y_{1}^{(1)},y_{1}^{(2)})-\mathcal{D}_{3}r_{1}
\end{equation*}

After these manipulations, we can state a proposition which will be
easily proved with technical lemmas of the Appendix.

\begin{proposition}
\label{prop:systfree}
  We assume that $\omega_1^2$ is a simple eigenvalue and
  $\omega_k^2-9\omega_1^2 \neq 0, ~ \omega_k^2-4\omega_1^2 \neq 0$ (no
  internal resonance), then  it exists $\varsigma>0$ such that for all
  $t \le t_{\epsilon} =\frac{\varsigma}{\epsilon^{2}}$, the solution
  $\tilde y_k=\epsilon y_k$
  of  \eqref{eq:yktilde..Phipcd} with the initial data
\begin{align}
\begin{split}
&\tilde y_1(0)    =\epsilon a_{1} + \epsilon^{2}~(-\frac{\check c_1 a_1^2}{3
  \omega_1^2}~+~  \frac{\check d_1 a_1^3}{32 \omega_1^2})+ \epsilon^3 r_1(\epsilon,0), \quad \dot{\tilde y}_1(0)=\mathcal O(\epsilon)\\
&\tilde y_{k}(0)= \epsilon^{2}[-\frac{\check c_k a_1^2}{2 \omega_k^2}
+\frac{\check c_k a_1^2}{2(4\omega_1^2-\omega_k^2)} + \frac{\check d_k
  a_1^3}{4(9 \omega_1^2-\omega_k^2)}]+ \epsilon^3 r_k(\epsilon,0),
\quad \dot{\tilde  y}_k(0)=\mathcal O(\epsilon)
\end{split}
\end{align}
has the following expansion:
\begin{align}
\left\{
\begin{array}{rl}
& \tilde y_{1}(t)=\epsilon a_{1}
\cos(\nu_{\epsilon}t)+\epsilon^{2}[-\frac{\check c_1 a_1^2}{2
  \omega_1^2} + \frac{\check c_1 a_1^2}{6 \omega_1^2} \cos(2(\nu_{\epsilon}t))+  \frac{\check d_1 a_1^3}{32 \omega_1^2} \cos(3(\nu_{\epsilon}t))]+ \epsilon^{3}r_{1}(\epsilon,t) \\
& \tilde y_{k}(t)= \epsilon^{2}[-\frac{\check c_k a_1^2}{2 \omega_k^2}
+\frac{\check c_k a_1^2}{2(4\omega_1^2-\omega_k^2)}
\cos(2(\nu_{\epsilon}t))+ \frac{\check d_k a_1^3}{4(9 \omega_1^2-\omega_k^2)}\cos(3(\nu_{\epsilon}t))] + \epsilon^{3}r_{k}(\epsilon,t)
\end{array}
\right.
\end{align}
with $r_k$ uniformly bounded in $C^{2}(0,t_{\epsilon^{2}})$ for $k=1,\dots n$
and the angular frequency
\begin{equation}
\label{eq:omegan...libre}
 \nu_{\epsilon}= \omega_{1}+ \epsilon ( \frac{3 \check d_1
   a_{1}^{2}}{8\omega_{1}}) +\epsilon^{2} \left(\frac{-5 \check c^{2}_1
   a_{1}^{2}}{12\omega_{1}^{3}} - \frac{15 \check d^{2}_1
   a_{1}^{4}}{256\omega_{1}^{3}} \right) + \mathcal O(\epsilon^3)
\end{equation}
with $\delta \phi_{1p}=(\phi_{1,p}-\phi_{1,p-1}), \delta \phi_{kp}=(\phi_{k,p}-\phi_{k,p-1})$, 
$\check c_1=c(\delta
\phi_{1p})^3$, $\check d_1= d (\delta\phi_{1p})^4$
and
 $$\check c_k= c(\delta\phi_{1p})^2 \delta\phi_{kp}, ~ \check d_k=d
 (\delta\phi_{1p})^3 \delta\phi_{kp}$$

\end{proposition}
\begin{corollary}
  The solution of \eqref{eq:syst} with initial conditions 
\begin{align}
\begin{split}
& ^{t}\phi_1 \tilde u(0)   =\epsilon a_{1} + \epsilon^{2}~(-\frac{\check c a_1^2}{3
  \omega_1^2}~+~  \frac{\check d a_1^3}{32 \omega_1^2})+ \epsilon^3
r_1(\epsilon,0), \quad ^{t} \phi_1 \dot u(0)=\mathcal O(\epsilon^2)\\
& ^{t}\phi_k \tilde u(0)= \epsilon^{2}[-\frac{\check c_k a_1^2}{2 \omega_k^2}
+\frac{\check c_k a_1^2}{2(4\omega_1^2-\omega_k^2)} + \frac{\check d_k
  a_1^3}{4(9 \omega_1^2-\omega_k^2)}]+ \epsilon^3 r_k(\epsilon,0),
\quad  ^{t}\phi_k  \dot u(0)=\mathcal O(\epsilon^2)
\end{split}
\end{align}

\begin{equation}
\text{ is }  \tilde u(t)=\sum_{k=1}^n \tilde y_k(t) \phi_k+ \epsilon^{3}r(t,\epsilon)
\end{equation}
with the expansion of $y_k$ of previous proposition.
\end{corollary}
\begin{proof}
 For the proposition, we use lemma \ref{eq:lemmew-syst}; set
 $S_1=\tilde S_{3,1}^{}, ~~
  S_k=S_{3,k} $ for $k=1, \dots n$; as we have enforced  \eqref{eq:nda1...db1},
  the functions $S_k$ are  periodic, bounded, and are orthogonal to  $e^{\pm it}$,  we have
  assumed  that $\omega_k$ and $\omega_1$ are $\mathbb{Z}$ independent for  $k \ne 1$; then $S$ satisfies
 satisfies the lemma hypothesis. Similarly, set $g=\tilde R$, its components are polynomials in $r$ with coefficients which are bounded functions, so it is lipschitzian on the bounded subsets  it satisfies the hypothesis of the lemma   and so the proposition is proved.
The corollary is an easy consequence of the proposition and the change of function \eqref{eq:uyphi}
\end{proof}

\begin{rem}
\begin{enumerate}
\item
  We have obtained a periodic asymptotic expansion of a solution of
  system \eqref{eq:syst}; they are called non linear normal
  modes in the mechanical community (\cite{nnm-kpgv,
    jiang-pierre-shaw04}. If the initial condition is close to an eigenvector $\Phi_1$ up to second order, the component of the solution on this eigenvector has an approximation which has the same form as for the single degree of freedom system; the other components remain small.
\item The frequency shift is given by a similar formula with $c$
  replaced by $\check c= c (\phi_{1,p} -\phi_{1,p-1})^3$, $d$ replaced by $\check d=  d (\phi_{1,p}-\phi_{1,p-1})^4$; so the frequency shift depends on the position of non-linearity with respect to  the components of the associated eigenvector.
\item In the spirit of inverse problems, this previous  point opens a way to localise the non-linearity.

\item
{\it We do not study the periodicity of  the solution  itself  but as
  the system is Hamiltonian, it could be obtained from general
  results, for example see \cite{zbMATH00052214}.
}
\item
In the next section, under the assumption of no internal resonance, we shall derive that
  the frequencies of the normal mode are close to resonant frequencies for an
  associated forced system, the so called primary resonance; with some changes, secondary
  resonance could be derived along similar lines.
\end{enumerate}
\end{rem}

\subsubsection{Numerical results}
We consider numerical solution of \eqref{eq:syst} with \eqref{eq:Phi=};
we have chosen $M=I$; $u=0$ at both ends, so  $K$ is the classical
matrix
$$k
\begin{pmatrix}
  2 & -1 & \hdotsfor[]{3}\\
-1 & 2 & -1 & \hdotsfor{2} \\
0&-1 &2&-1 & \dots \\
\hdotsfor[]{5} \\
 \hdotsfor[]{3} &-1 &2
\end{pmatrix}
;$$
$C=\lambda I$ with $\lambda=1/2$; for numerical balance, we have
computed $\frac{u}{\epsilon}$; with the choice $p=1$ we have
$\Phi_1=\epsilon [c u_1^2 + d u_1^3]$ with $c=1, d=1$.
In figure \ref{fig:phasefftns29}, for 29 degrees of freedom, we find the Fourier transform of
the components; some components have the same transform; the graphs
are slightly non symmetric;
we find also several curves in phase space
for some  components of the system. 

\begin{figure}
\label{fig:phasefftns29}
\includegraphics[height=8cm,width=15cm]{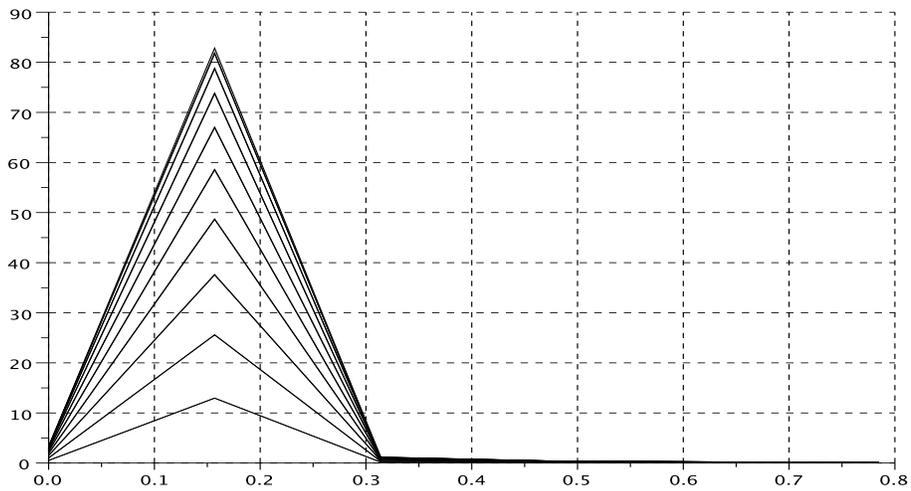}
\includegraphics[height=10cm,width=15cm]{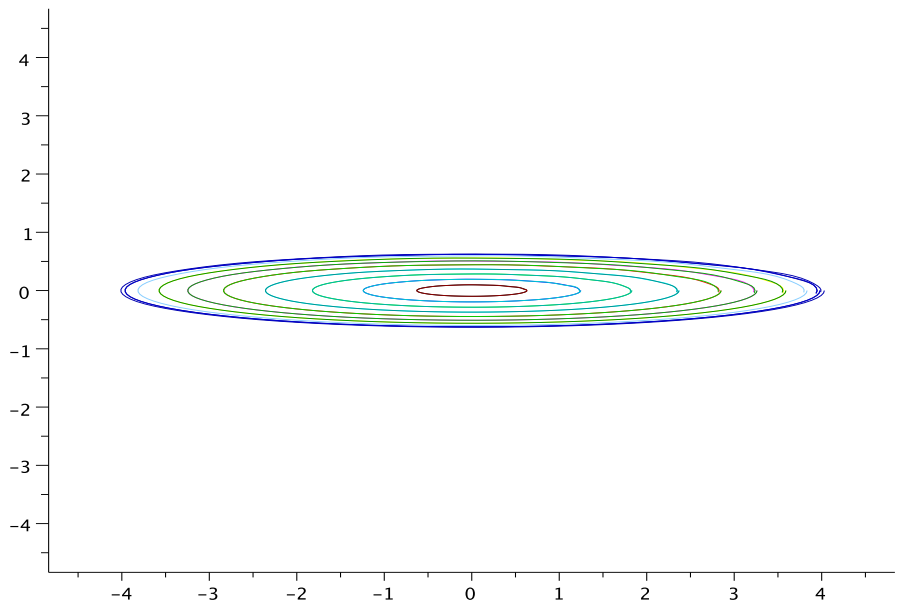}
\caption{ Absolute value of the Fourier transform for (fft) (left); phase portrait(right)}
\end{figure}

We remark that up to numerical integration errors, all frequencies are
equal and the components are periodic. All these characteristics are
coherent with the results obtained by asymptotic expansions: an
approximation of a non linear normal mode which is a continuation with
respect to $\epsilon$ of a linear normal mode.

\subsection{Forced, damped vibrations, triple scale expansion}

\subsubsection{Derivation of an asymptotic  expansion}

We consider a similar system of forced  vibrating masses attached to
springs with some  damping and submitted to a periodic forcing:
% nnnn 
\begin{equation}
\label{eq:systforcamort}
 M \ddot{\tilde u} + \epsilon C \dot{\tilde u} + K u + \Phi(\tilde u,\epsilon) = \epsilon^{2} F\cos \tilde{\omega_{\epsilon}} t
\end{equation} 
with the same assumptions as in subsection \ref{subsec:syst-free-vib}.
We assume that the non linearity is local, all components are zero except for two components $p-1, \;p$ which correspond to the endpoints of some spring assumed to be non linear.
As for free vibrations, we perform  the change of function 
\begin{equation}
  \label{eq:u=yphi}
\tilde u=\sum_{k=1}^n \tilde y_k \phi_k   
\end{equation}
with $\phi_k$, the generalised eigenvectors of
\eqref{eq:vect-val-propre}.
However, the distribution of damping is almost always unknown and it
is usually necessary to make an assumption about its distribution;  a
simple and widely used  hypothesis is to choose a modal damping (
hypothesis of Basile in french terminology):
\begin{equation*}
 C= \epsilon_{M} M + \epsilon_{K} K
\end{equation*}
Therefore
\begin{equation*}
  \ddot{\tilde y}_{k} + \epsilon \lambda_{k} \dot{\tilde y}_{k} +
  \omega_{k}^{2} \tilde y_{k}  + ^{t}\phi_{k}\Phi(\sum_{i=1}^{n}
  \tilde y_{i}\phi_{i},\epsilon) = \epsilon^{2} f_{k} \cos \tilde{\omega_{\epsilon}} T_{0},~~~~~ k=1 \dots, n\\
\end{equation*} 
 with $$\epsilon_{M} + \epsilon_{K} \omega_{k}^{2}=\lambda_{k} ~~~and ~~ ^{t}\phi_{k} F = f_{k} $$ 
As for the free vibration case,  $\Phi$ has only 2 components which
are not zero, so the system can be written:
\begin{multline}
\label{eq:ykamorttilde..Phipcd}
\ddot{\tilde y}_k + \epsilon \lambda_{k} \dot{\tilde y}_{k}
+\omega_k^2 \tilde y_k+\left (
  \phi_{k,p-1}-\phi_{k,p}\right ) \Bigg[c \left (\sum_{i=1}^n
  \tilde y_i
  (\phi_{i,p}-\phi_{i,p-1}) \right)^2 + \frac{d}{\epsilon} \left
  (\sum_{i=1}^n \tilde y_i( \phi_i-\phi_{i,p-1}) \right)^3 \Bigg]\\
=\epsilon^{2}  f_{k}\cos \tilde{\omega_{\epsilon}} T_{0}, \quad   {\text for}~~~ k=1 \dots, n
\end{multline}
\begin{rem}
  As we intend to look for a small solution, we consider a change of
  function \fbox{$\tilde y_k=\epsilon y_k$} and we obtain the transformed equation:
\end{rem}
\begin{multline}
\label{eq:ykamort..Phipcd}
\ddot y_k + \epsilon \lambda_{k} \dot y_{k} +\omega_k^2 y_k+\left (
  \phi_{k,p-1}-\phi_{k,p}\right ) \Bigg[\epsilon c \left (\sum_{i=1}^n y_i
  (\phi_{i,p}-\phi_{i,p-1}) \right)^2 + \epsilon d \left
  (\sum_{i=1}^n y_i( \phi_i-\phi_{i,p-1}) \right)^3 \Bigg]\\
=\epsilon^{}  f_{k}\cos( \tilde{\omega}_{\epsilon}t) \quad   {\text for}~~~ k=1 \dots, n
\end{multline}

We will highlight a link between the frequency of the free solution of
the preceding paragraph and the amplitude of the steady state forced
solution; it is assumed that the excitation frequency is close to the natural frequency of the linear system
\begin{equation}
 \tilde{\omega_{\epsilon}}=\omega_{1} + \epsilon \sigma 
\end{equation}

As in the previous case, we look for a small solution with a  triple scale expansion, more precisely, we look for  \textbf{a periodic solution close to an eigenmode of the linear system},
 for example, we consider mode $y_{1}$ (by permuting the indexes it
 could be any mode); we set:
\begin{equation*}
  T_0= \tilde \omega_{\epsilon} t, \quad T_1=\epsilon t , \quad T_2=\epsilon^{2} t 
\text{ hence } ~~  D_{0}y_{k}= \frac{\partial y_{k}}{\partial T_{0}}
,~~  D_{1}y_{k}= \frac{\partial y_{k}}{\partial T_{1}}  \text{ and } 
D_{2}y_{k}= \frac{\partial y_{k}}{\partial T_{2}}
\end{equation*}
\par{ Derivatives of $y_{k}$ may be expanded:}
\begin{equation}
\label{eq:doty_{k}...ramort}
 \frac{d y_{k}}{dt}=\tilde \omega_{\epsilon} D_{0}y_{k} + \epsilon D_{1}y_{k} + \epsilon^{2} D_{2}y_{k}\\
\end{equation}
 and
\begin{equation}
\label{eq:ddoty_{k}...ramort}
 \frac{d^{2} y_{k}}{dt^{2}}= \tilde \omega_{\epsilon}^2
 D_{0}^{2}y_{k}+ 2 \epsilon \tilde \omega_{\epsilon} D_{0}D_{1}y_{k}+2\epsilon^{2}D_{0}D_{2}y_{k}+\epsilon^{2}D_{1}^{2}y_{k}+2\epsilon^{3}D_{1}D_{2}y_{k}+\epsilon^{4}D_{2}^{2}y_{k}
\end{equation} 
we use the {\it ansatz}
\begin{equation}
\label{eq:y_{k}=alpha...ramort}
y_{k}(t)= y_{k}(T_0, T_{1},T_{2})= y_{k}^{(1)}(T_0, T_{1},T_{2})+\epsilon^{}y_{k}^{(2)}(T_0, T_{1},T_{2})+\epsilon ^{2} r_{k}( T_{0},T_{1},T_{2})
\end{equation} 
we get:
\begin{align*}
\frac{d y_{k}}{dt}&= \frac{d y_{k}^{(1)}}{dt}+ \epsilon^{}\frac{d y_{k}^{(2)}}{dt}+ \epsilon^{2}\frac{dr_{k}}{dt}
  =\epsilon \frac{d y_{k}^{(1)}}{dt}+ \epsilon^{2}\frac{dy_{k}^{(2)}}{dt}+\epsilon^{2}D_{0}r_{k}+ \epsilon^{2}(\frac{dr_{k}}{dt}- D_{0}r_{k})\\
  &= [ \tilde \omega_{\epsilon} D_0y_{k}^{(1)} +\epsilon D_1
  y_{k}^{(1)} + \epsilon^{2}D_{2}y_{k}^{(1)}] + \epsilon [ \tilde
  \omega_{\epsilon} D_0y_{k}^{(2)} +\epsilon D_1 y_{k}^{(2)} + \epsilon^{2}D_{2}y_{k}^{(2)}]\\
& \qquad \qquad \qquad + \epsilon^{2}\tilde \omega_{\epsilon} D_0r_{k} +\epsilon^{2}(\frac{dr_k}{d
t}- \tilde \omega_{\epsilon} D_{0}r_{k})
\end{align*}
we note that $\frac{dr_{k}}{dt}- \tilde \omega_{\epsilon} D_{0}r_{k}=
\epsilon D_{1}r_{k}+ \epsilon^{2}D_{2}r_{k}$; it is of order $1$ in $\epsilon$.
For the second derivative, as in the case of free vibration, we introduce:
\begin{align*}
 \mathcal{D}_3 r_{k}&= \frac{1}{\epsilon} (\frac{d^{2}r_{k}}{dt^{2}}- \tilde \omega_{\epsilon}^2 D_{0}^{2}r_{k})\\
&=2 \tilde \omega_{\epsilon} D_0D_1 r_{k}+ \epsilon\left[ 2\tilde
  \omega_{\epsilon} D_{0}D_{2}r_{k}+ D_1^2r_{k} +2 D_2D_1r_{k} \right]+ \epsilon^{3}D_{2}^{2}r_{k}
\end{align*}

\begin{align*}
 \frac{d^2y_{k}}{dt^2}&=\frac{d^2y_{k}^{(1)}}{dt^2}+\epsilon \frac{d^2y_{k}^{(2)}}{dt^2}+ \epsilon^{2}\frac{d^2r_{k}}{dt^2}
=\frac{d^2y_{k}^{(1)}}{dt^2}+\epsilon^{1}\frac{d^2y_{k}^{(2)}}{dt^2}+\epsilon^2
\tilde \omega_{\epsilon} D_0^2r_{k} +\epsilon^3 {\cal D}_3 r_{k}\\ 
 &= \tilde \omega_{\epsilon}^2 D_{0}^{2}y_{k}^{(1)}+
 \epsilon^{}\left[ 2 \tilde \omega_{\epsilon} D_{0}D_{1}y_{k}^{(1)} + D_{0}^{2}y_{k}^{(2)}\right]\\
  & \qquad \qquad +\epsilon^{2}\left[ 2 \tilde \omega_{\epsilon}
    D_{0}D_{2}y_{k}^{(1)} + D_{1}^{2}y_{k}^{(1)}+2 \tilde
    \omega_{\epsilon} D_{0}D_{1}y_{k}^{(2)}+ D_{0}^{2}r_{k}\right]\\
 &\qquad \qquad \qquad +\epsilon^{3}\left[2D_{1}D_{2}y_{k}^{(1)}+2 \tilde \omega_{\epsilon} D_{0}D_{2}y_{k}^{(2)}+D_{1}^{2}y_{k}^{(2)}   +  \mathcal{D}_3 r_{k}  \right]\\
&\qquad \qquad \qquad \qquad +\epsilon^{4}\left[  D_{2}^{2}y_{k}^{(1)}+ 2D_{1}D_{2}y_{k}^{(2)}+ \epsilon D_{2}^{2}y_{k}^{(2)}\right] 
\end{align*}

We plug previous expansions \eqref{eq:doty_{k}...ramort}, \eqref{eq:y_{k}=alpha...ramort} and \eqref{eq:ddoty_{k}...ramort} of $y^{k}$  into \eqref{eq:ykamort..Phipcd}; by identifying  the coefficients of the powers of $\epsilon$, we get:

\begin{align}
\left\{
\begin{array}{rll}
& \omega_1^2 D_{0}^{2}y_{k}^{(1)}+\omega_{k}^{2}y_{k}^{(1)}=0~~~~~,~~~~~ &k=1 \dots, n \\ 
&  \omega_1^2 D_{0}^{2}y_{k}^{(2)}+\omega_{k}^{2}y_{k}^{(2)}= S_{2,k}~~~~~~,~~~~~&k=1 \dots, n \label{eq:nddlamort-alpha}\\ 
& \omega_1^2 D_{0}^{2} r+\omega_{k}^{2} r = S_{3,k}~~~~~,~~~~~ &k=1
\dots, n 
%\label{eq:nddlamort-alpha3}
\end{array}
\right.
\end{align}
with
\begin{multline*}
S_{2,k}= -c \delta \phi_{kp}\left( \sum_{l,m} y_{l}^{(1)}\delta \phi_{lp}
  y_{m}^{(1)}\delta \phi_{mp}\right) - d \delta \phi_{kp}\left(
  \sum_{g,n,o}y_{g}^{(1)}\delta \phi_{gp}y_{n}^{(1)}\delta \phi_{np}y_{o}^{(1)}\delta \phi_{op}\right)
\\- 2 \omega_1 D_{0}D_{1} y_{k}^{(1)} - \lambda_{k} \omega_1 D_{0}y_{k}^{(1)}
-2\omega\sigma D_0^2y_k^{(1)}+  f_{k} \cos(T_{0}),
\end{multline*}
\begin{multline*}
 S_{3,k}= - c \delta \phi_{kp}\left( \sum_{l,j}
   y_{l}^{(1)}y_{j}^{(2)}\delta \phi_{lp} \delta \phi_{jp}\right) -  d
 \delta \phi_{kp}\left( \sum_{h,g,n}
   y_{h}^{(1)}y_{g}^{(1)}y_{n}^{(2)}\delta \phi_{hp} \delta \phi_{gp}
   \delta \phi_{np}\right)\\- 2 \omega_1 D_{0}D_{2}y_{k}^{(1)}  - D_{1}^{2}
 y_{k}^{(1)} - 2 \omega_1 D_{0}D_{1}y_{k}^{(2)} -\sigma^2 D_0^2 y_k^{(1)}  -2\omega_1 \sigma
 D_0^2 y_k^{(1)}-2\sigma D_0D_1y_k^{(1) } -2\omega_1 \sigma D_0^2y_k^{(2)}\\ 
- \lambda_{k} D_{1}y_{k}^{(1)}- -\lambda_k \sigma D_0y_k^{(1)}- \lambda_{k} \omega_1 D_{0}y_{k}^{(2)}-\epsilon
 R_{k}(\epsilon,r_{k},y_{1}^{(1)},y_{1}^{(2)})
\end{multline*}
where  $\delta \phi_{kp}=(\phi_{k,p}-\phi_{k,p-1})$ and with
%bbb%erreurs corrigées!!!
\begin{multline*}
 R_{k}(\epsilon,r_{k},y_{k}^{(1)},y_{k}^{(2)})=2D_{1}D_{2}y_{k}^{(1)}+2\omega_1
 D_{0}D_{2}y_{k}^{(2)}+D_{1}^{2}y_{k}^{(2)}  \\
+c \delta \phi_{kp}\left( \sum_{l,j}
     y_{j}^{(2)}y_{j}^{(2)}\delta \phi_{lp} \delta \phi_{jp}\right)  
+c \delta \phi_{kp}\left( \sum_{l,j}
     y_{j}^{(1)}r_{l}\delta \phi_{jp} \delta \phi_{lp}\right)  \\
+d \delta \phi_{kp}\left( \sum_{h,g,n}
     y_{h}^{(1)}y_{g}^{(2)}y_{n}^{(2)}\delta \phi_{hp} \delta \phi_{gp}
     \delta \phi_{np}\right)
+d \delta \phi_{kp}\left( \sum_{h,g,n}
     y_{h}^{(1)}y_{g}^{(1)}r_{n}\delta \phi_{hp} \delta \phi_{gp}
     \delta \phi_{np}\right) \\
   \lambda_{k}(\omega_1 D_0r +
 D_{2}y_{k}^{(1)}+ D_{1}y_{k}^{(2)} +\epsilon D_{2}y_{k}^{(2)}) +\mathcal{D}_3r\\
+\epsilon \left ( D_{2}^{2}y_{k}^{(1)}+ 2 D_{1}D_{2}y_{k}^{(2)}+ \epsilon D_{2}^{2}y_{k}^{(2)} \right)
 +\lambda_{k}(\frac{dr}{dt}-\omega_1 D_0r) + \epsilon \rho(y_{k}^{(1)},y_{k}^{(2)},r_{k},\epsilon) 
\end{multline*}
and the polynomial $\rho$ displayed in \eqref{eq:rho=}.
%bbbb

We solve the first set of equations \eqref{eq:nddlamort-alpha} imposing  initial Cauchy data for $k
 \neq 1$  of order $\mathcal O(\epsilon^2)$ and $D_0y_1^{(1)}(0)=0$
we get:
 \begin{equation}
\left\{ 
\begin{array}{rl}
& y_{1}^{(1)}= a_{1}(T_1,T_{2}) ~\cos(\theta )\\
& y_{k}^{(1)}= 0 ,~~~~ k=2, \dots, n
\end{array}
\right.
\end{equation} 
with 
$ \theta(T_0,T_1,T_{2})= T_0+\beta(T_1,T_{2})$ for which we have $D_0 \theta=1, \; D_1 \theta=D_1 \beta_{1};$
we put terms involving $y_k^1, \; k \ge 2$ into $R_k$; so we obtain:

\begin{multline*}
 S_{2,1}= -\delta \phi_{1p} \Big[\frac{ca_1^2}{2}(1+\cos(2\theta))\delta \phi_{1p}^2+\frac{da_1^3}{4} (\cos(3\theta)+3\cos(\theta))\delta \phi_{1p}^3 \Big] \\ + 2\omega_1(D_1 a_1 \sin(\theta)+ a_1(D_1 \beta_1+\sigma) \cos(\theta))+ \lambda_{1}a_{1}\omega_{1}\sin(\theta) \\ + f_{1} (\cos(\theta)\cos(\beta_{1})+ \sin(\theta)\sin(\beta_{1})) 
\end{multline*}
\begin{multline*}
 S_{2,k}= -\delta \phi_{kp}\Big[\frac{ca_1^2}{2}(1+\cos(2\theta))\delta
 \phi_{1p}^2+\frac{da_1^3}{4}(\cos(3\theta)+3\cos(\theta)) \delta
 \phi_{1p}^{3}\Big] \\ +  f_{k} (\cos(\theta)\cos(\beta_{1})+
 \sin(\theta)\sin(\beta_{1}))  , \quad k= 1, \dots, n.
\end{multline*}

We will eliminate the terms at angular frequency $\omega_1$ hence  the
functions $a_{1}(T_1,T_{2}) ~$and$~ \beta_{1}(T_1,T_{2})$ satisfy:
 \begin{align*}
\label{eq:=d1a1...nddl}
\left\{
\begin{array}{rl}
 & 2 \omega_{1}D_{1}a_{1}+ \lambda_{1} a_{1} \omega_{1}= -f_{1} \sin( \beta_{1})\\
 & 2 \omega_{1}a_{1} D_{1}\beta_{1} +2\omega_1 a \sigma-\frac{3 d \delta \phi_{1p}^{4}a_{1}^{3}}{4} = - f_{1}\cos(\beta_{1}) 
 \end{array}
\right.
\end{align*}

and the solution of the second equation of\eqref{eq:nddlamort-alpha} is:
%%%bbb signes??ok
\begin{align}
\left\{
\begin{array}{rl}
&y_{1}^{(2)}=\delta \phi_{1p} \Big[(\frac{-ca_1^2}{2\omega_{1}^{2}}+\frac{ca_1^2}{6\omega_{1}^{2}} \cos(2\theta))\delta \phi_{1p}^2+\frac{da_1^3}{32\omega_{1}^{2}} \cos(3\theta)\delta \phi_{1p}^3 \Big]\\
& y_{k}^{(2)}= \delta \phi_{kp}
\Big[-\frac{ca_1^2}{2(\omega_{k}^{2}-\omega_{1}^{2})}+\frac{ca_1^2}{2(4\omega_{1}^{2}
  -2\omega_{k}^{2})} \cos(2\theta))\delta \phi_{1p}^2+\frac{da_1^3}{4(9\omega_{1}^{2}-\omega_{k}^{2})} \cos(3\theta)\delta \phi_{1p}^3 \Big]\\
\end{array}
\right.
\end{align}
where we have omitted the term at frequency $\omega_{1}$ which is redundant with $y_{1}^{(1)}$

For the third equation of \eqref{eq:nddlamort-alpha}, the unknown is
$r_{k}$;  we do not  solve it but we show that the solution is bounded
on an interval dependent on $\epsilon$. After some manipulations, the right hand side is:
\begin{multline*}
 S_{3,1} = 
+ \sin \theta \left[ 2
  \omega_{1} D_{2}a_{1}+ \lambda_{1} a_{1} D_{1}\beta_{1} +  2 D_{1}a_{1}D_{1}\beta_{1} +
  a_{1}D_{1}^{2}\beta_{1}  +2\sigma D_1a_1+\lambda_1 a_1 \sigma \right] \\
+\cos\theta \left[  2 \omega_{1}a_{1}D_{2}\beta_{1} - \lambda_{1}
  D_{1}a_{1}- D_{1}^{2}a_{1}+ a_{1} (D_{1}\beta_{1})^{2}+\sigma^2 a_1
  +2 \sigma a_1 D_1 \beta_1+
\frac{5 c^{2} \delta \phi_{1p}^{6}a_{1}^{3}}{6\omega_{1}} -\frac{3 \delta \phi_{1p}^{8}d^{2} a_{1}^{5}}{128\omega_{1}}\right]   \\
 + S_{3}^{\sharp} - \epsilon R(\epsilon,r,u^{(1)},u^{(2)}) 
\end{multline*}
where
%revoir!!
\begin{align*}
S_{3,1}^{\sharp}&= \frac{5 c d \delta \phi_{1p}^{7}
  a_{1}^{4}}{8\omega_{1}^{2}} + \sin 2\theta \left[ \frac{4 c\delta
    \phi_{1p}^{3} a_{1}}{3\omega_{1}}D_{1}a_{1} +
  \frac{\lambda_{1}c\delta \phi_{1p}^{3} a_{1}^{2}}{3\omega_{1}}
  \right]\\
 &\quad+ \cos 2\theta \left[\frac{4 c  \delta \phi_{1p}^{3}a_{1}^{2}}{3\omega_{1}}  D_{1}\beta_{1}
  +\frac{15 c d\delta \phi_{1p}^{7}a_{1}^{4}}{32\omega_{1}^{2}}\right]\\
&\quad \quad + \sin 3\theta \left[ \frac{9 d\delta \phi_{1p}^{4} a_{1}^{2}}{16\omega_{1}}D_{1}a_{1} 
+ \frac{3 \lambda_{1} d \delta \phi_{1p}^{4}  a_{1}^{3}}{16\omega_{1}} \right]
 + \cos 3\theta \left[  \frac{9 d\delta \phi_{1p}^{4} a_{1}^{3}}{16\omega_{1}}D_{1}\beta_{1}-\frac{ c^{2}\delta \phi_{1p}^{6} a_{1}^{3}}{6\omega_{1}^{2}}-\frac{3 d^{2}\delta \phi_{1p}^{8} a_{1}^{4}}{64\omega_{1}^{2}} \right] \\
&\qquad \qquad \qquad \qquad + \cos 4 \theta \left[ -\frac{3 c d\delta \phi_{1p}^{7} a_{1}^{4}}{8\omega_{1}^{2}}\right]-\cos 5 \theta  \frac{3 d^{2}\delta \phi_{1p}^{8} a_{1}^{5}}{128\omega_{1}^{2}}
\end{align*}
and a similar expression for $S_{3,k}^{\sharp}$.
To eliminate the secular terms, we impose,
\begin{align*}
%\label{eq:d2...theta-syst}
\left\{
\begin{array}{rl}
& 2 \omega_{1} D_{2}a_{1}+ \lambda_{1} a_{1} D_{1}\beta_{1} +2
D_{1}a_{1}D_{1}\beta_{1} + a_{1}D_{1}^{2}\beta_{1} +2\sigma D_1
a_1+\lambda_1 a_1 \sigma= 0\\
& 2 \omega_{1}a_{1}D_{2}\beta_{1} - \lambda_{1} D_{1}a_{1}-
D_{1}^{2}a_{1}+ a_{1} (D_{1}\beta_{1})^{2}+\sigma^2a_1 +2\sigma a_1
D_1 \beta_1+
\frac{5 c^{2} \delta \phi_{1p}^{6}a_{1}^{3}}{6\omega_{1}^2} -\frac{3 \delta \phi_{1p}^{8}d^{2} a_{1}^{5}}{128\omega_{1}^2}= 0
\end{array}
\right.
\end{align*}

As  $a_{1}$ and $\beta_{1}$ do not depend on $T_0$, the following
relations hold:
 
\begin{equation}
\left\{
\begin{aligned}
\label{eq:na=alpha...D1D1}
  \frac{da_{1}}{dt}=\epsilon D_{1}a_{1}+\epsilon^{2}D_{2}a_{1}+\mathcal O(\epsilon^{3})\\
 \frac{d\beta_{1}}{dt}=\epsilon D_{1}\beta_{1} +\epsilon^{2}D_{2}\beta_{1} +\mathcal O(\epsilon^{3})
\end{aligned}
\right.
\end{equation}
On the other hand,  we can determine the expression of $ D_{2}a_{1}~~$
and $~~D_{2}\beta$, like  for one degree of freedom: 
\begin{align}
\label{eq:nD2a1...D2gamma}
\left\{
\begin{array}{rl}
& D_{2}a_{1} = \frac{3 d \lambda_{1} \delta \phi_{1p}^{4} a_{1}^{3}}{16 \omega_{1}^{2}} + \frac{\sigma f_1 \sin\gamma}{4 \omega_{1}^{2}}+ \frac{\lambda_{1} f_1 \cos\gamma}{8\omega_{1}^{2}}+\frac{9d\delta \phi_{1p}^{4}a_{1}^{2}f_1\sin\gamma}{32 \omega_{1}^{3}}\\
& D_{2}\gamma =-\frac{\lambda_{1}^{2}}{8\omega_{1}}- \frac{15 d^{2}\delta \phi_{1p}^{8} a_{1}^{4}}{256 \omega_{1}^{3}}-\frac{5c^{2}\delta \phi_{1p}^{6}a_{1}^{2}}{12\omega_{1}^{3}}\\ 
&\qquad\qquad \qquad\qquad  \qquad\qquad +\frac{\sigma f_1 \cos\gamma}{4\omega_{1}^{2}a_{1}}+\frac{3 d \delta \phi_{1p}^{4}a_{1}f_1 \cos\gamma}{32 \omega_{1}^{3}}-\frac{\lambda_{1} f_1 \sin\gamma}{8\omega_{1}^{2}a_{1}}
\end{array}
\right.
\end{align}

now we return to \eqref{eq:na=alpha...D1D1} introducing \eqref{eq:=d1a1...nddl}
%\eqref{eq:d1...gammanddl} 
and \eqref{eq:nD2a1...D2gamma}, we obtain:
%% bbb vérifier
\begin{align}
\begin{split}
\label{eq:dta1...dtgammanddl}
& \frac{da_{1}}{dt} = \epsilon \Big (-\frac{f_1 \sin(\beta)}{2\omega_{1}}+
\frac{\lambda_{1} a_{1}}{2} \Big)+ \epsilon^{2} \left(\frac{3 d
    \lambda_{1} \delta \phi_{1p}^{4} a_{1}^{3}}{16 \omega_{1}^{2}} +
  \frac{\sigma f_1 \sin\beta}{4 \omega_{1}^{2}}+ \frac{\lambda_{1} f_1
    \cos\beta}{8\omega_{1}^{2}}+\frac{9d\delta
    \phi_{1p}^{4}a_{1}^{2}f_1\sin\beta}{32 \omega_{1}^{3}} \right)+O(\epsilon^{3})\\ 
 & \frac{d\beta}{dt}= \epsilon \Big(-\sigma + \frac{3d  \delta
   \phi_{1p}^{4}a_{1}^{2}}{8\omega_{1}}-
\frac{f_1\cos(\beta)}{2 \omega_{1}a_{1}} \Big) \\
&\qquad + \epsilon^{2} \Big ( -\frac{\lambda_{1}^{2}}{8\omega_{1}}- \frac{15
    d^{2}\delta \phi_{1p}^{8} a_{1}^{4}}{256 \omega_{1}^{3}}- 
\frac{5 c^{2}\delta \phi_{1p}^{6}a_{1}^{2}}{12\omega_{1}^{3}}
   +\frac{\sigma f_1
   \cos\beta}{4\omega_{1}^{2}a_{1}}+\frac{3 d \delta
   \phi_{1p}^{4}a_{1}f_1 \cos\beta}{32\omega_{1}^{3}}
-\frac{\lambda_{1} f_1 \sin\beta}{8\omega_{1}^{2}a_{1}} \Big) +O(\epsilon^{3})
\end{split}
\end{align}

\begin{rem}
  In this approach, like for one free degree of freedom, we are using the method of reconstitution. We notice these equations are similar to \eqref{eq:dta1...dtgamma}
\end{rem}
\begin{rem}
$ S_{3}^{\sharp}+ R(\epsilon,r, u^{(1)},u^{(2)})$ has no term at frequency  $\omega_{1}$ or which goes to $\omega_{1}$ .\\
 This will allow us to justify this expansion in certain conditions, before we consider the stationary solution of the system \eqref{eq:dta1...dtgammanddl} and the stability of the solution close to the stationary solution.
\end{rem}
 
\subsubsection{Stationary solution and stability}

Let us consider the stationary solution of \eqref{eq:dta1...dtgammanddl}, it satisfies:

\begin{align}
\left \{
  \begin{array}[h]{rl}
 & g_1(a_1,\beta_1,\sigma,\epsilon)=0, \\ & g_2(a_1,\beta_1,\sigma,\epsilon)=0
  \end{array}
\right .
\end{align}

with
\begin{align}
\label{eq:ndta1...dtgamma=0}
\left\{
\begin{array}{rl}
&g_1=  \epsilon (-\frac{f_1 \sin(\beta)}{2\omega_{1}}+ \frac{\lambda_{1} a_{1}}{2} ) +\\
& \qquad  \qquad  \qquad \epsilon^{2}(\frac{3 d \lambda_{1} \delta
  \phi_{1p}^{4} a_{1}^{3}}{16 \omega_{1}^{2}} + \frac{\sigma f_1
  \sin\beta}{4 \omega_{1}^{2}}+ \frac{\lambda_{1} f_1
  \cos\beta}{8\omega_{1}^{2}}+\frac{9d\delta
  \phi_{1p}^{4}a_{1}^{2}f_1\sin\beta}{32 \omega_{1}^{3}})+ \mathcal O(\epsilon^{3})\\ 
&
g_2= \epsilon (-\sigma +\frac{3d  \delta
  \phi_{1p}^{4}a_{1}^{2}}{8\omega_{1}}-\frac{f_1\cos(\beta)}{2
  \omega_{1}a_{1}}) \\ &  \qquad  \qquad + \epsilon^{2}( -\frac{\lambda_{1}^{2}}{8\omega_{1}}- \frac{15 d^{2}\delta \phi_{1p}^{8} a_{1}^{4}}{256 \omega_{1}^{3}}- \frac{5
 c^{2}\delta \phi_{1p}^{6}a_{1}^{2}}{12\omega_{1}^{3}}
 +\frac{\sigma f_1 \cos\beta}{4\omega_{1}^{2}a_{1}}-\frac{3 d \delta
   \phi_{1p}^{4}a_{1}f_1 \cos\beta}{32
   \omega_{1}^{3}}-\frac{\lambda_{1} f_1
   \sin\beta}{8\omega_{1}^{2}a_{1}})+ \mathcal O(\epsilon^{3})
\end{array}
\right.
\end{align}

The situation is very close to the 1 d.o.f. case; except the
replacement of $c$ by $\check c= c \delta \phi_{1p}^{3}$ and $d$ by of $\check{d}=d\delta \phi_{1p}^{4}$, the system \eqref{eq:ndta1...dtgamma=0} is the same as \eqref{eq:dta1...dtgamma=0}; the other components are zero. We state a similar proposition.

\begin{proposition}
 When  $$\sigma \le \frac{3 \check{d} \bar a_1^2}{4\omega_{1}}
 -\frac{1}{2}\sqrt{\frac{9 \check{d}^{2}\bar a_1^4}{16
    \omega_{1}^{2}}-\lambda_{1}^2}$$
 and $\epsilon$ small enough, the stationary solution $(\bar a_1,\bar \beta_1)$ of \eqref{eq:dta1...dtgammanddl} is stable in the sense of Lyapunov (if the dynamic solution starts close to the stationary one, it remains close and converges to it); to the stationary case corresponds the approximate solution of \eqref{eq:ykamort..Phipcd} 
\begin{equation}
\tilde y_{1 \, app}=\epsilon \bar a_{1} \cos(\tilde \omega_{\epsilon} t+ \bar \beta )+ \epsilon^{2}\Big(
\delta \phi_{1p} \Big [ \Big(\frac{-c\bar a_1^2}{ \omega_{1}^{2}}+\frac{c\bar
  a_1^2}{6\omega_{1}^{2}} \cos2(\tilde \omega_{\epsilon} t+ \bar \beta
) \Big)\delta \phi_{1p}^2+\frac{d\bar a_1^3}{32\omega_{1}^{2}}
\cos(3(\tilde \omega_{\epsilon} t+ \bar \beta ))\delta \phi_{1p}^3
\Big ]\Big);
\end{equation}
\begin{multline}
 \tilde y_{k \, app}= \epsilon^{2} \Big(  \delta \phi_{kp}
    \Big[ \Big(\frac{-c\bar
      a_{1}^2}{2(\tilde \omega_{k}^{2}+\omega_{1}^{2})}-\frac{c\bar
      a_{1}^2}{2(\omega_{k}^{2}-4\omega_{1}^{2})} \cos 2(\tilde \omega_{\epsilon} t+ \bar
    \beta ) \Big)\delta \phi_{1p}^2 \\
-\frac{d\bar a_{1}^3}{4(\tilde \omega_{k}^{2}-9\omega_{1}^{2})} \cos (3(\tilde \omega_{\epsilon}
t+ \bar \beta ))\delta \phi_{1p}^3 \Big] \Big)  
\end{multline}
it is periodic. 
\end{proposition}

With this result of stability, we can state precisely the approximation of the solution of \eqref{eq:systforcamort}

\subsubsection{Convergence of the expansion}
 In order to prove that $r_{k}$ is bounded, after  eliminating terms at frequency $\nu_{1}$, we go back to the variable $t$ for the third set of equations of \eqref{eq:nddlamort-alpha} 
.
\begin{align*}
&\frac{d ^{2} r_{k}}{d t^{2}}+  \omega_{1}^2 r_{k}=
\tilde{S}_{3,k}\quad \text{ for } k=1, \dots n \quad
 \text{ with}\\
& \tilde{S}_{3,1}=S_{3,1}^{\sharp}(t,\epsilon) -  \epsilon  \tilde
R_{1}(y_{1}^{(1)}, y_{1}^{(2)},r_{1},\epsilon) \quad
\text{ and} \text{ for } k \neq 1\\
% !! à revoir !!
& S_{3,k}= -2 c \delta \phi_{kp}[y_{1}^{(1)}y_{1}^{(2)}\delta \phi_{1p}^{2}] -3 d \delta \phi_{kp}[y_{1}^{(1)2}y_{1}^{(2)}\delta \phi_{1p}^{3}]-\epsilon R_{k}(\epsilon,r_{k},y_{1}^{(1)},y_{1}^{(2)})  
\end{align*}
where
\begin{equation*}
\tilde  R_{k}(\epsilon,r_{k},y_{1}^{(1)},y_{1}^{(2)})= R_{k}(\epsilon,r_{k},y_{1}^{(1)},y_{1}^{(2)})-{\cal D}_2 r_k - \lambda_k \bigl (\frac{d r_k}{d t}-\omega_k D_0 r_k \bigr )
\end{equation*}
with all the terms expressed with the variable $t$.

\begin{proposition}
  Under the assumption that $\omega_k^2 \neq 4 \omega_1^2, ~
  \omega_k^2 \neq 9\omega_1^2$ and $\omega_1^2$ a simple eigenvalue 
  (no internal resonance) for $k \ne 1$, there exists $\varsigma>0$
  such that for all $t \le t_{\epsilon}
  =\frac{\varsigma}{\epsilon^{2}}$, the solution $\tilde y =\epsilon
  y$ of  \eqref{eq:ykamorttilde..Phipcd} with initial data
 \begin{align*}
&\tilde y_1(0)=\epsilon a_{1} + \epsilon^{2}~ \big(\frac{-\check c_1
  a_{10}^2}{2\omega_{1}^2}~+~\frac{\check c_1  a_{10}^2}{6\omega_1^2}\cos(2 \beta_{0})~+
\frac{\check d_1 a_{10}^3}{32 \omega_1^2} \cos(3\beta_{0}) \big)+
\epsilon^3 r(0,\epsilon), \\
&\tilde y_{k}(0)= \epsilon^{2} \left(\frac{-\check c_1a_{10}^2}{2(\omega_k^2-\omega_{1}^{2})}+\frac{\check c_1a_{10}^2}{2(4\omega_{1}^{2}-\omega_k^2)}
\cos(2\beta_{0}) + \frac{\check d_1a_{10}^3}{4(9\omega_{1}^{2}-\omega_k^2)}
\cos(3(\beta_{0}) ) \right)+  \epsilon^3 r(0,\epsilon), 
\end{align*}
with  similar expressions for  $\dot y_1(0), \dot y_k(0)$
and with ($a_{10}, \beta_{0}$) close to the stationary solution ($\bar a_1, \bar \beta$)

$$|a_{10} -\bar a_1 | \le \epsilon^{2} C^{1}, ~~ |\beta_{0}- \bar \beta|
\le \epsilon^{2} C^{1}$$
 
has  the following expansion
 \begin{multline*}
\tilde  y_{1} =\epsilon  a_{1} \cos(\tilde \omega_{\epsilon} t+\beta(t))+\epsilon^{2}[ 
((\frac{-\check c_1a_1^2}{2\omega_{1}^{2}}+\frac{\check c_1a_1^2}{6\omega_{1}^{2}}
\cos(2(\tilde \omega_{\epsilon} t+\beta(t)))
+\frac{\check d_1a_1^3}{32\omega_{1}^{2}}
\cos(3(\tilde \omega_{\epsilon} t+\beta(t))) )] +\epsilon^3 r_1(t)
\end{multline*}

\begin{multline*}
\tilde y_{k}= \epsilon^{2}\Big(   \Big[(\frac{-\check c_k
  a_1^2}{2(\omega_{k}^{2}-\omega_{1}^{2})}+\frac{\check c_k a_1^2}{2(4\omega_{1}^{2}-\omega_{k}^{2})} \cos(2(\tilde \omega_{\epsilon} t+\beta(t)))
 +\frac{\check d_ka_1^3}{4(9\omega_{1}^{2}-\omega_{k}^{2})}
 \cos(3(\tilde \omega_{\epsilon} t+\beta(t)))\Big]\Big)  +\epsilon^3 r_k(t)
\end{multline*}
with $a_1, \beta$ solution of  \eqref{eq:dta1...dtgammanddl} and
with $r_k$ uniformly bounded in ${\cal C}^{2}(0,t_{\epsilon})$ for $k=1,\dots n$
and $\omega_1, \phi_1$ are the eigenvalue and eigenvectors defined in
\eqref{eq:vect-val-propre}, 
with $\delta \phi_{1p}=(\phi_{1,p}-\phi_{1,p-1}), \delta \phi_{kp}=(\phi_{k,p}-\phi_{k,p-1})$, 
$\check c_1=c(\delta
\phi_{1p})^3$, $\check d_1= d (\delta\phi_{1p})^4$
and
 $\check c_k= c(\delta\phi_{1p})^2 \delta\phi_{kp}, ~ \check d_k=d
 (\delta\phi_{1p})^3 \delta\phi_{kp}$
as in proposition \ref{prop:systfree}.
\end{proposition}
\begin{corollary}
  The solution of \eqref{eq:systforcamort} with 
 \begin{align*}
&\phi_1^T \tilde u(0)=\epsilon a_{1} + \epsilon^2
\big(\frac{-\check c_1
  a_{10}^2}{2\omega_{10}^2}~+~\frac{-\check c_1  a_{10}^2}{6\omega_1^2}\cos(2 \gamma_{0})~+
\frac{\check d_1 a_{10}^3}{32 \omega_1^2} \cos(3\gamma_{0}) \big)
+ \epsilon^{3}r_{1}(0,\epsilon),\\
&\phi_k^T \tilde u(0)=  \epsilon^{2} \left(\frac{-\check c_1a_{10}^2}{2(\omega_k^2-\omega_{1}^{2})}+\frac{\check c_1a_{10}^2}{2(4\omega_{1}^{2}-\omega_k^2)}
\cos(2\gamma_{0}) + \frac{\check d_1a_{10}^3}{4(9\omega_{1}^{2}-\omega_k^2)}
\cos(3(\gamma_{0}) ) \right)
+ \epsilon^{3}r_{k}(0,\epsilon), 
\end{align*}
 with similar expressions for $  \phi_1^T \dot{ \tilde u}(0),  \phi_k^T  \dot{\tilde
 u}(0)$ and
with $\omega_k, \phi_k$  the eigenvalues and eigenvectors defined in \eqref{eq:vect-val-propre}.
\begin{equation}
\text{ is }  \tilde u(t)=\sum_{k=1}^n \tilde y_k(t) \phi_k
\end{equation}
with the expansion of $y_k$ of previous proposition.
\end{corollary}
\begin{proof}
 We follow a similar route as for one degree of freedom, we use  lemma   \ref{eq:lemmew-syst}. Set $S_1=S_{31}^{\sharp}, ~~ S_k=S_{3,k} $ for $k=1, \dots n$; as we have enforced  \eqref{eq:ndta1...dtgamma=0}, the functions $S_k$ are not periodic but close to a periodic function, bounded and are orthogonal to $e^{\pm it}$, we have assumed  that $\omega_k$ and $\omega_1$ are $\mathbb{Z}$ independent for $k \ne 1$; so  $S$ satisfies the lemma hypothesis. Similarly, set $g=\tilde R$, it is a polynomial in $r$ with coefficients which are bounded functions , so it is lipschitzian on the bounded subsets of $\mathbb R$, it satisfies the hypothesis of  lemma   \ref{eq:lemmew-syst} and so the proposition is proved.
The corollary is an easy consequence of the proposition and the change of function \eqref{eq:u=yphi}
\end{proof}

\subsubsection{Maximum of the  stationary solution}
We can state results similar to the case of one degree of freedom.
\begin{proposition}
  The stationary solution of \eqref{eq:dta1...dtgammanddl}
satisfies 
\begin{align}
\left\{
\begin{array}{rl}
&  (-\frac{f_1 \sin(\beta)}{2\omega_{1}}+ \frac{\lambda_{1}a_{1}}{2} ) +    \qquad \epsilon A_1(a_1,\beta,\sigma)  + \mathcal O(\epsilon^{2})=0\\
&  (-\sigma+\frac{3d \delta \phi_{1p}^4 a_{1}^2}{8\omega_{1}}-\frac{f_1
  \cos(\beta)}{2
  a_{1}\omega_{1}})+ \qquad \epsilon A_2(a_1,\beta,\sigma) + \mathcal O(\epsilon^{2})=0
\end{array}
\right.
\end{align}
with 
\begin{align*}
 & A_1(a,\beta,\sigma)=  \frac{3 d \delta \phi_{1p}^4 \lambda_{1} a_{1}^{3}}{16
  \omega_{1}^{2}} + \frac{\sigma f_1 \sin\beta}{4 \omega_{1}^{2}}+
\frac{\lambda f_1 \cos\beta}{8\omega_{1}^{2}}+
\frac{9d \delta \phi_{1p}^4 a_{1}^{2} f_1 \sin\beta}{32 \omega_{1}^{3}} \\
&A_2(a,\beta,\sigma)= -\frac{\lambda_{1}^{2}}{8\omega_{1}}- \frac{15 d^{2}\delta \phi_{1p}^8 a_{1}^{4}}{256 \omega_{1}^{3}}- \frac{5
 c^{2}\delta \phi_{1p}^6 a_{1}^{2}}{12\omega_{1}^{3}} \nonumber \\ 
& \qquad  \qquad  \qquad \qquad  \qquad+\frac{\sigma f_1 \cos\beta}{4\omega_{1}^{2}a_{1}}+ \frac{3 d \delta \phi_{1p}^4 a_{1}f_1 \cos\beta}{32 \omega_{1}^{3}}-\frac{\lambda_{1} f_1 \sin\beta}{8\omega_{1}^{2}a_{1}}
\end{align*}
this stationary solution reaches its maximum amplitude for $\sigma=\sigma_0^*+
\epsilon \sigma_1^* +O(\epsilon^2 )$
with 
\begin{gather}
   a_{1,0}^{\ast}=\frac{ f_1}{\lambda_{1} \omega_{1}}, \; \sigma^{\ast}_0
=\frac{3 \check d  a^{\ast 2}_{1,0}}{8 \omega_{1}}=\frac{3 \check d  f_1^2}{8 \lambda_{1}^2 \omega_{1}^3} ,
\quad \beta_0^*=-\frac{\pi}{2}
\end{gather}
and
 
$$\sigma_1^*=- \frac{87 \check
  d^{2} a_{1,0}^{* 4}}{256 \omega_{1}^{3}}- \frac{5 \check
  c^{2}a_{1,0}^{* 2}}{12\omega_{1}^{3}} - \frac{\lambda_{1}^2}{4 \omega_{1}}, \quad
\beta_1^*=-\frac{\lambda_{1}}{2 \omega_1},
 \quad a_{1,1}^*=-\frac{ a_{1,0}^* \sigma_0^*}{ \omega_{1}}$$
 the periodic forcing is at the angular frequency 
$$\tilde \omega_{\epsilon}= \omega_{1}+\epsilon \sigma_0^*
+\epsilon^2 \sigma_1^* + \mathcal O(\epsilon^2)
$$
up to the term involving the damping ratio $\lambda_{1}$, it is
slightly different of the approximate
angular frequency $\nu_{\epsilon}$ of  the undamped free periodic
solution \eqref{eq:omegan...libre}; for this frequency, the approximation (of the  solution $\tilde y =\epsilon
  y$ of  \eqref{eq:ykamorttilde..Phipcd} up to the
order $\epsilon^2$)  is periodic:
\begin{align}
\left\{
\begin{array}{rl}
&\tilde y_1(t)=\epsilon \bar a_{1}^* \cos( \tilde \omega_{\epsilon} t +
\bar \beta^*)+ \epsilon^{2} [(
\frac{-\check c_1\bar{a}_{1}^{*2}}{2\omega_{1}^{2}}+\frac{\check c_1\bar{a}_{1}^{*2}}
{6\omega_{1}^{2}}\cos(2 (\tilde \omega_{\epsilon} t + \bar \beta^*)))
 \\
&\qquad \qquad\qquad \qquad\qquad \qquad\qquad \qquad \qquad
+\frac{\check d_1\bar{a}_{1}^{*3}}{32\omega_{1}^{2}}\cos(3(\tilde \omega_{\epsilon} t + \bar \beta^*) ]
+ \epsilon^3 r_1(\epsilon,t)\\
&\tilde y_k(t)= \epsilon^{2}[(
\frac{-\check
  c_k\bar{a}_{1}^{*2}}{2(\omega_{k}^{2}-\omega_{1}^2)}-\frac{\check c_k\bar{a}_{1}^{*2}}
{2(\omega_{k}^{2}-4\omega_{1}^2)}\cos(2 (\tilde \omega_{\epsilon} t + \bar \beta^*)))
\\
&\qquad \qquad\qquad \qquad\qquad \qquad\qquad \qquad \qquad
-\frac{\check d_k \bar{a}_{1}^{*3}}{4(\omega_{k}^{2}-9\omega_{1}^2)}\cos(3(\tilde \omega_{\epsilon} t + \bar \beta^*) ] 
+ \epsilon^3 r_k(\epsilon,t)
\end{array}
\right.
\end{align}
and initial conditions  like in proposition \ref{prop:systfree}.
% \begin{align}
% \begin{split}
% y_1(0)=\epsilon \bar{a}_{1}^* + \epsilon^{2}(C_2+ B)+ \epsilon^3 r_1(\epsilon,0)\\
% y_k(0)=\epsilon^{2}(C_2k+ B_k)+ \epsilon^3 r_k(\epsilon,0)
% \end{split}
% \end{align}
\end{proposition}

\section{Conclusion}
For some differential systems modelling spring-masses vibrations with non linear springs,  we have derived and rigorously proved an asymptotic approximation of
periodic solution of free vibrations (so called non linear normal
modes); for damped vibrations with
periodic forcing with frequency close (but different) to free vibration frequency ( the
so called primary resonance case), we have obtained an asymptotic
expansion and derived that the amplitude is maximal close to the frequency of
the non linear normal mode.

We emphasize that the use of three time scales provides a more
accurate value of the link between frequency and amplitude (so called backbone) of a non
linear mode but it yields also a new insight
in the behavior of the solution which was not provided by a
double-scale analysis: the influence of the ratio of $c$ over $d$ on
the  shape of the backbone and the amplitude of the forced response to
an harmonic force as
is clearly displayed in figure  \ref{fig:ampligamm-freq1ddlc1} and \ref{fig:ampligamm-freq1ddld025c6}.

As an opening to a related problem, we can notice that
such non linear vibrating systems linked to a bar
generate acoustic waves; an analysis of the dilatation of a one-dimensional nonlinear crack impacted by a periodic elastic wave, with
a smooth model of the crack  may be carried over with a delay
differential equation, \cite{junca-lombard09}.

\paragraph{Acknowledgment}
%We thank the referees for their valuable suggestions.
We thank S. Junca  for his stimulating interest.

\section{Appendix}

\subsection{Technical lemmas}
\label{sec:techlem}

All these lemmas are recalled here for convenience of the reader; they already have been proposed in \cite{nbb-br-doubl}.
\label{sec:appendix}

\begin{lemma}
\label{eq:lemmew }
  Let   $w_{\epsilon}$  be solution of
  \begin{align}
\begin{split}
\label{eq:w"+w=}
    w"+w=S(t,\epsilon)+\epsilon g(t,w,\epsilon) \\
w(0)=0, \quad w'(0)=0
\end{split}
  \end{align}
If the right hand side satisfies the following  conditions
\begin{enumerate}
\item $S$  is a sum of periodic bounded functions:
  \begin{enumerate}
  \item 
for all $t$ and for all $\epsilon$ small enough, $S(t,\epsilon) \le M$
\item $\int_0^{2\pi}e^{i t}S(t,\epsilon)dt=0, \quad \int_0^{2\pi}e^{-i
    t}S(t,\epsilon)dt=0$  uniformly for  $\epsilon$ small enough
  \end{enumerate}
\item for all $R>0$, 
there exists $k_R$ such that for  $|u|\le R$ and $|v|\le R$, the
inequality  $|g(t,u,\epsilon)-g(t,v,\epsilon)|\le k_R |u-v|$ holds and
$|g(t,0,\epsilon)|$is bounded; in other words $g$ is locally
lipschitzian with respect to u.
%%il exists $M_R < +\infty$, such that  $$\sup\{|g(t,w,\epsilon)|, t>0, \;|u|<R, \; \epsilon>0, \text{ assez petit} \} \le M_R $$
\end{enumerate}
then, there exists $\gamma>0$ such that for $\epsilon$ small enough,
$w_{\epsilon}$ is uniformly bounded in $C^{2}(0,T_{\epsilon})$ with $T_{\epsilon}=\frac{\gamma}{\epsilon}$
\end{lemma}
\begin{proof}
The proof  is close  to the proof of  lemma 6.3 of
\cite{junca-br10}; but it is technically simpler since here we assume
$g$ to be locally lipschitzian with respect to $u$ whereas it is only
bounded in \cite{junca-br10}.
  \begin{enumerate}
  \item We first  consider
  \begin{align}
\begin{split}
    w_1"+w_1=S(t,\epsilon) \\
w_1(0)=0, \quad w_1'(0)=0
\end{split}
  \end{align}
as  $S$ is  a sum of periodic functions which are  uniformly
orthogonal to $e^{i t}$ and $e^{-i t}$, $w_1$ is bounded in  ${\cal C}^{2}(0,+\infty)$
\item Then we perform a change  of function: $w=w_1+w_2$, the following
  equalities hold
 \begin{align}
\begin{split}
    w_2"+w_2= \epsilon g_2(t,w_2,\epsilon)\\
w_2(0)=0, \quad w_2'(0)=0
\end{split}
  \end{align}
with $g_2$ which satisfies the same hypothesis as $g$:

for all  $R>0$,
there exists $k_R$ such that for $|u|\le R$ and $|v|\le R$, the
following inequality holds $|g_2(t,u,\epsilon)-g_2(t,v,\epsilon)|\le k_R |u-v|$.
Using Duhamel principle, the solution of this equation satisfies:
\begin{equation*}
  w_2=\epsilon \int_0^t\sin(t-s) g_2(s,w_2(s),\epsilon)ds
\end{equation*}
from which
\begin{equation*}
  |w_2(t)|\le \epsilon  \int_0^t |g_2(s,w_2(s),\epsilon)- g_2(s,0,\epsilon)|ds +
 \epsilon \int_0^t |g_2(s,0,\epsilon)| ds
\end{equation*}
so if $|w| \le R$, hypothesis of lemma imply
\begin{equation*}
   |w_2(t)|\le \epsilon \int_0^t  k_R|w_2| ds +\epsilon C t
\end{equation*}
A corollary of lemma of Bellman-Gronwall, see below, will enable to conclude.
It yields
\begin{equation*}
  |w_2(t)| \le \frac{ C}{k_R}\left (   \exp(\epsilon k_Rt) -1 \right)
\end{equation*}
Now set $T_{\epsilon}=\sup \{t| |w|\le R \}$, then we have
$$R  \le \frac{ C}{k_R}\left (   \exp(\epsilon k_Rt) -1 \right)$$
this shows that there exists $\gamma$ such that $|w_2| \le R$ for $t \le
T_{\epsilon}$, which means that it is  in $L^{\infty}(0,T_{\epsilon})$
for $T_{\epsilon}=\frac{\gamma}{\epsilon}$;  also, we have $w$ in
${\cal C}(0,T_{\epsilon})$ then as $w$ is solution of
\eqref{eq:w"+w=}, it is also bounded in 
  ${\cal C}^{2}(0,T_{\epsilon})$.
  \end{enumerate}
\end{proof}
\begin{lemma}
  (Bellman-Gronwall, \cite{bell-gronw,bellman-perturb}) Let  $u,\epsilon,\beta$ be continuous functions with $\beta \ge 0$, 
  \begin{equation*}
    u(t) \le \epsilon(t) + \int_0^t\beta(s)u(s) ds \text{ for } 0\le t\le T
  \end{equation*}
then 
\begin{equation*}
   u(t) \le \epsilon(t) + \int_0^t \beta(s) \epsilon(s) \left[ \exp(\int_s^t \beta(\tau)d\tau \right ] ds
\end{equation*}
\end{lemma}

\begin{lemma} ( a consequence of previous lemma, suited for
  expansions, see \cite{sanders-verhulst})
  Let $u$ be a positive function, $\delta_2 \ge 0$, $\delta_1 >0$ and
$$ u(t) \le \delta_2 t + \delta_1\int_0^tu(s)ds$$ then 
$$u(t) \le \frac{\delta_2}{\delta_1}\left ( exp(\delta_1 t) -1
\right) $$
\end{lemma}

\begin{lemma}
\label{eq:lemmew-syst}
  Let $v_{\epsilon}=[v_1^{\epsilon}, \dots, v_N^{\epsilon}]^T$ be the solution of the following system:
\begin{equation}
\omega^2 (v_k^{\epsilon})"+\omega_k^2v_k^{\epsilon}=S_k(t)+\epsilon g_k(t,v_{\epsilon})
\end{equation}
If   $\omega$ and $\omega_k$ are $\mathbb{Z}$ independent for all
$k=2 \dots N$ and 
the right hand side satisfies the following conditions with $M>0, \; C>0$ prescribed constants:
\begin{enumerate}
\item $S_k$ is a sum of bounded periodic functions, $|S_k(t)|\le M$ which satisfy the  non resonance conditions:

  \item $S_1$ is orthogonal to $e^{\pm it}$,
    i.e. $\int_0^{2\pi}S_1(t)e^{\pm it}dt=0$ uniformly for $\epsilon$
    going to zero

\item for all $R>0$ there exists $k_R$ such that for $\|u\| \le R$, $\|v\|\le R$, the following inequality holds for $k=1, \dots, N$ :
$$|g_k(t,u,\epsilon)-g_k(t,v,\epsilon)|\le k_R \|u-v\|$$ and $|g_k(t,0,\epsilon)|$ is bounded
\end{enumerate}
then there exists $\gamma>0$ such that for $\epsilon$ small enough
$v_{\epsilon}$ is bounded in ${\cal C}^{2}(0,T_{\epsilon})$ with $T_{\epsilon}=\frac{\gamma}{\epsilon}$
\end{lemma}
\begin{proof}
% { \bf no longer to do }\\
 \begin{enumerate}
 \item 
We first  consider the linear system
\begin{gather}
\begin{split}
\label{eq:vk1}
 \omega_1^2 (v_{k,1})"+\omega_k^2v_{k,1}=S_k  \\
 v_{k,1}(0)=0  \text{ and }  (v_{k,1})'=0
\end{split}
\end{gather}

 For $k=1$, with hypothesis 1.a,  $S_1$ is a sum of bounded periodic
 functions; it is orthogonal to  $e^{\pm it}$, there is no resonance.
For $k\neq 1$, there is no  resonance as
$\frac{\omega_{k}}{\omega_{1}} \notin \mathbb{Z}$ with  hypothesis
1.b.

So  $v_{k,1}$ belongs to $C^{(2)}$ for $k=1,..., n$
\item Then we perform a change of function
$$v_{k}^{\epsilon}=v_{k,1}+ v_{k,2}^{\epsilon}$$
and  $v_{k,2}^{\epsilon}$  are solutions of the following system :

\begin{gather}
\begin{split}
\label{eq:vk2}
  \omega_1^2 (v_{k,2})"+\omega_k^2v_{k,2}=\epsilon
  g_{k,2}(t,v_{k,2},\epsilon), ~ k=1, \dots, N\\
  v_{k,2}^{\epsilon}(0)=0, ~ (v_{k,2}^{\epsilon})'=0, ~ k=1, \dots, N
\end{split}
\end{gather}
%???avec $$v^{\epsilon}= v_{1}+v_{2}^{\epsilon} , v_{2}^{\epsilon}=(...,v_{k,2}^{\epsilon},...) et v_{1}=(....,v_{k,1},....)$$
with $$ g_{k,2}(t,....,v_{k,2}^{\epsilon},....)=g_{k}(t,...,v_{k,1}+v_{k,2}^{\epsilon},....)$$
where  $~ g_{k,2} ~$  satisfies the same  hypothesis as $g_{k}$:\\
for all $R>0$ there exists $k_R$ such that for  $ \parallel u_{k} \parallel \le R$, $ \parallel v_{k} \parallel \le R$, the following inequality holds for $k=1, \dots, N$ :
\begin{equation}
\label{eq:gk2}
 \parallel g_{k,2}(t,u_{k},\epsilon)-g_{k,2}(t,v_{k},\epsilon)\parallel \le k_R  \parallel u_{k}-v_{k} \parallel
\end{equation}

Using Duhamel principle, the 
 solution or the equation \eqref{eq:vk2} satisfies:
\begin{equation*}
  v_{k,2}^{\epsilon}=\epsilon \int_0^t\sin(t-s) g_{k,2}(s,v_{k,2}^{\epsilon}(s),\epsilon)ds
\end{equation*}
so 
\begin{multline*}
  \parallel  v_{k,2}^{\epsilon}(t)  \parallel  \le \epsilon
  \int_0^t  \parallel g_{k,2}(s,v_{k,2}^{\epsilon}(s),\epsilon)-
  g_{k,2}(s,0,\epsilon) \parallel  ds + \\
 \epsilon \int_0^t \parallel  g_{k,2}(s,0,\epsilon) \parallel  ds
\end{multline*}
so with \eqref{eq:gk2}, we obtain
\begin{equation*}
 \parallel  v_{k,2}^{\epsilon}(t) \parallel \le \epsilon \int_0^t  k \parallel v_{k,2}^{\epsilon}(t)  \parallel ds +\epsilon C t
\end{equation*}
We shall conclude using  Bellman-Gronwall
 lemma; we obtain

\begin{equation*}
\parallel  v_{k,2}(t) \parallel \le \frac{C}{k_R}(exp(\epsilon k_R t)-1)
\end{equation*}

this shows that there exists $\gamma$ such that $|v_{k,2}^{\epsilon}| \le R$ for $t \le
T_{\epsilon}$, which means that it is  in $L^{\infty}(0,T_{\epsilon})$
for $T_{\epsilon}=\frac{\gamma}{\epsilon}$;  also, we have $v_k$ in
${\cal C}(0,T_{\epsilon})$ then as $v_k$ is solution of
\eqref{eq:w"+w=}, it is also bounded in 
  ${\cal C}^{2}(0,T_{\epsilon})$.

 \end{enumerate}

\end{proof}

\begin{theorem}
\label{th:poinc-lyapu}
( of Poincar\'e-Lyapunov, for  example see \cite{sanders-verhulst})
Consider the equation 
$$\dot x=(A+B(t))x +g(t,x), \; x(t_0)=x_0, \; t\ge t_0$$
where $x,x_0 \in \mathbf{R}^n  $,  $A$ is a constant matrix $n\times
n$  with all its eigenvalues with negative real parts; $B(t)$ is a
matrix which is  continuous  with the property $\lim_{t \rightarrow +\infty} \|B(t)\|=0$.
The vector field  is continuous with respect to  $t$ and $x$ is
continuously differentiable with respect to $x$ in a neighbourhood of $x=0$; moreover 
$$g(t,x)= o(\|x\|) \text{ when } \; \|x\|\rightarrow 0 $$ uniformly in $t$.
Then, there exists  constants $C,t_0,\delta,\mu$ such that if $\|x_0\|<\frac{\delta}{C}$ 
$$\|x\| \le C\|x_0\|e^{-\mu(t-t_0)}, t \ge t_0$$ holds
\end{theorem}

\subsection{Numerical computations of Fourier transform }
Assuming a function $f$ to be almost-periodic, the Fourier
coefficients are :
\begin{equation}
  \label{eq:four-ap}
  \alpha_n=\lim_{T \rightarrow +\infty} \int_0^T f(t) e^{-i \lambda_n t} dt
\end{equation}
where $\lambda_n$ are countable Fourier exponents of $f$.
(for example, see Fourier coefficients of an almost-periodic function in http://www.encyclopediaofmath.org/).
For numerical purposes, we chose $T$ large enough and
%rrr
with a fast Fourier transform, we compute  numerically
%rrrr
the Fourier coefficients of a function of period $T$ equal to $f$ in this interval.

\subsection{Another way of computing the maximum amplitude}
\label{sec:anotherway}
This is another way of computing some results of $\S$ \ref{subsubsec:maxstatsol}.
Eliminating $\beta$ at first order in \eqref{eq:g1+A_g2+A=0} , we get that $a$ is solution of $f(a,\beta,\sigma,\epsilon)=0$ with
$$f=\frac{-F_m^2}{4 \omega^2}+(-\frac{\lambda a}{2}+\epsilon A_{1})^2+ 
\left (\frac{3da^{3}}{8\omega_{}}-\sigma a + \epsilon a A_{2}
\right )^2 + \mathcal O(\epsilon^2)$$

 We look for $a_{}$ maximum
with respect to $\sigma$; it will be reached at a value denoted
$\sigma^*$ which depends on $\epsilon$.
By differentiating, we get that
$$\frac{\partial a}{\partial \sigma}=-\frac{\frac{\partial f}{\partial
    \sigma}+\frac{\partial f}{\partial \beta}\frac{\partial \beta}{\partial \sigma}
}{\frac{\partial f}{\partial a}}$$
%nnnn en derivant je pense qu'il y a un signe (-) manquant bbb oui
So $\sigma^*$ is solution of 
\begin{equation}
  \label{eq:dfs+dfg=0}
  \frac{\partial f}{\partial \sigma} +\frac{\partial f}{\partial
    \beta}\frac{\partial \beta}{\partial \sigma}=0 ~\text{ with }~
  \frac{\partial f}{\partial a} \neq 0;
\end{equation}
we compute the terms involved in the previous equation;
$$ \frac{\partial f}{\partial \sigma}=2\epsilon (-\frac{\lambda
  a^{*}}{2}+\epsilon A_{1}) \frac{\partial  A_{1}}{\partial \sigma}+ 
2\left (\frac{3da_{}^{\ast 3}}{8\omega_{}}-\sigma^{*}
  a^{*} + \epsilon a^{*} A_{2}
\right ) \left (-a^{*} + \epsilon  a^{*} \frac{\partial
  A_{2}}{\partial \sigma} \right) + \mathcal O(\epsilon^2)
$$
or
\begin{equation}
\label{eq:dfds=}
  \frac{\partial f}{\partial \sigma}= -2\epsilon (\frac{\lambda
    a^{*}}{2}) \frac{\partial A_{1}}{\partial \sigma}
-2a^{*}  (\frac{3da^{\ast 3}}{8\omega_{}}-\sigma^{*}
  a^{*}) - 2\epsilon a^{\ast 2} A_{2} + 2\epsilon(\frac{3da^{\ast 3}}{8\omega_{}}-\sigma^{*}
  a^{*})a^{*}\frac{\partial A_2}{\partial \sigma}
 + \mathcal O(\epsilon^2)
\end{equation}
we simplify for $a=a_0^*+\mathcal O(\epsilon), \; \sigma=\sigma_0^*+\mathcal O(\epsilon)$
\begin{multline*}
  \frac{\partial f}{\partial \sigma}=     -2a_{0}^{*}\left (\frac{3da_{0}^{* 3}}{8\omega_{}}-\sigma_0^* a_0^{*} \right )
 -2\epsilon a_{1}^{*}\left (\frac{3da_{0}^{* 3}}{8\omega_{}}-\sigma_0^* a_0^{*} \right )
-2\epsilon a_0^{*}\left(\frac{9da_0^{* 2}a_1^{*}}{8\omega}-a_0^{*}\sigma_1^{*}-\sigma_0^{*} a_1^{*} \right)
  \\\epsilon\lambda a_0^{*} \frac{\partial A_{1,0}^*}{ \partial \sigma}
-2\epsilon  \left (\frac{3da_{0}^{* 3}}{8\omega_{}}-\sigma_0^* a_0^{*}
\right)  a_0^{}
\frac{\partial A_{2,0}^*}{ \partial \sigma} - 2\epsilon a_0^{* 2} A_{2,0}^* +\mathcal 0\epsilon^2)
\end{multline*}
We use \eqref{eq:a0s0=}  and the lower order term cancels;
\begin{equation*}
\begin{split}
  \frac{\partial f}{\partial \sigma}&=    
-2\epsilon a_0^{*}\left(\frac{9da_0^{* 2}a_1^{*}}{8\omega}-a_0^{*}\sigma_1^{*}-\sigma_0^{*} a_1^{*} \right)
  +\frac{\epsilon\lambda^2 a_0^{*2}}{4 \omega} 
 -2 \epsilon a_0^{* 2} A_{2,0}^* +\mathcal O(\epsilon^2)
\\
&=    
-2\epsilon a_0^{*}\left(3 \sigma_0^* a_1-a_0^{*}\sigma_1^{*}-\sigma_0^{*} a_1^{*} \right)
  +\frac{\epsilon\lambda^2 a_0^{*2}}{4 \omega} 
 - 2\epsilon a_0^{* 2} A_{2,0}^* +\mathcal O(\epsilon^2)
\\
 &=  \epsilon a_0^* \Big[
- 2 \left(2 \sigma_0^* a_1-a_0^{*}\sigma_1^{*} \right)
  +\frac{\lambda^2 a_0}{4 \omega}  -  2a_0 A_{2,0}^* \Big] +\mathcal O(\epsilon^2)
\\
&=  \epsilon a_0^* \Big[ 2a_0^{*}\sigma_1^{*}
- 4 \sigma_0^* a_1
  +\frac{\lambda^2 a_0}{4 \omega}  -  2a_0 A_{2,0}^* \Big]
  +\mathcal O(\epsilon^2).
\end{split}
\end{equation*}

We compute the derivative with respect to $\beta$;
\begin{equation}
 \frac{\partial f}{\partial \beta}= 2\epsilon \Big[ (-\frac{\lambda a}{2}+\epsilon
 A_{1})\frac{\partial A_1}{\partial \beta}-
\left (\frac{3da^{3}}{8\omega_{}}-\sigma a - \epsilon a A_{2}
\right ) a \frac{\partial A_2}{\partial \beta}  \Big] +\mathcal O(\epsilon^2)
\end{equation}
and for $a_0^*, \;\beta_0^*$
\begin{equation}
   \frac{\partial f}{\partial \beta}=
%-2 \epsilon  \frac{\lambda  a_0^*}{2}\frac{\partial A_1^*}{\partial  \beta} +O(\epsilon^2)=
- \epsilon  \lambda
     a_0^*\frac{\partial A_1^*}{\partial \beta} +\mathcal O(\epsilon^2)
\end{equation}
the partial derivatives of $A_1, \; A_2$ are computed at $a=a_0^*, \;
\beta=\beta_0^*$, we get:

\begin{equation}
   \frac{\partial f}{\partial \beta}= -\epsilon  \frac{\lambda^2
     a_0^{*}F_m}{8 \omega^2}=-\epsilon  \frac{\lambda^3
     a_0^{*2}}{8 \omega}
\end{equation}
and
\begin{equation}
\label{eq:dfdgdgds}
   \frac{\partial f}{\partial \beta}\frac{\partial \beta}{\partial 
     \sigma}= \epsilon  \frac{\lambda^2
     a_0^{*2}}{4 \omega}.
\end{equation}
We use \eqref{eq:dfds=}, \eqref{eq:dfdgdgds}  in \eqref{eq:dfs+dfg=0};
this last equation  defines implicitly $\sigma^*$ as a function
of $\epsilon$; we use the expansions \eqref{eq:ags_exp},   and we
get
\begin{align*}
  \frac{\partial f}{\partial \sigma}+ \frac{\partial f}{\partial
    \beta}\frac{\partial\beta}{\partial \sigma}&=\epsilon a_0^* \Big[ 2a_0^{*}\sigma_1^{*}
- 4 \sigma_0^* a_1
  +\frac{\lambda^2 a_0}{4 \omega}  -  2a_0 A_{2,0}^* \Big]
+\epsilon  \frac{\lambda^2    a_0^{*2}}{4 \omega}  +\mathcal O(\epsilon^2) \\
&=\epsilon a_0^* \Big[ 2a_0^{*}\sigma_1^{*}
- 4 \sigma_0^* a_1
  +\frac{\lambda^2 a_0}{2 \omega}  -  2a_0 A_{2,0}^* \Big]
  +\mathcal O(\epsilon^2) \\
&=2\epsilon a_0^* \Big[ a_0^{*}\sigma_1^{*}
- 2 \sigma_0^* a_1
  +\frac{\lambda^2 a_0}{4 \omega}  -  a_0 A_{2,0}^* \Big]
  +\mathcal O(\epsilon^2) \\
&=2\epsilon a_0^{*2} \Big[ \sigma_1^{*}
- 2 \frac{\sigma_0^* a_1}{a_0}
  +\frac{\lambda^2 }{4 \omega}  -  A_{2,0}^* \Big]
  +\mathcal O(\epsilon^2).
\end{align*}
So we obtain
\begin{align}
  \sigma_1^{*}&=
 2 \frac{\sigma_0^* a_1}{a_0}
  -\frac{\lambda^2 }{4 \omega}  +  A_{2,0}^*\\
&= 2 \frac{\sigma_0^* a_1}{a_0}
  -\frac{\lambda^2 }{4 \omega}  - \frac{5 \sigma_0^2}{12
    \omega}-\frac{5c^2a_0^{*2}}{12 \omega^3}
\end{align}

and we obtain with \eqref{eq:a1=}:
\begin{align}
  \sigma_1^*&= 2 \frac{\sigma_0^* }{a_0} \left(\frac{-a_0^*\sigma_0^*}{\omega}\right) -\frac{\lambda^2 }{4 \omega} -\frac{5\sigma_0^{,*2} }{12\omega}-\frac{5c^2a_0^{*2}}{12 \omega^3}\\
&= - \frac{29 \sigma_0^2}{12\omega} -\frac{\lambda^2 }{4 \omega} -\frac{5c^2a_0^{*2}}{12 \omega^3}
\end{align}

\bibliography{/home/br/bin/latex/biblio,/home/br/bin/latex/these-hamad,/home/br/bin/latex/biblio-vnl}

\def\cprime{$'$}

%%\end{article}
\end{document}